\documentclass[a4paper]{amsart}
\usepackage{amsmath,amsthm,amssymb}
\usepackage{dsfont}
\usepackage{mathrsfs} 
\usepackage[all,cmtip]{xy}
\input{diagxy}
\usepackage{url}
\usepackage{bm}

\usepackage[top=.5in, bottom=.5in, left=.7in, right=.7in]{geometry}

 \theoremstyle{plain}
 \newtheorem{thm}{Theorem}[section]
 \newtheorem{cor}[thm]{Corollary}
 \newtheorem{lem}[thm]{Lemma}
 
 \newtheorem{prop}[thm]{Proposition}

\theoremstyle{definition}
 \newtheorem{defn}[thm]{Definition}
 
 \newtheorem{hyp}[thm]{Hypothesis}

\theoremstyle{remark}
 \newtheorem{rem}[thm]{Remark}

 \newtheorem{nota}[thm]{Notation}
 \newtheorem{conv}[thm]{Convention}

 \numberwithin{equation}{section}

\DeclareMathOperator{\VF}{VF}
\DeclareMathOperator{\ACVF}{ACVF}
\DeclareMathOperator{\RV}{RV}

\DeclareMathOperator{\MM}{\mathcal{M}}

\DeclareMathOperator{\OO}{\mathcal{O}}

 \DeclareMathOperator{\ran}{ran}
 \DeclareMathOperator{\dom}{dom}

 \DeclareMathOperator{\id}{id}

 \DeclareMathOperator{\lh}{lh}

 \DeclareMathOperator{\aut}{Aut}

 \DeclareMathOperator{\td}{tr\,deg}

 \DeclareMathOperator{\cha}{char}
 \DeclareMathOperator{\ac}{\overline{ac}}

 \DeclareMathOperator{\acl}{acl}
 \DeclareMathOperator{\dcl}{dcl}
 \DeclareMathOperator{\pr}{pr}
 \DeclareMathOperator{\alg}{ac}

 \DeclareMathOperator{\mgl}{GL}

\DeclareMathOperator{\jcb}{Jcb}

\DeclareMathOperator{\K}{\overline{K}}

\DeclareMathOperator{\tbk}{tbk}

\def\Xint#1{\mathchoice
{\XXint\displaystyle\textstyle{#1}}%
{\XXint\textstyle\scriptstyle{#1}}%
{\XXint\scriptstyle\scriptscriptstyle{#1}}%
{\XXint\scriptscriptstyle\scriptscriptstyle{#1}}%
\!\int}
\def\XXint#1#2#3{{\setbox0=\hbox{$#1{#2#3}{\int}$}
\vcenter{\hbox{$#2#3$}}\kern-.5\wd0}}

\newcommand{\Z}{\mathds{Z}}

\newcommand{\F}{\mathds{F}}

\newcommand{\Q}{\mathds{Q}}
\newcommand{\N}{\mathds{N}}

\newcommand{\R}{\mathds{R}}

\newcommand{\p}{$p$\nobreakdash}
\newcommand{\omin}{$o$\nobreakdash}

\newcommand{\gC}{\mathfrak{C}}

\newcommand{\gL}{\mathfrak{L}}

\newcommand{\gO}{\mathfrak{O}}

\newcommand{\gc}{\mathfrak{c}}

\newcommand{\go}{\mathfrak{o}}
\newcommand{\gp}{\mathfrak{p}}
\newcommand{\gq}{\mathfrak{q}}
\newcommand{\gr}{\mathfrak{r}}

\newcommand{\0}{\emptyset}

\DeclareMathAlphabet{\mathpzc}{OT1}{pzc}{m}{it}

 \newcommand{\abs}[1]{\left\vert#1\right\vert}
 
 \newcommand{\set}[1]{\left\{#1\right\}}
 
 \newcommand{\norm}[1]{\left\Vert#1\right\Vert}
 
 \newcommand{\wt}[1]{\widetilde{#1}}
 \newcommand{\wh}[1]{\widehat{#1}}
 \newcommand{\lan}[1]{\mathcal{L}_{\textup{#1}}}

 \newcommand{\dlbr}{( \! (}
 \newcommand{\drbr}{) \! )}

\newcommand{\mdl}[1]{\mathcal{#1}}  
\newcommand{\bb}[1]{\mathbb{#1}}

\newcommand{\ex}[1]{\exists #1 \;} 
\newcommand{\fa}[1]{\forall #1 \;} 

\newcommand{\rest}{\upharpoonright}
\newcommand{\lbar}{\vec}

\newcommand{\fun}{\longrightarrow}
\newcommand{\efun}{\longmapsto}
\newcommand{\sub}{\subseteq}

\newcommand{\mi}{\smallsetminus}

\newcommand{\colim}[1]{\underset{#1}{\text{colim} \,}}  

\DeclareMathOperator{\mVF}{\mu \! \VF}
\DeclareMathOperator{\mgVF}{\mu_{\Gamma} \! \VF}
\DeclareMathOperator{\mRV}{\mu \! \RV}
\DeclareMathOperator{\mgRV}{\mu_{\Gamma} \! \RV}

\DeclareMathOperator{\mG}{\mu \Gamma}

\DeclareMathOperator{\mRES}{\mu RES}

\DeclareMathOperator{\rv}{rv}
\DeclareMathOperator{\sn}{sn}

\DeclareMathOperator{\csn}{csn}
\DeclareMathOperator{\rcsn}{\overline {csn}}
\DeclareMathOperator{\vv}{val}

\DeclareMathOperator{\RES}{RES}

\DeclareMathOperator{\gsk}{\mathbf{K}_+}
\DeclareMathOperator{\ggk}{\mathbf{K}}
\DeclareMathOperator{\ob}{Ob}
\DeclareMathOperator{\fn}{FN}
\DeclareMathOperator{\fib}{fib}

\DeclareMathOperator{\vol}{vol}
\DeclareMathOperator{\isp}{I_{sp}}
\DeclareMathOperator{\misp}{\mu I_{sp}}

\DeclareMathOperator{\vrv}{vrv}

\DeclareMathOperator{\RVH}{RVH}

\DeclareMathOperator{\can}{\mathbf{c}}

\DeclareMathOperator{\pvf}{pvf}
\DeclareMathOperator{\prv}{prv}

\DeclareMathOperator{\der}{d}

\DeclareMathOperator{\KRC}{\mathbf{K} \R}

\DeclareMathOperator{\vtp}{\gsk VTP}
\DeclareMathOperator{\mvtp}{\gsk \mu \! VTP}
\DeclareMathOperator{\mgvtp}{\gsk \mu_{\Gamma} \! VTP}
\DeclareMathOperator{\gvtp}{\ggk VTP}

\DeclareMathOperator{\res}{res}

\begin{document}

\title[Integration in $\ACVF$ with sections]{Integration in algebraically closed valued fields with sections}

\author[Y. Yin]{Yimu Yin}

\thanks{I would like to thank Udi Hrushovski and Fran\c{c}ois Loeser for their guidance. I would also like to thank the anonymous referee whose thorough reports have led to vast improvements of the paper. The research reported in this paper has been partially supported by the ERC Advanced Grant NMNAG}


\address{Institut Math\'{e}matique de Jussieu, 4 place Jussieu, 75252 Paris Cedex 05, France }

\email{yin@math.jussieu.fr}

\begin{abstract}
We construct Hrushovski-Kazhdan style motivic integration in certain expansions of $\ACVF$. Such an expansion is typically obtained by adding a full section or a cross-section from the $\RV$-sort into the $\VF$-sort and some (arbitrary) extra structure in the $\RV$-sort. The construction of integration, that is, the inverse of the lifting map $\bb L$, is rather straightforward. What is a bit surprising is that the kernel of $\bb L$ is still generated by one element, exactly as in the case of integration in $\ACVF$. The overall construction is more or less parallel to the main construction of~\cite{hrushovski:kazhdan:integration:vf}, as presented in~\cite{Yin:special:trans, Yin:int:acvf}. As an application, we show uniform rationality of Igusa zeta functions for non-archimedean local fields with unbounded ramification degrees.
\end{abstract}

\maketitle

\tableofcontents

\section{Introduction}

We have presented the main construction of the Hrushovski-Kazhdan integration theory~\cite{hrushovski:kazhdan:integration:vf} in~\cite{Yin:special:trans, Yin:int:acvf}. The integration constructed there is ``unrefined'' in the sense that, although the kernel of the lifting map $\bb L$, that is, the congruence relation $\isp$, is surprisingly simple, being generated by a single element, and the whole theory is structurally sound, satisfying, among other things, a Fubini-type theorem and a change of variables formula, computation of most integrals appear to be too complicated or utterly intractable. This is so even without volume forms and when only simple geometrical objects are involved, such as an open ball with one closed hole and a closed ball with two open holes, computing the standard contractions of which, according to~\cite[Proposition~6.18]{Yin:special:trans}, would tell us whether there is a definable bijection of the two in $\ACVF$. Refinement may proceed in several directions, for example, see~\cite[\S 10]{hrushovski:kazhdan:integration:vf} and~\cite{hru:kazh:val:ring:2006}, all of which involve manipulations of the Grothendieck (semi)rings that provide values for motivic integrals, such as groupifying, coarsening (usually by way of introducing external algebraic structures), and decomposing into tensor product. This last manipulation makes computation of certain integrals much more transparent, especially when integrating functions with one variable, such as the one mentioned above.

In this paper we shall first construct ``unrefined'' motivic integration maps in certain expansions of algebraically closed valued fields and then refine the target semirings of these maps by decomposing them into tensor products in a canonical way. Such an expansion of algebraically closed valued fields is typically obtained in two independent steps: adding a full section (an $\RV$-section) or a cross-section from the $\RV$-sort into the $\VF$-sort and then adding arbitrary relations and functions in the $\RV$-sort. Expansions with extra structure in the $\RV$-sort has been considered in~\cite[\S 12]{hrushovski:kazhdan:integration:vf}, where a homomorphism between Grothendieck semirings is obtained more or less along the line of the main construction, in particular, the congruence relation $\isp$ retains the same degree of simplicity. Expansions with a section from the residue field into the valued field (a $\K$-section) has been considered in~\cite{hru:kazh:2009}. This is in the context of adelic structures over curves, where an integration in the style of \cite{hrushovski:kazhdan:integration:vf} is not needed and hence is not developed.

Our motivation for extending the Hrushovski-Kazhdan theory to such expansions is twofold. Firstly, this is to prepare the ground for a plausible theory of motivic characters, especially multiplicative ones, which is something we should have if we are to further the (already far-reaching) application of the theory of motivic integration to, say, geometry and representation theory, as demonstrated, for example, in~\cite{CCGS:computable:char, CGH:inv.dis:unram, cluckers:hales:loeser:transfer, hru:kazh:2009, hru:loe:lef}. The use of characters in constructing representations in function spaces is beautifully expounded in the (perhaps a bit old-fashioned but still tremendously insightful) work~\cite{ggps:1990}. Secondly, motivic integration in real closed fields is alluded to in the introduction of \cite{hrushovski:kazhdan:integration:vf} as a hope. We shall realize this hope in a future paper \cite{Yin:int:tcvf}. The framework for doing so calls for a cross-section and its technical aspects closely resemble those of this paper.

The construction in this paper is entirely modeled on and heavily relies on the (auxiliary) results of the construction presented in~\cite{Yin:special:trans, Yin:int:acvf}. In particular, we still adhere to the three-step procedure as laid out in the introduction of \cite{Yin:int:acvf}. For clarity, let us repeat it once again. Let $\bb T$ be an expansion of $\ACVF$, which includes an $\RV$-section $\sn : \RV \fun \VF$ or a cross-section $\csn : \Gamma \fun \VF$ or both. Let $\VF_*$ and $\RV[*]$ be two suitable categories of definable sets that are respectively associated with the $\VF$-sort and the $\RV$-sort. To construct a canonical
homomorphism from the Grothendieck semigroup $\gsk \VF_*$ to the Grothendieck semigroup $\gsk \RV[*] / \isp$, where $\isp$ is a suitable semigroup congruence relation, we proceed as follows:
\begin{itemize}
 \item\emph{Step~1.} There is a natural lifting map $\bb L$ from the set of objects of $\RV[*]$ into the set of objects of $\VF_*$. We show that $\bb L$ hits every isomorphism class of $\VF_*$.

 \item\emph{Step~2.} We show that $\bb L$ induces a semigroup homomorphism from $\gsk \RV[*]$ into $\gsk \VF_*$, which is also denoted by $\bb L$.

 \item\emph{Step~3.} In order to obtain a precise description of the semigroup congruence relation on $\gsk \RV[*]$ induced by $\bb L$, that is, the kernel of $\bb L$, we introduce two operations: special bijection in the $\VF$-sort and blowup in the $\RV$-sort. In a sense these two operations mirror each other. Using this correlation we show that, for any objects $\mathbf{U}_1$, $\mathbf{U}_2$ in $\RV[*]$, there are isomorphic blowups $\mathbf{U}_1^{\sharp}$, $\mathbf{U}_2^{\sharp}$ of $\mathbf{U}_1$, $\mathbf{U}_2$ if and only if $\bb L(\mathbf{U}_1)$, $\bb L(\mathbf{U}_2)$ are isomorphic.
\end{itemize}
Through certain standard algebraic manipulations, the inverse of $\bb L$ gives rise to various ring homomorphisms and module homomorphisms. These are understood as generalized Euler characteristic or, if volume forms are present, integration. Note that, in principle, the construction is already completed in Step~2 (See \S\ref{section:dim}). However, to facilitate computation in future applications, it seems much more satisfying to have a precise description of the semigroup congruence relation as obtained in Step~3 (See \S\ref{section:ker}). Perhaps a bit surprisingly, this kernel of $\bb L$ is still generated by one element, exactly as in the case of integration in $\ACVF$.

There is really just one new (nontechnical) idea in this paper, which is very straightforward. For every $\bb T$-definable set $A$ we seek a definable function $\pi : A \fun \RV^m$ such that each fiber $\pi^{-1}(\lbar t)$ is $\sn(\lbar t)$-definable in $\ACVF$, similarly if the $\RV$-section $\sn$ is replaced by the cross-section $\csn$ (we have to work with $\csn$ instead of $\sn$ in the situation with volume forms). Such a function is called an $\RV$- or a $\Gamma$-partition of $A$. If it exists then we may assign a volume to $A$ by first computing the volumes of the fibers, using the results for $\ACVF$, and then sum them up more or less formally. In fact such a partition always exists for a definable set. Conceptually, the few foregoing sentences capture the gist of this paper so well that it is actually tempting to end the discussion right here. But that is probably not very convincing for someone who is not already familiar with the intricate working of the Hrushovski-Kazhdan theory, especially when highly nontrivial modifications of certain technical results are called for. So, we opt for spelling out more details in a few pages. Inevitably, the writing will repeat (variations of) some things that have already been said in~\cite{Yin:special:trans, Yin:int:acvf}.

In~\cite{Yin:QE:ACVF:min} we have compared expansions with $\RV$-section and expansions with $\K$-section in terms of minimality conditions. It is not hard to see that our method here also works for expansions of $\ACVF$ with $\K$-section.

We now describe an application to local zeta functions. Let $f(\lbar X) \in \Q_p[X_1, \ldots, X_n]$, $\kappa$ be a positive real number, and $L$ be a finite extension of $\Q_p$. The norm of $a \in L$ is denoted by $|a|_L$ and the Haar measure on $L$ is denoted by $|\der \lbar X|_L$. Suppose that $A \sub L^n$ is bounded and is $\Q_p$-definable in the language with a cross-section. Note that here the parameters used to define $f$ and $A$ are allowed to vary in a suitable way as $p$ and $L$ vary, for example, the ramification degree of $L$ may be a defining parameter for $A$. Consider the Igusa local zeta function
\[
\zeta(A, L, \kappa) = \int_A |f(\lbar X)|_L^{\kappa} |\der \lbar X|_L.
\]
Following the specialization procedure in \cite{hrushovski:kazhdan:integration:vf}, we can show that $\zeta(A, L, \kappa)$ is uniformly rational for all \p-adic fields (see Definition~\ref{def:uni} for the precise meaning of uniformity). This can also be derived using the Denef-Pas method in \cite{Denef1984, denef:85, Pa89, Pas:1990}.

The paper is organized as follows. In \S\ref{section:prelim} we first introduce the class of expansions of $\ACVF$ that shall be considered. Obviously not much can be done without quantifier elimination, which is derived immediately. Other basic structural properties are also collected in this section, which shall be used throughout the rest of the paper. In \S\ref{section:G:RV} categories associated with the $\RV$-sort are introduced and their Grothendieck semigroups are studied. Here the reader should notice that, by having a cross-section, the various target semirings of the Grothendieck homomorphisms actually become simpler than those in \cite{hrushovski:kazhdan:integration:vf}. The main result of this section is the expression of these semirings as certain tensor products. This essentially repeats some of the work in \cite[\S9-10]{hrushovski:kazhdan:integration:vf}. However, as in~\cite{Yin:special:trans, Yin:int:acvf}, we give much simpler and more direct proofs. In \S\ref{section:dim} we begin with an investigation of dimension in the $\VF$-sort and other related notions, such as the Jacobian. Then the categories associated with the $\VF$-sort are introduced. This is parallel to the corresponding discussion in~\cite{Yin:special:trans} and the modifications are all very natural for the current setting. The first two steps of the three-step procedure described above are completed in \S\ref{section:dim}. In order to obtain a precise description of the kernel of the lifting map $\bb L$, we need an analog of~\cite[Theorem~5.4]{Yin:special:trans}, which guarantees, after modification using only special bijections, contractibility of an arbitrary function. This is also done in \S\ref{section:dim}, which is the most technical part of the construction and is needed for the application to local zeta functions. In \S\ref{section:ker} we study blowups in the $\RV$-categories and then describe the kernel of $\bb L$. Subsequently various Grothendieck homomorphisms are constructed. These follow very closely the corresponding discussion in~\cite{Yin:int:acvf}.  In the last section, we specialize some of the results to non-archimedean local fields, which is more or less automatic by compactness, and derive the uniform rationality of local zeta functions described above.

\section{Preliminaries and some basic structural properties}\label{section:prelim}

The reader is referred to~\cite{Yin:QE:ACVF:min, Yin:int:acvf, Yin:special:trans} for notation and terminology. For example, the various notational conventions concerning coordinate projection maps in~\cite[Notation~2.10]{Yin:special:trans} shall be used frequently:

\begin{nota}\label{indexing}
Let $A \sub \VF^n \times \RV^m$. For any $n \in \N$, let $I_n
= \set{1, \ldots, n}$. Let $I = I_n \uplus I_m$, $E \sub I$, and $\tilde{E} = I \mi E$. If $E$ is a
singleton $\set{i}$ then we always write $E$ as $i$ and
$\tilde{E}$ as $\tilde{i}$. We write $\pr_E(A)$ for the projection
of $A$ to the coordinates in $E$. For any $\lbar a \in
\pr_{\tilde{E}} (A)$, the fiber $\{\lbar b : (\lbar b, \lbar a)
\in A \}$ is denoted by $\fib(A, \lbar a)$. Note that we shall often tacitly identify the two subsets $\fib(A, \lbar a)$ and $\fib(A, \lbar a) \times \set{\lbar
a}$. Also, it is often more convenient to use simple descriptions
as subscripts. For example, if $E = \set{1, \ldots, k}$ etc.\ then
we may write $\pr_{\leq k}$ etc. If $E$ contains exactly the
$\VF$-indices (respectively $\RV$-indices) then $\pr_E$ is written
as $\pvf$ (respectively $\prv$). If $E'$ is a subset of the coordinates of $\pr_E (A)$ then the composition $\pr_{E'} \circ \pr_E$ is written as $\pr_{E, E'}$. Naturally
$\pr_{E'} \circ \pvf$ and $\pr_{E'} \circ \prv$ are written as $\pvf_{E'}$ and $\prv_{E'}$, respectively.
\end{nota}


We shall work with certain expansions of the $\lan{RV}$-theory $\ACVF$ (see~\cite[Definitions~2.1, 2.2]{Yin:special:trans}). Recall that the $\RV$-sort contains an element $\infty = \rv(0)$. It also serves as the element $0$ in the residue field $\K$. For psychological reasons, we shall write it as $0$ when $\K$ is concerned (also see Convention~\ref{conv:imag}).

The expansions of $\ACVF$ that we shall consider are obtained in two steps: we first add a section of $\RV$ and a cross-section of $\Gamma$ (see below), and then arbitrary relations and functions in the $\RV$-sort.

\begin{defn}\label{defn:sn}
A function $\sn : \RV \fun \VF$ is a \emph{section} of $\RV$ if
\begin{enumerate}
  \item $\sn \rest \RV^{\times}$ is a homomorphism of multiplicative groups and $\sn(\infty) = 0$,
  \item $\sn(t) \in t$ for every $t \in \RV$,
  \item $\sn(\K^{\times}) \cup \set{0}$ is a subfield of $\OO$.
\end{enumerate}
Similarly, $\sn$ is a \emph{section} of $\K$ if it is the restriction of a section of $\RV$ to $\K^{\times}$ augmented by $\sn(0) = 0$.
\end{defn}

\begin{rem}\label{rem:loc}
Any non-archimedean local field of positive characteristic carries a natural section. However, a non-archimedean local field of characteristic $0$ is only equipped with a natural \emph{weak section}, that is, a function $\sn : \RV \fun \VF$ that only satisfies the first two conditions in Definition~\ref{defn:sn}, which is given by the nonzero Teichm\"{u}ller representatives and a choice of uniformizer.
\end{rem}

\begin{defn}
A \emph{cross-section} of $\Gamma$ is a group homomorphism $\csn : \Gamma \fun \VF^{\times}$ such that $\vv \circ \csn = \id$. The \emph{reduced cross-section} of $\Gamma$ is the function $\rcsn = \rv \circ \csn : \Gamma \fun \RV^{\times}$. Set $\csn(\infty) = 0$ and $\rcsn(\infty) = \infty$. The \emph{twistback} function $\tbk : \RV \fun \K$ is given by $u \efun u / \rcsn(\vrv(u))$, where we set $\infty / \infty = 0$.
\end{defn}

The expansions of $\lan{RV}$ with the function symbols $\sn$, $\csn$ are respectively denoted by $\lan{RV}^{1}$, $\lan{RV}^{2}$. The expansion of $\lan{RV}^{1}$ with the function symbol $\rcsn$ is denoted by $\lan{RV}^{3}$. The theories $\ACVF^{1}$ in $\lan{RV}^{1}$, $\ACVF^{2}$ in $\lan{RV}^{2}$, and $\ACVF^{3}$ in $\lan{RV}^{3}$ state that, in addition to the axioms of $\ACVF$, $\sn$ is a section of $\RV$, $\csn$ is a cross-section of $\Gamma$, and $\rcsn$ is a reduced cross-section of $\Gamma$. If the characteristics are specified then we write $\ACVF^{1}(0, p)$ etc.

\begin{conv}\label{conv:imag}
Let $\textup{res} : \RV \fun \K$ be the function given by $\textup{res} \rest \K^{\times} = \id$ and $\textup{res}(t) = 0$ for all $t \notin \K^{\times}$. Technically speaking, $+ : \K^2 \fun \K$ is a function symbol only in the imaginary sort $\K$, which, as in \cite{hrushovski:kazhdan:integration:vf, Yin:QE:ACVF:min, Yin:special:trans, Yin:int:acvf}, is subsumed into the $\RV$-sort. Terms that appear potentially ill-formed should be interpreted accordingly. For example, in the term $\sn(\tau + \tau')$, the symbol $\sn$ should be understood as a section of $\K$ and $\tau$, $\tau'$ should be replaced by $\textup{res}(\tau)$, $\textup{res}(\tau')$.
\end{conv}

\begin{thm}[]\label{thm:dag:qe}
The theories $\ACVF$, $\ACVF^{1}$, $\ACVF^{2}$, and $\ACVF^{3}$ all admit quantifier elimination. Consequently, if the characteristics are specified then these theories are complete.
\end{thm}
\begin{proof}
For $\ACVF$ and $\ACVF^{1}$ quantifier elimination is proved in \cite[Theorem~3.10, Theorem~3.14]{Yin:QE:ACVF:min}. It is easy to adapt the proof there for $\ACVF^{2}$ and $\ACVF^{3}$ (also see Proposition~\ref{T:qe} and Remark~\ref{rem:T:qe} below). Completeness is clear by inspecting the quantifier-free sentences.
\end{proof}

Let $\mathbb{T}$ be an expansion of $\ACVF^{1}$ in a language $\mdl L_{\bb T}$. We assume that the language $\mdl L_{\bb T}$ contains additional relation and function symbols only in the $\RV$-sort, for example, a cross-section, a Denef-Pas angular component map, or a subfield of the residue field. After Proposition~\ref{T:qe} below we shall work exclusively with such expansions of $\ACVF^{3}$. But, before that, there is no need to require the presence of a cross-section. Note that if $\bb T$ does expand $\ACVF^{3}$ then it makes sense to speak of the $\mdl L_{\wt{\bb T}}$-reduct $\wt{\bb T}$ of $\bb T$, where $\mdl L_{\wt{\bb T}}$ is the language obtained from $\mdl L_{\bb T}$ by replacing the functions $\sn$ and $\rcsn$ with the function $\csn$. Also note that since, for example, $\ACVF^{1}(0, p)$ is complete, every model of it embeds into a sufficiently saturated model of $\bb T(0, p)$. By adding more primitives, without changing the class of definable sets, we also assume that the reduct of $\bb T$ to the $\RV$-sort eliminates quantifiers.

Let $M_{\bb T} \models \bb T$ and $A \sub M_{\bb T}$. Let $M_1$, $M$ be the $\lan{RV}^{1}$-, $\lan{RV}$-reducts of $M_{\bb T}$, respectively. We shall write $\dcl^{\bb T}(A)$, $\acl^{\bb T}(A)$ for the definable and the (model-theoretic) algebraic closures of $A$ in $M_{\bb T}$, $\dcl^{1}(A)$, $\acl^{1}(A)$ for those of $A$ in $M_1$, etc. Note that, in general, $\dcl(A)$ is not closed under $\sn$ and hence cannot be expanded to a substructure of $M_1$ without changing the underlying set. If $\dcl(A)$ is closed under $\sn$ then it may be identified with $\dcl^{1}(A)$. In this case we shall write $\dcl(A) = \dcl^{1}(A)$. The same convention applies when the operators $\dcl^{\bb T}$, $\acl^{1}$, etc.\ are involved. For example, we have
\[
\acl(\sn(\RV(M))) = \acl^{1}(\sn(\RV(M))) = \acl^{1}(\RV(M)).
\]

The proof of \cite[Theorem~3.14]{Yin:QE:ACVF:min} works more or less for the following proposition. For clarity, we give some details.

\begin{prop}\label{T:qe}
The theories $\bb T$ and $\wt{\bb T}$ eliminate all quantifiers.
\end{prop}
\begin{proof}
We shall only be concerned with $\bb T$, since the argument for $\wt{\bb T}$ is similar. Let $M, N \models \bb T$ such that $N$ is $\norm{M}^+$-saturated. Let $S$ be a substructure of $M$ and $f : S \fun N$ a monomorphism. It is enough to extend $f$ to a monomorphism $M \fun N$.

Since $\bb T$ eliminates quantifiers in the $\RV$-sort, there is a monomorphism $g : \RV(M) \fun \RV(N)$ extending $f \rest \RV(S)$. Since the henselization of $S$ is an immediate extension (in the sense of valuation theory), we may extend $f$ accordingly and hence may just assume that $S$ is henselian.

Let $\dot f : \dot S \fun N$ be the $\lan{RV}^{1}$-reduct of $f$. Let $t \in \K(M)$ be algebraic over $\K(\dot S)$ in the field theoretic sense. We have
\[
[\VF(\dot S)(\sn(t)) : \VF(\dot S)] = [\K(\dot S)(t) : \K(\dot S)]
\]
and hence $\Gamma(\dcl^{1}(\dot  S \cup {\sn(t)})) = \Gamma(\dot  S)$. Since the fields $\K(\dot S)(t)$ and $\K(f(\dot S))(g(t))$ are isomorphic via $g$, we may extend $\dot f$ to an $\lan{RV}^{1}$-monomorphism $\dot f_t : \dcl^{1}(\dot  S \cup {\sn(t)}) \fun N$ by $\sn(t) \efun \sn(g(t))$. Clearly $\dot f_t$ is compatible with $g$. Repeating this procedure, we may assume that $\K(\dot S)$ is algebraically closed. Next, let $t \in \RV(M) \mi \K(M)$ such that $t^n \in \RV(\dot S)$ for some $n > 0$ and $n$ is minimal with respect to this condition. We have
\[
[\Gamma(\dcl^{1}(\dot  S \cup {\sn(t)})) : \Gamma(\dot  S)] = n, \quad \K(\dcl^{1}(\dot  S \cup {\sn(t)})) = \K(\dot  S).
\]
Since $\sn(t)^n = \sn(t^n)$ and $\sn(g(t))^n = \dot f(\sn(t^n))$, as above, by setting $\sn(t) \efun \sn(g(t))$, we obtain an extension of $\dot f$ that is compatible with $g$. So we may assume that $\Gamma(\dot S)$ is divisible. Now, $\acl^{1}(\dot S)$ is a model of $\ACVF^{1}$ and $\RV(\acl^{1}(\dot S)) = \RV(\dot S)$, by Theorem~\ref{thm:dag:qe}, we may assume $\dot S = \acl^{1}(\dot S)$.

For any $t \in \RV(M) \mi \dot S$, the proof of~\cite[Lemma~3.13]{Yin:QE:ACVF:min} goes through with the choice $\sn(t) \efun \sn(g(t))$, which yields an extension of $\dot f$ that is compatible with $g$. Repeating the whole process thus far, we eventually obtain an extension $\dot f_1$ of $\dot f$ that includes the $\lan{RV}^{1}$-reduct of $g$. Hence $f_1 = \dot f_1 \cup g$ is an $\mdl L_{\bb T}$-monomorphism. At this point, any $\lan{RV}^{1}$-extension of $\dot f_1$ induces an obvious $\mdl L_{\bb T}$-extension of $f_1$, so we are done by Theorem~\ref{thm:dag:qe}.
\end{proof}

\begin{rem}\label{rem:T:qe}
If $\bb T$ is the expansion $\ACVF^{3}$ of $\ACVF^1$ or the one with a reduced angular component map $\ac : \RV^{\times} \fun \K^{\times}$ then $\bb T$ eliminates quantifiers in the $\RV$-sort and hence Proposition~\ref{T:qe} holds for $\bb T$. The proofs are routine and are left to the reader. Since there is a section of $\RV$, we can always define an angular component map from a cross-section and vice versa.

Quantifier elimination still holds if $\bb T$ is an expansion of $\ACVF$ (in the $\RV$-sort only). This follows from a simpler version of the above proof,  or from standard syntactical manipulations that reduces it to the case of $\ACVF$.
\end{rem}

From now on we assume that $\bb T$ expands $\ACVF^{3}$. We fix a sufficiently saturated model $\gC^{\bb T} \models \bb T$ of pure characteristic $0$. The (imaginary) sort of value group is denoted by $\Gamma$. The $\lan{RV}$-reduct (resp.\ $\lan{RV}^{1}$-reduct, etc.) of $\gC^{\bb T}$ is denoted by $\gC$ (resp.\ $\gC^1$, etc.).

\begin{conv}\label{conv:s}
Except in the last section, for convenience and without loss of generality, by a substructure we shall always mean a substructure that is equal to its definable closure. Let $S$ be a small substructure of $\gC^{\bb T}$. Note that any reduct of $S$ is $\VF$-generated. For simplicity, all the reducts of $S$ shall simply be denoted by $S$ if there is no danger of confusion. The corresponding expanded languages (with constants in $S$) are still referred to as $\lan{RV}$, $\lan{RV}^{1}$, etc. Parameters from $S$ are allowed and they will not be specified unless it is necessary. So in effect we shall be working with the complete theories $\ACVF(S)$, $\ACVF^{1}(S)$, etc and by an $\lan{RV}$-definable (resp.\ $\lan{RV}^{1}$-definable, etc.) subset we mean an $S$-$\lan{RV}$-definable (resp.\ $S$-$\lan{RV}^{1}$-definable, etc.) subset. In general, by a definable subset we mean an $\mdl L_{\bb T}$-definable subset, unless indicated otherwise in context. Parameters from sources other than $S$ will be specified in context.
\end{conv}

\begin{nota}\label{nota:ac}
If $A \sub \VF$ then the field generated by $A$ over $\VF(S)$ is denoted as usual by $\VF(S)(A)$ and the field-theoretic algebraic closure of $A \cup \VF(S)$ is denoted by $A^{\alg}$.
\end{nota}

\begin{lem}\label{sub:rel}
For any $U \sub \RV$, the $\lan{RV}$-reduct of $\dcl^1(U)$ is $\dcl(\sn(U))$ and hence $\RV(\dcl^1(U))$ is equal to $\RV(\dcl(\sn(U))) = \RV(\dcl(U))$.
\end{lem}
\begin{proof}
Let $M = \acl^1(U) \models \ACVF^{1}(S)$ and $N = \acl(U \cup \sn(U)) \models \ACVF(S)$. It is clear from the proof of \cite[Theorem~3.14]{Yin:QE:ACVF:min} that $\VF(M) = \VF(N) = \sn(U)^{\alg}$. Hence $\sigma \in \aut_{\dcl^1(U)}(M)$ if and only if $\sigma \in \aut_{\dcl(\sn(U))}(N)$. The claim follows.
\end{proof}

Since a $\VF$-sort equality can be equivalently expressed as an $\RV$-sort equality, we may and shall assume that an $\mdl L_{\bb T}$-formula contains no $\VF$-sort equalities at all.

\begin{defn}
Let $M, N \sub \gC$ be substructures and $\sigma : M \fun N$ be an $\lan{RV}$-isomorphism. We say that $\sigma$ is an \emph{immediate isomorphism} if $\sigma(t) = t$ for all $t \in \RV(M)$.
\end{defn}


\begin{lem}\label{imm:ext}
Let $s_1, s_2 : \RV \fun \VF^{\times}$ be two full sections. Then any immediate isomorphism $\sigma : M \fun N$ such that $\sigma(s_1(t)) = s_2(t)$ for all $t \in \RV(M)$ may be extended to an immediate automorphism $\bar \sigma \in \aut_S(\gC)$ such that $\sigma(s_1(t)) = s_2(t)$ for all $t \in \RV$.
\end{lem}
\begin{proof}
With extra bookkeeping, the proof of \cite[Theorem~3.10]{Yin:QE:ACVF:min} works.
\end{proof}

\begin{lem}\label{auto:decom}
Let $U \sub \RV$ and $\sigma$ be an automorphism of $\gC$ over $\dcl(U)$. Then there is an automorphism $\rho$ of $\gC^{1}$ over $\dcl^1(U)$ and an immediate automorphism $\bar \sigma \in \aut_S(\gC)$ such that $\sigma = \bar \sigma \circ \rho$.
\end{lem}
\begin{proof}
First note that $\sigma(\sn(\RV))$ induces a full section $\sn^* : \RV \fun \VF^{\times}$. By Lemma~\ref{sub:rel}, $\RV(\dcl^1(U)) = \RV(\dcl(U))$ and hence the restriction of $\sigma$ to the $\lan{RV}$-reduct of $\dcl^1(U)$ is an immediate automorphism with $\sigma(\sn(t)) = \sn^*(t)$ for all $t \in \RV(\dcl(U))$. By Lemma~\ref{imm:ext}, this restriction of $\sigma$ may be extended to an immediate automorphism $\bar \sigma$ of $\gC$ with $\bar \sigma(\sn(t)) = \sn^*(t)$ for all $t \in \RV$. Now set $\rho = \bar \sigma^{-1} \circ \sigma$.
\end{proof}

\begin{cor}\label{lrvdag:same:lrv}
Let $\lbar t \in \RV$. If $A \sub \RV^m$ is parametrically $\lan{RV}$-definable and is also $\lbar t$-$\lan{RV}^{1}$-definable then it is $\lbar t$-$\lan{RV}$-definable.
\end{cor}
\begin{proof}
We only need to show that any automorphism of $\gC$ over $\dcl(\lbar t)$ fixes $A$ setwise. This is immediate by Lemma~\ref{auto:decom}, since $A$ is trivially invariant under immediate automorphisms.
\end{proof}

\begin{defn}\label{defn:complex}
Let $\tau$ be an $\mdl L_{\bb T}$-term. For any variable $X$, the \emph{$X$-complexity $\abs{\tau}_{X} \in \N$} of $\tau$ is defined inductively as follows.
\begin{enumerate}
  \item If either $X$ does not occur in $\tau$ or $\tau$ is an $\lan{RV}$-term then $\abs{\tau}_{X} = 0$.
  \item If $X$ occurs in $\tau$ and $\tau$ is of the form $\sn(\sigma)$ then $\abs{\tau}_{X} = \abs{\sigma}_{X} + 1$.
  \item If $\tau$ is not of the form $\sn(\sigma)$ then $\abs{\tau}_{X}$ is the maximum of the $X$-complexities of the proper subterms of $\tau$.
\end{enumerate}
The \emph{complexity $\abs{\tau}$} of $\tau$ is the maximum of all $X$-complexities of $\tau$.

Let $\phi(\lbar X, \lbar Y)$ be an $\mdl L_{\bb T}$-formula, where $\lbar X = (X_1, \ldots, X_n)$ are the occurring $\VF$-sort variables and $\lbar Y = (Y_1, \ldots, Y_m)$ are the occurring $\RV$-sort variables. The \emph{$X_i$-complexity $\abs{\phi}_{X_i}$} of $\phi$ is the maximal $X_i$-complexity of the terms occurring in $\phi$; the $Y_i$-complexity $\abs{\phi}_{Y_i}$ of $\phi$ is defined similarly. Let $\abs{\phi}_{\VF}$ be the maximum of the $X_i$-complexities of $\phi$; similarly for $\abs{\phi}_{\RV}$. Lastly set $\abs{\phi} = \max \{ \abs{\phi}_{\VF}, \abs{\phi}_{\RV} \}$.

Let $\phi$ be an $\mdl L_{\bb T}$-term or a quantifier-free $\mdl L_{\bb T}$-formula. If a term $F$ occurs in $\phi$ in the form $\rv(F)$ (respectively $\sn(F)$) then $F$ is said to be an \emph{occurring $\VF$-term} (respectively \emph{occurring $\RV$-term}) of $\phi$. Note that if $F$ is an occurring $\VF$-term of $\phi$ with $| F | = 0$ then it is called an occurring polynomial of $\phi$ in~\cite{Yin:int:acvf, Yin:special:trans}. We shall keep this terminology. Obviously if $\abs{\phi} > 0$ then we have
\[
\abs{\phi} = \max \{| F | : F \text{ is an occurring $\VF$-term of } \phi \} = \max \{| F | : F \text{ is an occurring $\RV$-term of } \phi \} + 1.
\]
If F is an occurring $\VF$-term of $\phi$ that is not a subterm of an occurring $\VF$-term of a higher complexity then $F$ is a \emph{top occurring $\VF$-term} of $\phi$; similarly for a \emph{top occurring $\RV$-term} of $\phi$.
\end{defn}

\begin{lem}\label{lrvdag:lrv:1}
If $A \sub \RV$ is $\lan{RV}^1$-definable then it is $\lan{RV}$-definable.
\end{lem}
\begin{proof}
Let $\phi(Y)$ be a quantifier-free formula that defines $A$, where $Y$ is an $\RV$-sort variables. We do induction on $\abs{\phi}$. Since the base case $\abs{\phi} = 0$ is tautological, we proceed to the inductive step directly.

Let $F_k(Y)$ enumerate the occurring $\VF$-terms of $\phi(Y)$ of complexity $1$. We may write each $F_k(Y)$ in the form $\sum_i a_{i}\sn(Y^i)$, where $a_{i} \in \VF(S)$. For each $t \in A$ and each $k$ let
\[
I_{k, t} = \{i : \vrv(\rv(a_{i}) t^i) \leq \vrv(\rv(a_{j}) t^j) \text{ for all } j\}.
\]
Then set $e_{i} = \sn(\rv(a_{i})) \in \VF(S)$ and $E_{k, t}(Y) = \sum_{i \in I_{k, t}} e_{i}\sn(Y^i)$. Across a disjunction we may assume that, for every $k$ and all $t, s \in A$, $I_{k, t} = I_{k, s}$ and hence $E_{k, t}(Y) = E_{k, s}(Y)$. Then we may write $I_k$ and $E_{k}(Y)$ instead. Note that the equality $E_{k}(Y) = 0$ is equivalent to an $\lan{RV}$-formula. Therefore we may further assume that, for every $k$, either $E_{k}(t) = 0$ for all $t \in A$ or $E_{k}(t) \neq 0$ for all $t \in A$.

If $E_{k}(t) = 0$ for some $k$ and some $t \in A$ then $A$ is finite and hence, by Corollary~\ref{lrvdag:same:lrv}, $A$ is $\lan{RV}$-definable. So we may assume that $E_{k}(t) \neq 0$ for all $k$ and all $t \in A$. Then $\rv(F_k(t)) = \rv(E_{k}(t))$ for all $k$ and all $t \in A$. Since, without loss of generality, $E_{k}(Y)$ is of the form $1 + \sum_{i} e_{i}\sn(Y^i)$, we have $\rv(E_{k}(t)) = 1 + \rv(e_i)t^i$ for all $t \in A$. This means that $\phi(Y)$ is equivalent to a formula of complexity $< \abs{\phi}$ and hence, by the inductive hypothesis, $A$ is $\lan{RV}$-definable.
\end{proof}

\begin{lem}\label{lrvdag:lrv}
If $A \sub \RV^m$ is $\lan{RV}^1$-definable then it is $\lan{RV}$-definable.
\end{lem}
\begin{proof}
We do induction on $m$. The base case $m=1$ is proved above. For the inductive step, by the inductive hypothesis, $\pr_{<m}(A)$ is $\lan{RV}$-definable. On the other hand, for every $\lbar t \in \pr_{<m}(A)$, $\fib(A, \lbar t)$ is both $\dcl^1(\lbar t)$-$\lan{RV}$-definable and $\lbar t$-$\lan{RV}^1$-definable and hence, by Corollary~\ref{lrvdag:same:lrv}, it is $\lbar t$-$\lan{RV}$-definable. For any $\lan{RV}$-formula $\phi(\lbar Y, Z)$, let $B_{\phi} \sub \pr_{<m}(A)$ be the $\lan{RV}^1$-definable subset such that $\lbar t \in B_{\phi}$ if and only if $\phi(\lbar t, Z)$ defines $\fib(A, \lbar t)$. By the inductive hypothesis again, $B_{\phi}$ is $\lan{RV}$-definable. Now the claim follows from compactness.
\end{proof}

\begin{cor}\label{RV:sub:nosn}
Any $\mdl L_{\bb T}$-definable subset $A \sub \RV^m$ may be defined by an $\mdl L_{\bb T}$-formula that does not involve $\sn$, that is, $A$ is definable in the reduct of $\gC^{\bb T}$ to the $\RV$-sort.
\end{cor}
\begin{proof}
Let $\phi$ be a quantifier-free formula that defines $A$. We do induction on $|\phi|$. Let $F_k(\lbar Y)$ enumerate the top occurring $\VF$-terms of $\phi$. We may write each $F_k(\lbar Y)$ in the form $\sum_i a_{i}\sn(\tau_{ki}(\lbar Y))$, where $a_{i} \in \VF(S)$. Let $F^*_k$ be the $\VF$-term obtained from $F_k(\lbar Y)$ by replacing each $\tau_{ki}(\lbar Y)$ with a new variable $X_{ki}$. Let $\phi^*$ be the formula obtained from $\phi$ by replacing each $\rv(F_k(\lbar Y))$ with a new variable $Z_k$. Let $A^*$ be the subset defined by the formula
\[
\phi^* \wedge \bigwedge_{k, i} (Z_k = \rv(F^*_k) \wedge X_{ki} = \tau_{ki}(\lbar Y)).
\]
Since $A = \pr_{\leq m}(A^*)$, the claim follows from the inductive hypothesis and Lemma~\ref{lrvdag:lrv}.
\end{proof}

Therefore, as far as the $\RV$-sort is concerned, $\bb T$ and $\wt{\bb T}$ are the same theory (in the sense that they have the same definable subsets) and there is no need to treat $\wt{\bb T}$ separately. Consequently, if $\gC^{\bb T}$ is a $\Gamma$-minimal expansion of $\gC^3$, that is, if any $\mdl L_{\bb T}$-definable subset $I \sub \Gamma^m$ is $\lan{RV}^{3}$-definable, then we may unambiguously speak of definable subsets in the $\Gamma$-sort:

\begin{defn}\label{def:normal:form}
An \emph{imaginary $\K$-term} is a term of the form $\sum_{i = 1}^k \res(\rv(F_{i}(\lbar X)) \cdot r_{i} \cdot \lbar Y^{\lbar n_i})$, where $\lbar X$ are $\VF$-sort variables, $\lbar Y$ are $\RV$-sort variables, $\lbar n_i \in \N$, $r_i \in \RV$, and $F_{i}(\lbar X)$ is a polynomial with coefficients in $\VF$. An \emph{imaginary $\Gamma$-term} is a term of the same form with $\res$ replaced by $\vrv$.

We should think of these as real terms if we work with the language $\mdl L_{\K \Gamma}$ (resp.\ $\mdl L_{\K \Gamma}^{\rcsn}$) that corresponds to the three-sorted structure of the reduct of $\gC$ (resp.\ $\gC^2$ or $\gC^3$) to the $\RV$-sort. The complexity of an $\mdl L_{\K \Gamma}^{\rcsn}$-formula with respect to $\vrv$ and $\rcsn$ is defined as in Definition~\ref{defn:complex}.
\end{defn}

\begin{lem}\label{gam:same}
Let $\Gamma_{\infty} = \Gamma \cup \{\infty\}$. If $I \sub \Gamma_{\infty}^m$ is $\lan{RV}^{3}$-definable then it is $\lan{RV}$-definable.
\end{lem}
\begin{proof}
By Corollary~\ref{RV:sub:nosn}, we may work in the reduct of $\gC^3$ to the $\RV$-sort and hence with the language $\mdl L_{\K \Gamma}^{\rcsn}$, where we still have quantifier elimination. Let $\phi(\lbar Z)$ be a quantifier-free formula that defines $I$. Consider any term $\tau(\lbar Z)$ that occurs in $\phi(\lbar Z)$ in one of the following ways: $\vrv(\tau(\lbar Z))$, $\textup{res}(\tau(\lbar Z))$, and $\tau(\lbar Z) \, \Box \, 1$ or $\tau(\lbar Z) \, \Box \, \infty$, where $\Box$ is either $=$ or $\neq$ in the $\RV$-sort. Then $\tau(\lbar Z)$ may be written as $t \csn(F(\lbar Z))$, where $t \in \RV(S)$. If $\vrv(\tau(\lbar Z))$ occurs then it may be replaced by $\vrv(t) + F(\lbar Z)$. If $\textup{res}(\tau(\lbar Z))$ occurs then it may be replaced by either $0$ or $t \rcsn(\vrv(t))^{-1}$. If $\tau(\lbar Z) \, \Box \, 1$ occurs then it may be replaced by $\vrv(t) + F(\lbar Z) \, \Box \, 0$ (note that this is so because if $t \neq \rcsn(\vrv(t))$ then $\fa{\lbar Z} \tau(\lbar Z) \neq 1$ is true); similarly for the case $\tau(\lbar Z) \, \Box \, \infty$. In all situations, across a disjunction, the complexity of the formula decreases. So the claim follows from a routine induction on complexity.
\end{proof}

\begin{rem}\label{rem:gam:stab}
Recall that the (imaginary) $\Gamma$-sort is stably embedded in $\gC$; that is, any parametrically $\lan{RV}$-definable subset in the $\Gamma$-sort can be parametrically defined in the reduct of $\gC$ to the $\Gamma$-sort (see the discussion preceding~\cite[Lemma~4.17]{Yin:QE:ACVF:min}). Therefore, all $\lan{RV}$-definable functions in the $\Gamma$-sort are piecewise $\Q$-linear. Here an $R$-linear map for any ring $R$ is allowed to have a constant term, unless indicated otherwise. By Lemma~\ref{gam:same}, this is also true in $\gC^{\bb T}$ if it is a $\Gamma$-minimal expansion of $\gC^3$.

There are two ways of treating an element $\gamma \in \Gamma_{\infty}$: as a point (when we study $\Gamma$ as an independent structure) or a subset of $\gC^{\bb T}$ (when we need to remain in the realm of definable subsets of $\gC^{\bb T}$). The former perspective simplifies the notation but is of course dispensable. We shall write $\vrv^{-1}(\gamma)$ when we want to emphasize that $\gamma \in \Gamma$ is a subset of $\gC^{\bb T}$.
\end{rem}

In fact, Lemma~\ref{gam:same} may be strengthened:

\begin{cor}\label{gam:3:1}
Let $\lbar t \in \RV$ and $I \sub \Gamma_{\infty}^m$ be a $\lbar t$-$\lan{RV}^{3}$-definable subset. Then $I$ is $\lbar t$-$\lan{RV}$-definable.
\end{cor}
\begin{proof}
By stable embeddedness, $I$ is $\Gamma(\dcl^3(\lbar t))$-definable in the reduct of $\gC^3$ (or $\gC$) to the $\Gamma$-sort. On the other hand, it is not hard to see that, by Corollary~\ref{RV:sub:nosn}, $\Gamma(\dcl^3(\lbar t)) = \Gamma(\dcl(\lbar t))$, that is, the subgroup of $\Gamma$ generated by $\vrv(\lbar t)$. So $I$ is also $\lbar t$-$\lan{RV}$-definable.
\end{proof}

\begin{nota}
Given a function $f : A \fun B$, we shall often write $A_b$ for the fiber over $b \in B$ under $f$. In particular, given a definable subset $A$, we shall often write $A_{x}$ for the fiber over $x$ under a function of the form $\rv \rest A$, $\vv \rest A$, $\vrv \rest A$, etc.\  Of course which function is being considered should always be clear in context.
\end{nota}

\begin{defn}
For any subset $U \sub \RV^n$ and $\lbar \gamma \in \Gamma_{\infty}^n$, the subset $\tbk(U_{\lbar \gamma}) \sub \K^n$ is called the \emph{$\lbar \gamma$-twistback} of $U$. The subset $\bigcup_{\lbar \gamma \in \vrv(U)} \{\lbar \gamma\} \times \tbk(U_{\lbar \gamma}) \sub \Gamma_{\infty}^n \times \K^n$ is denoted by $\Omega(U)$. Conversely, $U_{\lbar \gamma}$ is called the \emph{$\lbar \gamma$-twist} of $\tbk(U_{\lbar \gamma})$. If $\tbk(U_{\lbar \gamma}) = \tbk(U_{\lbar \gamma'})$ for all $\lbar \gamma, \lbar \gamma' \in \vrv(U)$ then $U$ is called a \emph{twistoid}, in which case we simply write $\tbk(U)$ for the unique twistback.
\end{defn}

These notions of course depend on the choice of the cross-section $\csn$. Note that for a subset $W \sub \K^n$ and a $\lbar \gamma \in \Gamma_{\infty}^n$, the $\lbar \gamma$-twist $W_{\lbar \gamma}$ of $W$ is defined only if the $0$-coordinates in $W$ match the $\infty$-coordinates in $\lbar \gamma$. For $D \sub \Gamma^{n}$ we write $\Xi (W, D)$ for $\bigcup_{\lbar \gamma \in D} W_{\lbar \gamma}$.

\begin{lem}\label{shift:K}
Let $U \sub \RV^n$ be an $\lan{RV}$-definable subset and $\vrv(U) = D$. Then there is a definable finite partition $D_k$ of $D$ such that each $U_k = U \cap \vrv^{-1}(D_k)$ is a twistoid and the corresponding twistback is $\lan{RV}$-definable.
\end{lem}
\begin{proof}
We work in the reduct of $\gC^3$ to the $\RV$-sort, considered as an $\mdl L_{\K \Gamma}^{\rcsn}$-structure. Let $\phi(\lbar Z, \lbar Y) = \bigvee_i \phi_i(\lbar Z, \lbar Y)$ be a quantifier-free $\mdl L_{\K \Gamma}^{\rcsn}$-formula in disjunctive normal form that defines $\Omega(U) \sub \Gamma_{\infty}^n \times \K^n$. Let $\res(t \csn(F(\lbar Z)))$ be a term that occurs in $\phi$. If $\vrv(t) + F(\lbar Z) \neq 0$ then $\res(t \csn(F(\lbar Z)))$ may be replaced by $0$, otherwise it may be replaced by $t \rcsn(\vrv(t))^{-1}$. Therefore, without loss of generality, each $\phi_i(\lbar Z, \lbar Y)$ may be written as a conjunction of $\theta_i(\lbar Z)$ and $\psi_i(\lbar Y)$, where the variables are displayed. Let $B_i \sub \Gamma_{\infty}^n$ be the subset defined by $\theta_i(\lbar Z)$ and $V_i \sub \K^n$ the subset defined by $\psi_i(\lbar Y)$.

Now we may easily translate each $\psi_i(\lbar Y)$ back into an $\lan{RV}$-formula. On the other hand, by Lemma~\ref{gam:same}, each $B_i$ is also $\lan{RV}$-definable. Let $D_k$ be the $\lan{RV}$-definable finite partition of $D$ induced by the subsets $B_i$. Clearly, for every $\lbar \gamma \in D_k$, $\tbk(U_{\lbar \gamma}) = \bigcup_{\lbar \gamma \in B_i} V_i$.
\end{proof}

\begin{defn}
Let $f : A \fun B$ be a function. If $A$, $B$ only have $\VF$- and $\RV$-coordinates then $f$ is \emph{$\rv$-contractible} if $(\rv \circ f)(\gp \cap A)$ is a singleton for every $\rv$-polydisc $\gp \sub \RVH(A)$ (see~\cite[Definition~2.4, Definition~4.21]{Yin:QE:ACVF:min}). If $A$, $B$ only have $\RV$-coordinates then $f$ is \emph{$\vrv$-contractible} if $(\vrv \circ f)(A_{\lbar \gamma})$ is a singleton for every $\lbar \gamma \in \vrv(A)$. The contractions of $f$, that is, the induced functions $\rv(A) \fun \rv(B)$, $\vrv(A) \fun \vrv(B)$, are usually denoted by $f_{\downarrow}$.
\end{defn}

In context, we shall often drop the prefixes and simply say that $f$ is a contractible function.

\begin{rem}\label{rem:div:algo}
Unlike  in \cite{hrushovski:kazhdan:integration:vf}, the conclusion of this remark is not needed for the $\Gamma$-categories below (see Defintiton~\ref{def:Ga:cat}). We present it here for the sake of comparison (see \cite[Lemma~3.28, Definition~9.1]{hrushovski:kazhdan:integration:vf}).

Obviously the composition of two $\vrv$-contractible functions is a $\vrv$-contractible function. Let $f : A \fun B$ be an $\lan{RV}$-definable $\vrv$-contractible function. For all $\lbar t\in \RV$, since the underlying substructure $S$ is $\VF$-generated (see Convention~\ref{conv:s}), it is clear that if $\gamma \in \Gamma(\dcl(\lbar t))$ then $\vrv^{-1}(\gamma) \cap \RV(\dcl(\lbar t)) \neq \0$. This implies that,  if $f_{\downarrow} : \vrv(A) \fun \vrv(B)$ is a bijection then there is a $\lbar t$-$\lan{RV}$-definable $g(\lbar t) \in \vrv^{-1}(f^{-1}_{\downarrow}(\vrv(\lbar t)))$ for every $\lbar t\in B$ and hence, by compactness, we have an $\lan{RV}$-definable function $g$ on $B$ such that $g_{\downarrow} = f^{-1}_{\downarrow}$. Let $G$ be a $\mgl(\Z)$-transformation on $\vrv(A)$, that is, a bijection of the form $T \lbar x+ \lbar \gamma$ with $T \in \mgl(\Z)$ and $\lbar \gamma \in \Gamma(S)$. Let $\lbar t \in \RV(S)$ with $\vrv(\lbar t) = \lbar \gamma$ and $G^{\uparrow}$ the $\mgl(\Z)$-transformation on $A$ given by $T \lbar x \cdot \lbar t$. Then $f_{\downarrow} \circ G^{-1}$ is the contraction of $f \circ (G^{\uparrow})^{-1}$; similarly if $G$ is a $\mgl(\Z)$-transformation on $\vrv(B)$. Lastly, let $\pr_E$ be a coordinate projection on $B$ and $\vrv(B)$. It is straightforward to check that $\pr_E \circ f_{\downarrow}$ is the contraction of $\pr_E \circ f$.

We have just shown that the class $\gO$ of contractions of $\lan{RV}$-definable $\vrv$-contractible functions is closed under composition, inversion, composition with $\mgl(\Z)$-transformations, and composition with coordinate projections.

Now suppose that $A \sub \RV^n$, $B \sub \RV$, and $(\lbar \alpha, \beta) \in \vrv(f) \cap \Gamma^{n+1}$. Let $\phi(\lbar Y, Z)$ be a quantifier-free $\mdl L_{\K \Gamma}$-formula with parameters $\lbar \alpha$, $\beta$ that defines $f_{(\lbar \alpha, \beta)} : A_{\lbar \alpha} \fun B_{\beta}$. Clearly we may assume that all $\RV$-sort literals occurring in $\phi$ are of the form $t \lbar Y^{\lbar n} Z^m \, \Box \, 1$, where $t \in \RV(S)$, $\lbar n, m \in \Z$, and $\Box$ is $=$ or $\neq$. Since $t \lbar Y^{\lbar n} Z^m = 1$ is equivalent to $\res(t \lbar Y^{\lbar n} Z^m) = 1$ and $f_{(\lbar \alpha, \beta)}$ is a function, we see that $\phi$ contains irredundant $\K$-sort equalities between sums of terms of the form $\res(t \lbar Y^{\lbar n} Z^m)$ with $\vrv(t) + \sum_i n_i \alpha_i + m \beta = 0$. Observe that
\[
\res(t \lbar Y^{\lbar n} Z^m) = (t / \rcsn(\vrv(t))) \cdot (\lbar Y / \rcsn(\lbar \alpha))^{\lbar n} \cdot (Z / \rcsn(\beta))^m.
\]
We may treat $\lbar Y / \rcsn(\lbar \alpha)$, $Z / \rcsn(\beta)$ as variables in these equalities and consequently may assume $\lbar n, m \in \N$. Applying the Euclidean algorithm, we see that, away from a $\rcsn(\lbar \alpha, \beta)$-$\lan{RV}$-definable subset of $\vrv^{-1}(\lbar \alpha)$ of $\RV$-dimension $< n$ (recall~\cite[Definition~4.9]{Yin:special:trans}), the twistback of $f_{(\lbar \alpha, \beta)}$ is given by
\[
Z / \rcsn(\beta) = \sum_i F_i(\lbar Y / \rcsn(\lbar \alpha)) \bigg / \sum_j G_j (\lbar Y / \rcsn(\lbar \alpha)),
\]
where $F_i(\lbar Y / \rcsn(\lbar \alpha))$, $G_j (\lbar Y / \rcsn(\lbar \alpha))$ are monomials such that, for any $i$, $j$ and any $\lbar t \in \vrv^{-1}(\lbar \alpha)$,
\[
\beta = \vrv(F_i(\lbar t)) - \vrv(G_j(\lbar t)).
\]
This means that there are integers $n_i \in \Z$ and a $\delta \in \Gamma(S)$ (note that $(\lbar \alpha, \beta)$ is not needed to define this $\delta$) such that $\beta = \delta + \sum_i n_i \alpha_i$.

In summary, by compactness, all functions in $\gO$ are definably piecewise $\Z$-linear (with constant terms). Moreover, if $h : I \fun J$ is a bijection in $\gO$ and $I, J \sub \Gamma^n$ then $h$ is definably a piecewise $\mgl_n(\Z)$-transformation. This follows from the next lemma, which holds in a more general setting.
\end{rem}

\begin{lem}\label{gen:mat:inv}
Let $R$ be an integral domain and $M$ be a torsion-free $R$-module, viewed as the main sort of a first-order structure of some expansion of the usual $R$-module language. Let $\gO$ be a class of definable functions in the sort $M$ such that
\begin{enumerate}
  \item $\gO$ contains all the identity functions and all functions in $\gO$ are definably piecewise $R$-linear,
  \item $\gO$ is closed under composition, inversion, composition with $\mgl(R)$-transformations, and composition with coordinate projections (in the sense described above).
\end{enumerate}
If $g : D \fun E$ is a bijection in $\gO$, where $D, E \sub M^n$, then $h$ is definably a piecewise $\mgl_n(R)$-transformation.
\end{lem}
\begin{proof}
We proceed by induction on $n$. The base case $n=1$ is clear. For the inductive step, without loss of generality, we may assume that both $g$ and $g^{-1}$ are $R$-linear, given respectively by $\lbar x \efun A \lbar x + \lbar a$ and $\lbar x \efun B \lbar x + \lbar b$. Observe that if there are distinct $\lbar x_1, \lbar x_2 \in D$ such that $\pr_{\tilde k}(\lbar x_1) = \pr_{\tilde k}(\lbar x_2)$ for some $1 \leq k \leq n$ then the $k$th column of $B A - I_n$ must be $0$; similarly for $E$ and $A B - I_n$. Therefore we are reduced to the situation where this fails for $D$, $E$ with respect to some $1 \leq k, l \leq n$. After $\mgl_n(R)$-transformations if necessary, we may assume that $k = l = n$. Then $D$ is the graph of a function $\alpha : \pr_{< n}(D) \fun \pr_{n}(D)$ and $E$ is the graph of a function $\beta : \pr_{< n}(E) \fun \pr_{n}(E)$. Note that $\alpha = (\pr_n \rest D) \circ (\pr_{<n} \rest D)^{-1}$ and $\beta = (\pr_n \rest E) \circ (\pr_{<n} \rest E)^{-1}$. Let $g^* :\pr_{< n}(D) \fun \pr_{< n}(E)$ be the bijection $(\pr_{<n} \rest E) \circ g \circ (\pr_{<n} \rest D)^{-1}$. By the assumed two conditions, $\alpha$, $\beta$, and $g^*$ are all in $\gO$.

By compactness, we may further assume that $\alpha$ is given by $\lbar x \efun \lbar r \cdot \lbar x + c$, where $\lbar r  \in R^{n-1}$, and, by the inductive hypothesis, $g^*$ is given by $\lbar x \efun T \lbar x + \lbar d$, where $T \in \mgl_{n-1}(R)$. Let $(\lbar s_n, s_n)$ be the last row of $A$ and $a_n$ the last entry of $\lbar a$. Set
\[
A^* = \begin{bmatrix}
       T & 0           \\[0.3em]
       \lbar s_n + s_n\lbar r + \lbar r & -1
     \end{bmatrix}
     \quad \text{and} \quad \lbar a^* = (\lbar d, s_n c + c + a_n).
\]
Then $g$ is given by $\lbar x \efun A^* \lbar x + \lbar a^*$, as required.
\end{proof}

\begin{lem}\label{resg:decom}
Let $A \sub \RV^k \times \Gamma^l$ be an $\lan{RV}^{3}$-definable (resp.\ $\lan{RV}$-definable) subset. Set $\pr_{\leq k}(A) = U$ and suppose that $\vrv(U)$ is finite. Then there is an $\lan{RV}^{3}$-definable (resp.\ $\lan{RV}$-definable) finite partition $U_i$ of $U$ such that, for each $i$, $\fib(A, \lbar t) = \fib(A, \lbar t')$ for all $\lbar t, \lbar t' \in U_i$.
\end{lem}
\begin{proof}
Since the $\Gamma$-sort is stably embedded (see Remark~\ref{rem:gam:stab}), by Lemma~\ref{gam:3:1}, $\fib(A, \lbar t)$ is $\vrv(\lbar t)$-$\lan{RV}$-definable for every $\lbar t \in U$. Since $\vrv(\lbar t)$ is definable, the lemma simply follows from compactness.
\end{proof}

\begin{rem}\label{rem:monster:imme}
We clearly have $\acl^{1}(\RV) = \acl^{\bb T}(\RV)$, which is a model of $\bb T(S)$, and $\VF(\acl^{1}(\RV)) = (\VF(S) \cup \sn(\RV))^{\alg}$. But $\acl^{1}(\RV)$, as a valued field, is not maximally complete. In fact the underlying valued field of $\gC$ may be taken to be the unique maximal completion of $\acl^{1}(\RV)$, which is isomorphic to the field $\K \dlbr \Gamma \drbr$ of generalized formal Laurent series. Each element $a \in \VF$ may be written in the form $\sum_{i \in I} \sn(t_i)$, where $I$ is a well-ordered set and if $i < i'$ then $\vrv(t_i) < \vrv(t_{i'})$. We say that $\sn(t_i)$ is the \emph{$\vrv(t_i)$-component} of $a$ and denote it by $(a)_{\vrv(t_i)}$. Observe that if $a_1, \ldots, a_n \in \VF$ are of the same value $\gamma$ then $\vv(\sum_i a_i) > \gamma$ if and only if $\sum_i (a_i)_{\gamma} = 0$.

For any consistent set $\Phi(\lbar X)$ of $\mdl L_{\bb T}$-formulas with parameters in $\acl^{1}(\RV)$, where $\lbar X$ are the free variables and are all of the $\VF$-sort, if $\Phi(\lbar X)$ is realized in an immediate extension of $\acl^{1}(\RV)$ then it is realized in $\gC$, because any immediate extension of $\acl^{1}(\RV)$ may be embedded into $\gC$.
\end{rem}

\begin{conv}\label{conv:can}
We reiterate \cite[Convention~4.20]{Yin:QE:ACVF:min} here, since this trivial looking convention is actually quite crucial for understanding the whole construction, especially the parts that involve special bijections. For a subset $A\sub \VF^n \times \RV^m$, let
\[
\can(A) = \{(\lbar a, \rv(\lbar a), \lbar t) : (\lbar a, \lbar t) \in A\}.
\]
This is called the \emph{canonical image} of $A$ and $\can: A \fun \can(A)$ is called the \emph{canonical bijection} on $A$. The convention is that we shall tacitly substitute $\can(A)$ for $A$ in the discussion below if it is necessary or is just more convenient. Whether or not this substitution has been performed should be clear in context.
\end{conv}

\section{The Grothendieck semirings of $\RV$}\label{section:G:RV}

The main purpose of this section is to express the Grothendieck semirings of $\RV$-categories as tensor products of the Grothendieck semirings of $\Gamma$-categories and $\RES$-categories, which will be defined below. This works if certain conditions are met by $\bb T$, in particular, if $\bb T = \ACVF^{3}(0, 0)$. On the other hand, it does not seem straightforward to work out these conditions and it does seem to be an unworthy distraction here to digress into that direction. It is perhaps better to deal with it on a case-by-case basis when it is called for in future applications. In Hypothesis~\ref{hyp:jac} we describe what some of these conditions might be.

Of course, at the very least we can assume that $\gC^{\bb T}$ is an $\RV$-minimal expansion of $\gC^3$, that is, all definable $\RV$-sort subsets in $\gC^{\bb T}$ are already definable in $\gC^3$. However, for concreteness, we shall work in $\gC^{3}$ throughout this section. Hence, all definable subsets in this section are $\lan{RV}^{3}$-definable, unless indicated otherwise.

\begin{defn}
Let $A \sub \RV^n$ be a definable subset. A \emph{$\Gamma$-partition} of $A$ is a definable function $\pi : A \fun \Gamma_{\infty}^l$ such that, for $\lbar \gamma \in \Gamma_{\infty}^l$, $\pi^{-1}(\lbar \gamma)$ is contained in a $(\K^{\times})^n$-coset and is $\rcsn(\lbar \gamma)$-$\lan{RV}$-definable. If $\pi$ is a $\Gamma$-partition of $A$ then the \emph{$\RV$-dimension} of $\pi$, denoted by $\dim_{\RV} (\pi)$, is the number $\max \{\dim_{\RV} (\pi^{-1}(\lbar \gamma)) : \lbar \gamma \in \Gamma_{\infty}^{l} \}$.
\end{defn}

By Corollary~\ref{lrvdag:same:lrv}, the existence of such a $\Gamma$-partition of $A$ is easily verified by straightforward syntactical manipulation of any quantifier-free formula that defines $A$. This definition can be extended to definable subsets $A \sub \RV^n \times \Gamma_{\infty}^m$ in the obvious way.

\begin{lem}\label{G:part:dim:inv}
For any two $\Gamma$-partitions $\pi_1$, $\pi_2$ of $A \sub \RV^n \times \Gamma_{\infty}^m$, we have $\dim_{\RV} (\pi_1) = \dim_{\RV} (\pi_2)$.
\end{lem}
\begin{proof}
We may assume that $\pi_1$ is constant. Recall~\cite[Lemma~4.10, Lemma~4.11]{Yin:special:trans}, which essentially say that the $\RV$-dimension of any $\lan{RV}$-definable subset in the $\RV$-sort equals its algebraic dimension (Zariski dimension). Now observe that the algebraic dimension of $A$ over $\dcl(\rcsn(\Gamma))$ is still $\dim_{\RV} (A)$. This implies that $\dim_{\RV} (\pi_2^{-1}(\lbar \gamma)) = \dim_{\RV} (A)$ for some $\lbar \gamma \in \ran(\pi_2)$ and hence $\dim_{\RV} (A) = \dim_{\RV} (\pi_2)$.
\end{proof}

Therefore the $\RV$-dimension $\dim_{\RV}(A)$ of a definable subset $A \sub \RV^n \times \Gamma_{\infty}^m$ may be defined as the $\RV$-dimension of any $\Gamma$-partition of $A$. Note that the proof of the above lemma shows that $\dim_{\RV}(A)$ does not depend on parameters and if $f : A \fun \RV^k \times \Gamma_{\infty}^l$ is a definable function then $\dim_{\RV}(A) \geq \dim_{\RV}(f(A))$. Hence there is a definable finite-to-one function $f : A \fun \RV^k \times \Gamma_{\infty}^l$ if and only if there is a definable function $f : A \fun \RV^k$ such that all fibers are of $\RV$-dimension $0$ if and only if $\dim_{\RV}(A) \leq k$. We say that a property holds \emph{almost everywhere} on $A$ or \emph{for almost every element} in $A$ if it holds away from a definable subset $A' \sub A$ of a smaller $\RV$-dimension. This terminology will also be used when other notions of dimension are involved.

\begin{lem}\label{shift:K:csn:def}
Let $U \sub \RV^n$ be a definable subset and $\vrv(U) = D$. Then there is a definable finite partition $D_i$ of $D$ such that each $U_i = U \cap \vrv^{-1}(D_i)$ is a definable twistoid and the corresponding twistback is $\lan{RV}$-definable.
\end{lem}
\begin{proof}
Let $\pi : U \fun \Gamma_{\infty}^l$ be a $\Gamma$-partition. Without loss of generality, we may assume $\pi(U) \sub \Gamma^l$. By Corollary~\ref{gam:3:1} and compactness, there is an $\lan{RV}$-definable subset $B \sub \RV^{n+l}$ such that, for every $\lbar \gamma \in \Gamma^l$, $\pi^{-1}(\lbar \gamma) \times \{\rcsn(\lbar \gamma)\}$ is precisely the fiber of $B$ over $\rcsn(\lbar \gamma)$. Applying Lemma~\ref{shift:K} to $B$, we find an $\lan{RV}$-definable finite partition $I_i$ of $\vrv(B)$ such that each $B_i = B \cap \vrv^{-1}(I_i)$ is a twistoid with an $\lan{RV}$-definable $V_i \sub \K^{n+l}$ as its twistback. For each $\lbar \gamma \in D$ let
\[
E_{\lbar \gamma} = \{i : (\{\lbar \gamma\} \times \pi(U_{\lbar \gamma})) \cap I_i \neq \0 \}.
\]
Let $D_k$ be the definable finite partition of $D$ determined by the condition that $\lbar \gamma$, $\lbar \gamma'$ are in the same piece if and only if $E_{\lbar \gamma} = E_{\lbar \gamma'}$. Write $E_k$ for any $E_{\lbar \gamma}$ with $\lbar \gamma \in D_k$. Observe that, for any $\lbar \gamma \in D_k$, $\tbk(U_{\lbar \gamma}) = \bigcup_{i \in E_k} \fib(V_i, \lbar 1)$, where $\lbar 1 \in \K^l$. The lemma follows.
\end{proof}

The conclusion of this lemma shall be referred to as the \emph{twistoid condition}. This is a condition that should be imposed on a more general $\bb T$ (see Hypothesis~\ref{hyp:jac}). This will not interfere with the possibility of adding more structure to the residue field that expands the theory of algebraically closed fields.

\begin{cor}\label{twbk:fin}
If $U \sub \RV^n$ is a definable subset such that $\vrv(U)$ is a singleton then $U$ is $\lan{RV}$-definable.
\end{cor}

Therefore, for any $A \sub \RV^n$, $\vrv \rest A$ is a $\Gamma$-partition of $A$. This implies that $\dim_{\RV}(A) = k$ if and only if $\dim_{\RV}(\tbk(A_{\lbar \gamma})) = k$ for some $\lbar \gamma$.

\begin{cor}\label{csn:Gamma:fin}
Let $\lbar \gamma \in \Gamma^n$, $A \sub \vrv^{-1}(\lbar \gamma)$, and $f : A \fun \Gamma^m$ be a definable function. Then $f(A)$ is finite.
\end{cor}
\begin{proof}
This is immediate by applying Lemma~\ref{shift:K:csn:def} to the subset $\bigcup_{\lbar \alpha \in \Gamma^m} f^{-1}(\lbar \alpha) \times \{\rcsn(\lbar \alpha)\}$.
\end{proof}

\begin{cor}
Let $A \sub \RV^n$, $B \sub \RV^m$, and $F \sub A \times B$ be a definable finite-to-finite correspondence. Then $\vrv(F)$ is a finite-to-finite correspondence between $\vrv(A)$ and $\vrv(B)$.
\end{cor}

\begin{defn}
A nonempty definable subset $U \sub \RV^{n}$ is \emph{$\Gamma$-regular} if $\dim_{\RV}(U_{\lbar \gamma}) = n$ for all $\lbar \gamma \in \vrv(U)$.
\end{defn}

Note that if $U \sub \RV^{n}$ is $\Gamma$-regular then we actually have $U \sub (\RV^{\times})^{n}$. By convention, $U \sub \K^n$ is $\Gamma$-regular if $U \cap (\K^{\times})^n$ is $\Gamma$-regular.

\begin{lem}\label{RV:dim:zar}
Let $U \sub \RV^{n}$ be $\Gamma$-regular and $\vrv(U) = D$. Then $\dim_{\RV}(\vrv^{-1}(D) \mi U) < n$.
\end{lem}
\begin{proof}
This is immediate by Corollary~\ref{twbk:fin} and \cite[Lemma~4.10]{Yin:special:trans}.
\end{proof}

\begin{lem}\label{rem:gam:fit}
Let $D, E \sub \Gamma_{\infty}^n$. Let $W \sub \K^{k}$ be $\Gamma$-regular and $f : W \times D \fun W \times E$ be a definable bijection. Then there are definable subsets $A \sub W \times D$, $B \sub W \times E$ and a definable bijection $e: D \fun E$ such that, for all $\lbar \alpha \in D$, $f(\fib(A, \lbar \alpha)) = \fib(B, e(\lbar \alpha))$ and $\fib(A, \lbar \alpha)$, $\fib(B, e(\lbar \alpha))$ are $\Gamma$-regular.
\end{lem}
\begin{proof}
By Corollary~\ref{twbk:fin}, Corollary~\ref{csn:Gamma:fin}, and Lemma~\ref{RV:dim:zar}, for each $\lbar \alpha \in D$ there is a unique $\lbar \alpha$-definable $e(\lbar \alpha) \in E$ such that $A_{\lbar \alpha} = (W \times \{\lbar \alpha\}) \cap f^{-1}(W \times \{e(\lbar \alpha)\})$ is $\rcsn(\lbar \alpha)$-$\lan{RV}$-definable and $\dim_{\RV}((W \times \{\lbar \alpha\}) \mi A_{\lbar \alpha}) < k$. Symmetrically this also holds for each $\lbar \beta \in E$. Now the assertion simply follows from compactness.
\end{proof}

Since the $\Gamma$-sort is \omin-minimal, we can use the dimension theory of \omin-minimal structures. We shall call it $\Gamma$-dimension and denote the operator by $\dim_{\Gamma}$.

\begin{defn}[$\overline \Gamma$-categories]\label{defn:oG:cat}
The objects of the category $\overline{\Gamma}[k]$ are the definable subsets with coordinates in $\Gamma_{\infty}$ of $\Gamma$-dimension $k$. Its morphisms are the definable bijections between the objects. Set $\overline{\Gamma}_* = \bigcup_k \overline{\Gamma}[k]$.
\end{defn}

\begin{defn}[$\overline \RV$- and $\overline \RES$-categories]\label{defn:oRV:cat}
The objects of the category $\overline{\RV}[k]$ are the definable subsets with coordinates in $\RV$ of $\RV$-dimension $k$. Its morphisms are the definable bijections between the objects.

The category $\overline{\RES}[k]$ is the full subcategory of $\overline{\RV}[k]$ such that $U \in \overline{\RES}[k]$ if and only if all coordinates of $U$ are in $\K$.

Set $\overline{\RV}_* = \bigcup_{k} \overline{\RV}[k]$; similarly for $\overline{\RES}_*$.
\end{defn}

By the computation in \cite{kage:fujita:2006}, $\Z[X]/(X^2 + X) = \ggk \overline{\Gamma}_*$ via the map $X \efun [(0, \infty)]$, which is much simpler than $\gsk \overline{\Gamma}_*$. On the other hand, it is well-known that $\ggk \overline{\RES}_*$ is still quite complicated (see \cite[Example~3.7]{kra:scan:2000}). Anyway, following the philosophy of \cite{hrushovski:kazhdan:integration:vf}, we shall work with Grothendieck semirings whenever possible.

We clearly have $\gsk \overline{\Gamma}[0] = \N$. By Corollary~\ref{twbk:fin}, $U \in \overline{\RES}[0]$ if and only if $U$ is finite and hence $\gsk \overline{\RES}[0]$ contains $\N$ as a proper sub-semiring.

Each $\gsk \overline{\Gamma}[k]$ is identified canonically with a sub-semigroup of $\gsk \overline{\Gamma}_*$. These sub-semigroups satisfy the conditions:
\[
\begin{cases}
  \gsk \overline{\Gamma}[k] \cap \gsk \overline{\Gamma}[l] = \{0\}, & \text{if } k \neq l,\\
  \gsk \overline{\Gamma}[k] + (\gsk \overline{\Gamma}[l] \mi \{0\}) = \gsk \overline{\Gamma}[l], & \text{if } k \leq l.
\end{cases}
\]
In this situation, we may think of $\gsk \overline{\Gamma}_*$ as a disjoint union $\biguplus_k \gsk \overline{\Gamma}[k]$; similarly $\gsk \overline{\RES}_* = \biguplus_k \gsk \overline{\RES}[k]$ and $\gsk \overline{\RV}_* = \biguplus_k \gsk \overline{\RV}[k]$. Note that $\overline{\Gamma}_*$ is equivalent via the reduced cross-section to a full subcategory of $\overline{\RV}[0]$ and hence $\gsk \overline{\Gamma}_*$ may be canonically identified with a sub-semiring of $\gsk \overline{\RV}[0]$.

For $U \in \overline{\RV}[k]$ and $I \in \overline{\Gamma}_*$ we write $U \times_{\rcsn} I$ for the object $U \times \rcsn(I) \in \overline{\RV}[k]$. The map from $\gsk \overline{\RES}_* \times \gsk \overline{\Gamma}_*$ to $\gsk \overline \RV_*$ naturally determined by the assignment $([U], [I]) \efun [U \times_{\rcsn} I]$ is clearly $\N$-bilinear. Hence it induces a semiring homomorphism:
\[
\overline{\bb D}: \gsk \overline{\RES}_* \otimes \gsk \overline{\Gamma}_* \fun \gsk \overline \RV_*.
\]

\begin{lem}\label{red:D:iso}
$\overline{\bb D}$ is a semiring isomorphism.
\end{lem}
\begin{proof}
Surjectivity of $\overline{\bb D}$ follows immediately from Lemma~\ref{shift:K:csn:def}. For injectivity, let $U_i, V_j \in \overline{\RES}_*$, $I_i, J_j \in \overline{\Gamma}_*$, and
\[
f : A = \biguplus_i U_i \times_{\rcsn} I_i \fun B = \biguplus_j V_j \times_{\rcsn} J_j
\]
be a definable bijection. We need to show that $\sum_i [U_i] \otimes [I_i] = \sum_j [V_j] \otimes [J_j]$. Set
\[
C = \{(\lbar t, \lbar s, \lbar \alpha, \lbar \beta) : f(\lbar t, \rcsn(\lbar \alpha)) = (\lbar s, \rcsn(\lbar \beta))\}
\]
and $W = \prv(C)$. By Lemma~\ref{resg:decom}, there is a definable finite partition $W_k$ of $W$ such that
\[
\fib(C, (\lbar t, \lbar s)) = \fib(C, (\lbar t', \lbar s')) \quad \text{for all} \quad (\lbar t, \lbar s), (\lbar t', \lbar s') \in W_k.
\]
Since $f$ is a bijection, clearly each $\fib(C, (\lbar t, \lbar s))$ is the graph of a bijection and, in each $W_k$, $\lbar t = \lbar t'$ if and only if $\lbar s = \lbar s'$. So $W_k$ is also the graph of a bijection. Actually, we can form the disjoint unions $A$, $B$ in such a way (say, by tagging on both sides of the products) that $W_k \sub U_i \times V_j$ for some $i$, $j$. The desired equality follows easily from these conditions.
\end{proof}

\begin{cor}\label{gsk:gamma:inj}
For any $[U] \in \gsk \overline{\RES}_*$, if $U$ is $\Gamma$-regular then the semigroup homomorphism
\[
[U] \otimes - : \gsk \overline{\Gamma}_* \fun \gsk \overline{\RES}_* \otimes \gsk \overline{\Gamma}_*
\]
is injective.
\end{cor}
\begin{proof}
Suppose that $[U] \otimes [I] = [U] \otimes [J]$. By Lemma~\ref{red:D:iso}, $U \times I$ is definably bijective to $U \times J$. By Lemma~\ref{rem:gam:fit}, $[I] = [J]$.
\end{proof}

Let $U \sub \RV^n \times \Gamma^m$, $V \sub \RV^{n'} \times \Gamma^{m'}$, and $C \sub U \times V$ be definable subsets. For all $((\lbar u, \lbar \alpha), (\lbar v, \lbar \beta)) \in C$, the \emph{$\Gamma$-Jacobian} of $C$ at $((\lbar u, \lbar \alpha), (\lbar v, \lbar \beta))$, written as $\jcb_{\Gamma} C((\lbar u, \lbar \alpha), (\lbar v, \lbar \beta))$, is the element
\[
- \Sigma (\vrv(\lbar u), \lbar \alpha) + \Sigma (\vrv(\lbar v), \lbar \beta) \in \Gamma,
\]
where $\Sigma (\gamma_1, \ldots, \gamma_n) = \gamma_1 + \cdots + \gamma_n$. If $U, V \sub \K^n$, $\dim_{\RV}(U) = \dim_{\RV}(V) = n$, and $C \sub U \times V$ is a finite-to-finite correspondence then, for almost all $(\lbar u, \lbar v) \in C$, the Jacobian at $(\lbar u, \lbar v)$ may be defined in the natural way (see the discussion preceding~\cite[Definition~9.14]{Yin:special:trans}), which is a $(\lbar u, \lbar v)$-definable element in $\K^{\times}$ and is denoted by $\jcb_{\K} C(\lbar u, \lbar v)$. More generally, if $U, V \sub (\RV^{\times})^n$, then, for any $\lbar \alpha, \lbar \beta \in \Gamma^n$, we may consider the $(\lbar \alpha, \lbar \beta)$-twistback $\tbk(C_{\lbar \alpha, \lbar \beta})$ of $C$:

\begin{defn}\label{defn:jac}
The \emph{Jacobian $\jcb_{\RV} C(\lbar u, \lbar v) = (\jcb_{\K} C(\lbar u, \lbar v), \jcb_{\Gamma} C(\lbar u, \lbar v))$} of $C$ at $(\lbar u, \lbar v)$ is a $(\lbar u, \lbar v)$-definable pair in $\K^{\times} \times \Gamma^{\times}$, where $\jcb_{\K} C(\lbar u, \lbar v)$, if it exists, is given by
\[
\jcb_{\K} \tbk(C_{\vrv(\lbar u), \vrv(\lbar v)})(\tbk(\lbar u), \tbk(\lbar v)).
\]
\end{defn}

It is routine to check that the Jacobian is defined for almost all $(\lbar u, \lbar v) \in C$.

\begin{hyp}\label{hyp:jac}
Here we can provide a bit more information than at the beginning of this section on what conditions a more general $\bb T$ should satisfy in order to make the construction work. The twistoid condition should hold. The $\Gamma$-sort should be \omin-minimal. There should be a notion of $\RV$-dimension that agrees with the Zariski dimension, that is, if $U \sub \K^n$ is an $\mdl L_{\bb T}$-definable subset then its $\RV$-dimension equals the Zariski dimension of its Zariski closure. Consequently, the Jacobian in the $\RV$-sort may be defined as in Definition~\ref{defn:jac}.
\end{hyp}

\begin{defn}[Coarse $\RV$-categories]\label{defn:c:RV:cat}
An object of the category $\RV[k, \cdot]$ is a definable pair $(U, f)$, where $U \in \overline{\RV}_*$ and $f : U \fun (\RV^{\times})^k$ is a function. Given two such objects $(U, f)$ and $(V, g)$, any definable bijection $F : U \fun V$ is a \emph{morphism} of $\RV[k, \cdot]$. Such a morphism $F$ induces a correspondence between $f(U)$ and $g(V)$:
\[
\{(\lbar t, \lbar s) \in f(U) \times g(V) : \ex{\lbar u \in U} (f(\lbar u) = \lbar t \wedge (g \circ F)(\lbar u) = \lbar s) \},
\]
which is denoted by $F^{\rightleftharpoons}$.

An object of the category $\mgRV[k]$ is a definable triple $(U, f, \omega_{\Gamma})$, where $(U, f) \in \RV[k, \cdot]$ and $\omega_{\Gamma} : U \fun \Gamma$ is a function, which is understood as a \emph{$\Gamma$-volume form} on $U$. A \emph{morphism} $F : (U, f, \omega_{\Gamma}) \fun (U', f', \omega'_{\Gamma})$ of $\mgRV[k]$ is an $\RV[k, \cdot]$-morphism such that, for all $(f(\lbar u), (f' \circ F)(\lbar u)) \in F^{\rightleftharpoons}$,
\[
\omega_{\Gamma}(\lbar u) = \omega'_{\Gamma}(F (\lbar u)) + \jcb_{\Gamma} F^{\rightleftharpoons}(f(\lbar u), (f' \circ F)(\lbar u)).
\]

Set $\RV[\leq k, \cdot] = \coprod_{i \leq k} \RV[i, \cdot]$ and $\RV[*, \cdot] = \coprod_k \RV[k, \cdot]$; similarly for $\mgRV[\leq k]$ and $\mgRV[*]$.
\end{defn}

\begin{rem}
The categories $\RV[k, \cdot]$ only play an auxiliary role in the construction and could have been defined in a simpler way, that is, the function $f$ may be deleted from $(U, f)$ without any real consequences. We have chosen to define them in this way so to make other definitions below more compact. In those definitions the ``presentation'' $f$ of $U$ is indeed essential.

Note that, in the above definition and other similar ones below, all morphisms are actually isomorphisms. Also, for the cases $k =0$, the reader should interpret things such as $(\RV^{\times})^0$ and how they interact with other things in a natural way. For example, $(\RV^{\times})^0$ may be treated as the empty tuple. This results in the interpretation that the requirement above on $\Gamma$-volume forms for $k=0$ is simply $\omega_{\Gamma}(\lbar u) = \omega'_{\Gamma}(F (\lbar u))$.

About the notation: $\overline{\Gamma}_*$ etc.\ suggests that the category is filtrated and the notation $\RV[*, \cdot]$ etc.\ suggests that the category is actually graded.
\end{rem}

\begin{defn}[Fine $\RV$-categories]\label{defn:f:RV:cat}
The category $\RV[k]$ is the full subcategory of $\RV[k, \cdot]$ such that $(U, f) \in \RV[k]$ if and only if $\dim_{\RV}(f^{-1}(\lbar t)) = 0$ for all $\lbar t \in (\RV^{\times})^k$. Note that, for any $\RV[k]$-morphism $F : (U, f) \fun (V, g)$, by Corollary~\ref{twbk:fin}, the correspondence $F^{\rightleftharpoons}_{\lbar \gamma}$ is finite-to-finite for all $\lbar \gamma \in \Gamma^{2k}$.

An object of the category $\mRV[k]$ is a definable triple $(U, f, \omega)$, where $(U, f) \in \RV[k]$ and $\omega : U \fun \K^{\times} \times \Gamma$ is a function, which is understood as a \emph{volume form} on $U$. We also write $\omega$ as a pair $(\omega_{\K}, \omega_{\Gamma})$. A $\mgRV[k]$-morphism $F : (U, f, \omega) \fun (U', f', \omega')$ is a \emph{pseudo-morphism} of $\mRV[k]$. If, in addition, for all $(f(\lbar u), (f' \circ F)(\lbar u)) \in F^{\rightleftharpoons}$ that are away from a subset of $F^{\rightleftharpoons}$ of $\RV$-dimension $< k$,
\[
\omega_{\K}(\lbar u) = \omega'_{\K}(F (\lbar u)) \jcb_{\K} F^{\rightleftharpoons}(f(\lbar u), (f' \circ F)(\lbar u))
\]
then $F$ is a \emph{morphism} of $\mRV[k]$.

Set $\RV[*] = \coprod_k \RV[k]$; similarly for $\mRV[*]$.
\end{defn}

Observe that if $(U, f, \omega), (V, g, \sigma) \in \mRV[k]$ and $\dim_{\RV}(f(U)) < k$ (in particular, if $\dim_{\RV}(U) < k$) then any $\mRV[k]$-pseudo-morphism between them is indeed a $\mRV[k]$-morphism.

\begin{defn}[$\RES$-categories]\label{defn:RES:cat}
The category $\RES[k, \cdot]$ is the full subcategory of $\RV[k, \cdot]$ such that $(U, f) \in \RES[k, \cdot]$ if and only if all coordinates of $U$ and $f(U)$ are in $\K$. Similarly, $\RES[k]$ is such a full subcategory of $\RV[k]$, which is also a full subcategory of $\RES[k, \cdot]$. The category $\mRES[k]$ is the full subcategory of $\mRV[k]$ such that $(U, f, \omega) \in \mRES[k]$ if and only if $(U ,f) \in \RES[k]$ and $\omega_{\Gamma} = 0$.

The category $\RES^{c}[k]$ (resp.\ $\mRES^{c}[k]$) is the smallest full subcategory of $\RES[k]$ (resp.\ $\mRES[k]$) that contains the isomorphism class of $\mathbf{T}^k = ((\K^{\times})^k, \id)$ (resp.\ $\mathbf{T}_{\mu}^k = ((\K^{\times})^k, \id, (1, 0))$) and is closed under disjoint union.

Set $\RES[*, \cdot] = \coprod_k \RES[k, \cdot]$; similarly for $\RES[*]$, $\mRES[*]$, $\RES^{c}[*]$, and $\mRES^{c}[*]$.
\end{defn}

We do not have $\RES$-categories with $\Gamma$-volume forms because, in light of Corollary \ref{csn:Gamma:fin}, there will be no need to. Also note that $\RES[k, \cdot]$ is canonically isomorphic to a full subcategory of $\mgRV[k]$ via the map $(U, f) \efun (U, f, 0)$ and $\RES^{c}[k]$ is a full subcategory of $\RES[k, \cdot]$. Also, by Corollary~\ref{twbk:fin}, every $(U, f) \in \RES[k, \cdot]$ is $\lan{RV}$-definable and hence $(U, f) \in \RES[k]$ if and only if $f$ is finite-to-one.

\begin{defn}[$\Gamma$-categories]\label{def:Ga:cat}
The objects of the category $\Gamma[k]$ are the definable pairs $(I, f)$, where $I \in \overline{\Gamma}_*$ and $f : I \fun \Gamma^k$ is a function. Given $(I, f), (J, g)\in \Gamma[k]$, any definable bijection $F : I \fun J$ is a \emph{morphism} of $\Gamma[k]$.

An object of the category $\mG[k]$ is a definable triple $(I, f, \omega)$, where $(I, f) \in \Gamma[k]$ and $\omega : I \fun \Gamma$ is a function, which is understood as a \emph{volume form} on $I$. Let $\omega_f : I \fun \Gamma$ be the function given by $\lbar \gamma \efun \Sigma f(\lbar \gamma) + \omega(\lbar \gamma)$. A \emph{morphism} $F : (I, f, \omega) \fun (I', f', \omega')$ of $\mG[k]$ is a $\Gamma[k]$-morphism such that $\omega_f(\lbar \gamma) = \omega'_{f'}(F(\lbar \gamma))$ for all $\lbar \gamma \in I$.

The category $\Gamma^{c}[k]$ is the full subcategory of $\Gamma[k]$ such that $(I, f) \in \Gamma^{c}[k]$ if and only if $I$ is finite. The category $\mG^{c}[k]$ is the full subcategory of $\mG[k]$ such that $(I, f, \omega) \in \Gamma^{c}[k]$ if and only if $I$ is finite and $\omega_f(\lbar \gamma) = 0$ for all $\lbar \gamma \in I$.

Set $\Gamma[*] = \coprod_k \Gamma[k]$; similarly for $\mG[*]$, $\Gamma^{c}[*]$, and $\mG^{c}[*]$.
\end{defn}

Obviously $\gsk \RV[*] = \bigoplus_k \gsk \RV[k]$; similarly for the other graded categories.

Note that the semigroups $\gsk \RES^{c}[k]$, $\gsk \mRES^{c}[k]$, $\gsk \Gamma^{c}[k]$, and $\gsk \mG^{c}[k]$ may be identified with $(\N, +)$ and hence the semirings $\gsk \RES^{c}[*]$, $\gsk \mRES^{c}[*]$, $\gsk \Gamma^{c}[*]$, and $\gsk \mG^{c}[*]$ may be identified with $\N[X]$, the semiring of polynomials with coefficients in $\N$. Let us abbreviate
\[
\gsk \RES[k, \cdot] \otimes \gsk \Gamma[k], \quad \gsk \RES[k, \cdot] \otimes \gsk \mG[k], \quad \gsk \mRES[k] \otimes \gsk \mG[k]
\]
as $\vtp[k]$, $\mgvtp[k]$, $\mvtp[k]$, respectively. Note that both $\vtp[k]$ and $\mgvtp[k]$ use $\gsk \RES[k, \cdot]$ as the first factor. Set $\vtp[*] = \bigoplus_k \vtp[k]$, similarly for $\mgvtp[*]$ and $\mvtp[*]$. These are graded semirings.

For $(U, f) \in \RES[k, \cdot]$ and $(I, g) \in \Gamma[k]$, let $f \times_{\rcsn} g : U \times_{\rcsn} I \fun (\RV^{\times})^k$ be the function given by
\[
(\lbar t, \rcsn(\lbar \gamma)) \efun (f(\lbar t)_i \rcsn(g(\lbar \gamma))_i).
\]
We write $(U, f) \times_{\rcsn} (I, g)$ for the object
\[
(U \times_{\rcsn} I, f \times_{\rcsn} g) \in \RV[k, \cdot].
\]
Note that if $(U, f) \in \RES[k]$ then $(U, f) \times_{\rcsn} (I, g) \in \RV[k]$. For $(I, g, \sigma) \in \mG[k]$, let $(U, f) \times_{\rcsn} (I, g, \sigma)$ be the object
\[
((U, f) \times_{\rcsn} (I, g), \sigma_{\Gamma}) \in \mgRV[k],
\]
where $\sigma_{\Gamma}$ is the volume form on $U \times_{\rcsn} I$ given by $(\lbar t, \rcsn(\lbar \gamma)) \efun \sigma(\lbar \gamma)$. Finally, for $(U, f, \omega) \in \mRES[k]$, let $(U, f, \omega) \times_{\rcsn} (I, g, \sigma)$ be the object
\[
((U, f) \times_{\rcsn} (I, g), \omega \times_{\rcsn} \sigma) \in \mRV[k],
\]
where $\omega \times_{\rcsn} \sigma$ is the volume form on $U \times_{\rcsn} I$ given by $(\lbar t, \rcsn(\lbar \gamma)) \efun (\omega(\lbar t), \sigma(\lbar \gamma))$.

The assignment $([\textbf{U}], [\textbf{I}]) \efun [\textbf{U} \times_{\rcsn} \textbf{I}]$ naturally determines a map
\[
\gsk \RES[k, \cdot] \times \gsk \Gamma[k] \fun \gsk \RV[k, \cdot],
\]
which is clearly $\N$-bilinear. Similarly there are such maps
\[
\gsk \RES[k, \cdot] \times \gsk \mG[k] \fun \gsk \mgRV[k] \quad \text{and} \quad \gsk \mRES[k] \times \gsk \mG[k] \fun \gsk \mRV[k].
\]
Hence we have three induced semigroup homomorphisms:
\[
\bb D_k: \vtp[k] \fun \gsk \RV[k, \cdot], \quad \mu_{\Gamma}  \bb D_k: \mgvtp[k] \fun \gsk \mgRV[k], \quad \mu \bb D_k: \mvtp[k] \fun \gsk \mRV[k].
\]

\begin{prop}\label{D:iso}
$\bb D_k$, $\mu_{\Gamma}  \bb D_k$, and $\mu \bb D_k$ are isomorphisms.
\end{prop}
\begin{proof}
Since $\bb D_0 = \mu_{\Gamma}  \bb D_0 = \overline{\bb D}$ and $\mu \bb D_0$ is a restriction of $\overline{\bb D}$, let us assume $k > 0$. We shall only be concerned with $\mu \bb D_k$, since for $\bb D_k$ or $\mu_{\Gamma}  \bb D_k$ the argument is similar and simpler. In fact, the proof is more or less the same as that of Lemma~\ref{red:D:iso} and hence we shall be brief.

For any $(U, f, \omega) \in \mRV[k]$, by Corollary~\ref{csn:Gamma:fin}, there is a definable finite partition $U_i$ of $U$ such that the restrictions $(\vrv \circ f) \rest U_i$, $\omega_{\Gamma} \rest U_i$ factor through $\vrv \rest U_i$. So, without loss of generality, we may assume that $U$ has this property. By Lemma~\ref{shift:K:csn:def}, we may further assume that (the graphs of) $f$ and $\omega_{\K}$ are twistoids. Then it is clear that $(U, f, \omega)$ is isomorphic to a product in the desired form.

For injectivity, in a similar notation to that in the proof of Lemma~\ref{red:D:iso}, we are reduced to showing that the bijections coded in $W_k$ are indeed $\mRES[k]$- and $\mG[k]$-morphisms. It is straightforward to check this.
\end{proof}

\begin{cor}
$\bb D = \bigoplus_k \bb D_k$, $\mu_{\Gamma} \bb D = \bigoplus_k \mu_{\Gamma} \bb D_k$, and $\mu \bb D = \bigoplus_k \mu \bb D_k$ are isomorphisms of graded semirings.
\end{cor}

\begin{cor}\label{gsk:mu:gamma:inj}
For any $[(U, f)] \in \gsk \RES[k, \cdot]$, if $U$ is $\Gamma$-regular then the semigroup homomorphism $[(U, f)] \otimes - : \gsk \Gamma[k] \fun \vtp[k]$ is injective; similarly for the other two cases.
\end{cor}
\begin{proof}
Since we have Proposition~\ref{D:iso}, bijections as described in Lemma~\ref{rem:gam:fit} may be obtained as in the proof of Corollary~\ref{gsk:gamma:inj}, which are indeed morphisms of the corresponding categories.
\end{proof}

For $(I, f) \in \Gamma[k]$ and $(I, f, \omega) \in \mG[k]$, we define their \emph{canonical lifting} into the corresponding $\RV$-categories:
\[
\bb L_{\Gamma}(I, f) = \mathbf{T}^k \times_{\rcsn} (I, f), \quad \mu_{\Gamma} \bb L_{\Gamma}(I, f, \omega) = \mathbf{T}^k \times_{\rcsn} (I, f, \omega), \quad \mu \bb L_{\Gamma}(I, f, \omega) = \mathbf{T}_{\mu}^k \times_{\rcsn} (I, f, \omega).
\]

\begin{cor}
These lifting maps induce canonical embeddings of graded semirings:
\[
\bb L_{\Gamma} : \gsk \Gamma[*] \fun \vtp[*], \quad \mu_{\Gamma} \bb L_{\Gamma} : \gsk \mG[*] \fun \mgvtp[*], \quad \mu \bb L_{\Gamma} : \gsk \mG[*] \fun \mvtp[*],
\]
which yield the canonical identifications:
\[
(\bb D \circ \bb L_{\Gamma})(\gsk \Gamma^{c}[*]) = (\mu_{\Gamma} \bb D \circ \mu_{\Gamma} \bb L_{\Gamma})(\gsk \mG^{c}[*]) = \gsk \RES^{c}[*] \quad \text{and} \quad (\mu \bb D \circ \mu \bb L_{\Gamma})(\gsk \mG^{c}[*]) = \gsk \mRES^{c}[*].
\]
\end{cor}

There is an alternative description of the semiring $\gsk \mG[*]$. For that, we introduce the following notation:

\begin{nota}\label{nota:par:exp}
Let $P$ be a subset of additional parameters. If $\mdl C$ is a category of $P$-definable subsets then we shall emphasize this by writing $\mdl C_P$. Let $A$, $B$ be two subsets. We write $[A] =_P [B]$ if $A$, $B$ are isomorphic objects in $\mdl C_P$.
\end{nota}

\begin{defn}\label{def:fn:ring}
A function $f : \Gamma \fun \gsk \Gamma[k]$ is \emph{definable} if there is a definable subset $I \sub \Gamma \times \Gamma^{m+k}$ such that, for all $\alpha \in \pr_1(I)$, $\fib(I, \alpha)$ encodes naturally a representative of $f(\alpha) \in \gsk \Gamma[k]_{\alpha}$. The subset $I$ is considered as a \emph{representative} of $f$. With pointwise addition, such definable functions form a semigroup $\fn(\Gamma, \gsk \Gamma[k])$. Given another definable function $g : \Gamma \fun \gsk \Gamma[l]$ with a representative $J$, it is routine to check that their convolution product is well-defined as follows:
\[
(f * g) (\gamma) = \biggl [ \bigcup_{\alpha + \beta = \gamma}\fib(I, \alpha) \times \fib(J, \beta) \biggr ] \in \gsk \Gamma[k+l]_{\gamma}.
\]
This makes $\fn(\Gamma, \gsk \Gamma[*]) = \bigoplus_k \fn(\Gamma, \gsk \Gamma[k])$ a graded semiring.
\end{defn}

\begin{lem}
Each $\fn(\Gamma, \gsk \Gamma[k])$ is canonically isomorphic to $\gsk \mG[k]$ and hence $\fn(\Gamma, \gsk \Gamma[*])$ is canonically isomorphic to $\gsk \mG[*]$.
\end{lem}
\begin{proof}
For $(I, f, \omega) \in \mG[k]$, set $\gamma \efun [(\omega_f^{-1}(\gamma), f \rest \omega_f^{-1}(\gamma))]$ for $\gamma \in \omega_f(I)$, which is a definable function in $\fn(\Gamma, \gsk \Gamma[k])$. Conversely, for any definable function $f : \Gamma \fun \gsk \Gamma[k]$ with a representative $I$, let $U_I = I$, $f_I : U_I \fun \Gamma^k$ be the projection to the last $k$ coordinates, and $\omega_{I} : U_I \fun \Gamma$ be the function given by $(\gamma, \lbar \alpha, \lbar \beta) \efun \gamma - \Sigma \lbar \beta$. Then $(U_I, f_I, \omega_{I}) \in \mG[k]$. It is routine to check that these maps induce isomorphisms as desired.
\end{proof}

There are two Euler characteristics $\chi_g$, $\chi_b$ that can be associated to the $\Gamma$-sort (see \cite[\S 4.2]{dries:1998}, \cite{kage:fujita:2006}, and also~\cite[\S 9]{hrushovski:kazhdan:integration:vf}). They are distinguished by $\chi_g((0, \infty)) = -1$ and $\chi_b((0, \infty)) = 0$. We shall denote both of them by $\chi$ if no distinction is needed. Using these and the groupifications of the results above, we can obtain various retractions to the Grothendieck rings of the $\RES$-categories.

\begin{lem}\label{gam:euler}
There are two homomorphisms $\mdl E_g, \mdl E_b : \ggk \Gamma[*] \fun \Z[\tau]$ and two homomorphisms $ \mu\mdl E_g, \mu \mdl E_b : \ggk \mG[*] \fun \Z[\tau]$ of graded rings.
\end{lem}
\begin{proof}
For $(I, f) \in \Gamma[k]$ and $(I, f, \omega) \in \mG[k]$ we simply set $\mdl E_k(I, f) = \chi(I)$ and $\mu \mdl E_k(I, f, \omega) = \chi(I)$. Clearly these maps induce graded ring homomorphisms $\mdl E = \bigoplus_k \mdl E_k$ and $\mu \mdl E = \bigoplus_k \mu \mdl E_k$.
\end{proof}

\begin{nota}\label{nota:RV:ele}
Let $\RV^{> 1} = \rv(\MM)$ and $(\RV^{\times})^{> 1} = \rv(\MM \mi \{0\})$. We introduce the following shorthand for some elements of the Grothendieck semigroups and their groupifications (and closely related constructions):
\begin{gather*}
[1]_0 = [\{1\}] \in \gsk \RES[0], \quad [1]_1 = [(\{1\}, \id)] \in \gsk \RES[1], \quad [1_{\mu}]_1 = [(\{1\}, \id, \id)] \in \gsk \mRES[1],\\
[0]_0 = [\{0\}] \in \gsk \Gamma[0], \quad [0]_1 = [(\{0\}, \id)] \in \gsk \Gamma[1], \quad [0_{\mu}]_1 = [(\{0\}, \id, \id)] \in \gsk \mG[1],\\
[\mathbf{H}]_1 = [((0, \infty), \id)] \in \gsk \Gamma[1], \quad [\mathbf{H}_{\mu}]_1 = [((0, \infty), \id, 0)] \in \gsk \mG[1],\\
\mathbf{j} = [((\RV^{\times})^{> 1}, \id)] -  [1]_1 \in \ggk \RV[1], \quad \mathbf{j}_{\mu_{\Gamma}} = [((\RV^{\times})^{> 1}, \id, 0)] -  [1]_1 \in \ggk \mgRV[1], \\
\mathbf{j}_{\mu} = [((\RV^{\times})^{> 1}, \id, (1, 0))] -  [1_{\mu}]_1 \in \ggk \mRV[1],\\
\mathbf{A} = [\mathbf{T}^1] +  [1]_1 \in \ggk \RES[1], \quad \mathbf{A}_{\mu} = [\mathbf{T}_{\mu}^1] +  [1_{\mu}]_1 \in \ggk \mRES[1].
\end{gather*}
As in~\cite{hrushovski:kazhdan:integration:vf}, the elements $[1]_0 + \mathbf{j} \in \ggk \RV[*, \cdot]$, $\mathbf{j}_{\mu_{\Gamma}} \in \ggk \mgRV[*]$, and $\mathbf{j}_{\mu} \in \ggk \mRV[*]$ are instrumental in the discussions below.
\end{nota}

\begin{prop}\label{prop:eu:retr}
There are two ring homomorphisms
\[
\bb E_g : \ggk \RV[*, \cdot] \fun \ggk \RES[*, \cdot][\mathbf{A}^{-1}] \quad \text{and} \quad \bb E_b : \ggk \RV[*, \cdot] \fun \ggk \RES[*, \cdot][[1]^{-1}_1]
\]
such that
\begin{enumerate}
  \item the ranges of $\bb E_g$, $\bb E_b$ are precisely the zeroth graded pieces of the targets,
  \item $\bb E_g([1]_0 + \mathbf{j}) = \bb E_b([1]_0 + \mathbf{j}) = 0$,
  \item for $x \in \ggk \RES[k, \cdot]$, $\bb E_g (x) = x \mathbf{A}^{-k}$ and $\bb E_b(x) = x [1]_1^{-k}$.
\end{enumerate}
With volume forms, we have two pairs of homomorphisms of graded rings:
\begin{gather*}
\mu_{\Gamma}\bb E_g: \ggk \mgRV[*] \fun \ggk \RES[*, \cdot]/ (\mathbf{A}) \quad \text{and} \quad \mu_{\Gamma} \bb E_b: \ggk \mgRV[*] \fun \ggk \RES[*, \cdot] / ([1]_1)\\
\mu\bb E_g: \ggk \mRV[*] \fun \ggk \mRES[*] / (\mathbf{A}_{\mu}) \quad \text{and} \quad \mu \bb E_b: \ggk \mRV[*] \fun \ggk \mRES[*] / ([1_{\mu}]_1)
\end{gather*}
such that their restrictions to $\ggk \RES[*, \cdot]$, $\ggk \mRES[*] $ are the natural projections and
\[
\mu_{\Gamma}\bb E_g (\mathbf{j}_{\mu_{\Gamma}})  = \mu_{\Gamma}\bb E_b (\mathbf{j}_{\mu_{\Gamma}}) = 0, \quad  \mu \bb E_g (\mathbf{j}_{\mu})  = \mu \bb E_b (\mathbf{j}_{\mu}) = 0.
\]
\end{prop}
\begin{proof}
For each $n$, let $E_{g, n}: \bigoplus_{i \leq n} \gvtp[i] \fun \ggk \RES[n, \cdot]$ and $E_{b, n}: \bigoplus_{i \leq n} \gvtp[i] \fun \ggk \RES[n, \cdot]$ be the surjective group homomorphisms given respectively by
\[
x \otimes y \efun \mdl E_{g, k}(y) x \mathbf{A}^{n-k} \quad \text{and} \quad x \otimes y \efun \mdl E_{b, k}(y) x [1]_1^{n-k},
\]
where $x \in \ggk \RES[k, \cdot]$, $y \in \ggk \Gamma[k]$, and $\mdl E_{g, k}$, $\mdl E_{b, k}$ are defined with respect to $\chi_g$, $\chi_b$ as in Lemma~\ref{gam:euler}. For each $n > 0$, we have
\begin{gather*}
(E_{g, n} \circ \bb D^{-1})([1]_0 + \mathbf{j}) = E_{g, n}([1]_0 \otimes [0]_0 + [\mathbf{T}^1] \otimes [\mathbf{H}]_1 - [1]_1 \otimes [0]_1) = \mathbf{A}^{n} - [\mathbf{T}^1] \mathbf{A}^{n-1} - [1]_1 \mathbf{A}^{n-1} = 0, \\
(E_{b, n} \circ \bb D^{-1})([1]_0 + \mathbf{j}) = E_{b, n}([1]_0 \otimes [0]_0 + [\mathbf{T}^1] \otimes [\mathbf{H}]_1 - [1]_1 \otimes [0]_1) = [1]_1^{n} + 0 - [1]_1 [1]_1^{n-1} = 0.
\end{gather*}
The group homomorphisms $g_n, b_n : \ggk \RES[n, \cdot] \fun \ggk \RES[n+1, \cdot]$ given respectively by $x \efun x \mathbf{A}$, $x \efun x [1]_1$ determine two colimit systems and the group homomorphisms $\bb E_{g, n} = E_{g, n} \circ \bb D^{-1}$, $\bb E_{b, n} = E_{b, n} \circ \bb D^{-1}$ determine two homomorphisms of colimit systems. Hence we have two surjective ring homomorphisms:
\[
\colim{n} \bb E_{g, n} : \ggk \RV[*, \cdot] \fun \colim{g_n} \ggk \RES[n, \cdot] \quad \text{and} \quad \colim{n} \bb E_{b, n} : \ggk \RV[*, \cdot] \fun \colim{b_n} \ggk \RES[n, \cdot].
\]
These yield the desired homomorphisms since the two colimits are respectively isomorphic to the zeroth graded pieces of $\ggk \RES[*, \cdot][\mathbf{A}^{-1}]$ and $\ggk \RES[*, \cdot][[1]_1^{-1}]$.

The cases with volumes forms are not very different and are left to the reader.
\end{proof}

Note that these homomorphisms are slightly different from the ones constructed in \cite[Theorem~10.5, Theorem~10.11]{hrushovski:kazhdan:integration:vf}.

\section{The Grothendieck semirings of $\VF$ and special bijections}\label{section:dim}

Let $A \sub \VF^n \times \RV^m$ be a definable subset. Recall that if $A$ equals its $\RV$-hull $\RVH(A)$ (see \cite[Definition~4.21]{Yin:QE:ACVF:min}) then $A$ is an \emph{$\RV$-pullback}. An $\rv$-polydisc $\gp \sub \VF^n \times \RV^m$ is \emph{degenerate} if $\dim_{\VF}(\gp) < n$. This happens if and only if some $\VF$-coordinate of $\gp$ is $0$, if there is one at all. An $\RV$-pullback is \emph{degenerate} if it contains a degenerate $\rv$-polydisc and is \emph{strictly degenerate} if it only contains degenerate $\rv$-polydiscs.

Let $\psi$ be a quantifier-free formula that defines $A$. By inspection of the complexity of the occurring $\VF$-terms of $\psi$, we see that there is a definable function $\pi : A \fun \RV^l$ and an $\lan{RV}$-formula $\phi$ such that each $\pi^{-1}(\lbar t)$ is contained in an $\rv$-polydisc and is defined by the formula $\phi(\sn(\lbar t))$. Hence the following definition is not empty.

\begin{defn}\label{defn:lrv}
An \emph{$\RV$-partition} of $A$ is a definable function $\pi : A \fun \RV^l$ such that, for every $\lbar t \in \ran(\pi)$, the fiber $\pi^{-1}(\lbar t)$ is $\sn(\lbar t)$-$\lan{RV}$-definable. We do not explicitly require that $\pi^{-1}(\lbar t)$ is contained in an $\rv$-polydisc, but this can always be achieved if needed.
\end{defn}

Similarly, by syntactical inspection, the twistoid condition (see Hypothesis~\ref{hyp:jac}), and compactness, the following specialization of the above definition is not empty either.

\begin{defn}
If $A$ is $\mdl L_{\wt{\bb T}}$-definable then a \emph{$\Gamma$-partition} of $A$ is an $\mdl L_{\wt{\bb T}}$-definable function $\pi : A \fun \Gamma_{\infty}^l$ such that each fiber $\pi^{-1}(\lbar \gamma)$ is contained in a coset of $(\rv^{-1}(\K^{\times}))^n \times (\K^{\times})^m$ and is $\csn(\lbar \gamma)$-$\lan{RV}$-definable.
\end{defn}

Note that if $A$ has no $\VF$-coordinates or is an $\RV$-pullback then it is $\mdl L_{\wt{\bb T}}$-definable and hence admits a $\Gamma$-partition.

If $A$ is $\lan{RV}$-definable then we may speak of the $\VF$-dimension of $A$ (see~\cite[Definition~4.1]{Yin:special:trans}). We may extend this notion of dimension to $\RV$-partitions, which is parallel to how $\RV$-dimension is extended to $\Gamma$-partitions above:

\begin{defn}
Let $\pi$ be an $\RV$-partition of $A$. The \emph{$\VF$-dimension} of $\pi$, denoted by $\dim_{\VF} (\pi)$, is the number $\max \{\dim_{\VF} (\pi^{-1}(\lbar t)) : \lbar t \in \ran(\pi) \}$.
\end{defn}

Let $B \sub \VF$ be an arbitrary subset. For any $(\lbar a, \lbar t) \in A$ let $\td_B(\lbar a, \lbar t)$ be the transcendental degree of $B^{\alg}(\lbar a)$ over $B^{\alg}$ (see Notation~\ref{nota:ac}). Let $\td_B(A) = \max \{\td_B(\lbar a, \lbar t) : (\lbar a, \lbar t) \in A \}$. If $B = \0$ then we omit it from the expression.

\begin{lem}\label{VF:dim:conserv}
Suppose that $A$ is $\lan{RV}$-definable. Then $\td(A) = \td_{\sn(\RV)}(A)$ and consequently $\dim_{\VF}(A) = \dim_{\VF} (\pi)$ for any $\RV$-partition $\pi$ of $A$.
\end{lem}
\begin{proof}
By~\cite[Corollary~5.6]{Yin:special:trans}, $A$ is $\lan{RV}$-definably bijective to an $\RV$-pullback $A'$. By~\cite[Lemma~3.3]{Yin:special:trans}, we have $\td(A) = \td(A')$ and $\td_{\sn(\RV)}(A) = \td_{\sn(\RV)}(A')$. So, for the first equality, it is enough to show that $\td(A') = \td_{\sn(\RV)}(A')$. By~\cite[Lemma~4.4, Lemma~4.6]{Yin:special:trans}, $A'$ contains an $\rv$-polydisc $\{(0, \ldots, 0)\} \times \rv^{-1}(\lbar t) \times \{\lbar s\}$, where $\lbar t \in (\RV^{\times})^k$ and $\td(A') = k$. By Remark~\ref{rem:monster:imme}, it is easy to see that there is an $\lbar a \in \rv^{-1}(\lbar t)$ that is algebraically independent over $\VF(\acl^{1}(\RV))$. Hence $\td_{\sn(\RV)}(A') = k$.

Now, let $\pi$ be an $\RV$-partition of $A$. We have $\td_{\sn(\lbar t)}(\pi^{-1}(\lbar t)) \leq \td(A)$ for every $\lbar t \in \ran(\pi)$ and, by the first equality, $\td_{\sn(\lbar t)}(\pi^{-1}(\lbar t)) = \td(A)$ for some $\lbar t \in \ran(\pi)$. Hence the second equality follows from~\cite[Lemma~4.4]{Yin:special:trans}.
\end{proof}

\begin{lem}\label{part:dim:inv}
For any two $\RV$-partitions $\pi_1$, $\pi_2$ of $A$, we have $\dim_{\VF} (\pi_1) = \dim_{\VF} (\pi_2)$.
\end{lem}
\begin{proof}
Let $\pi$ be the $\RV$-partition of $A$ given by $\lbar x \efun (\pi_1(\lbar x), \pi_2(\lbar x))$. For every $\lbar t \in \ran(\pi_1)$, since $\pi^{-1}_1(\lbar t)$ is $\sn(\lbar t)$-$\lan{RV}$-definable, by Lemma~\ref{VF:dim:conserv}, $\dim_{\VF}(\pi^{-1}_1(\lbar t)) = \dim_{\VF}(\pi \rest \pi^{-1}_1(\lbar t))$ and hence $\dim_{\VF} (\pi_1) = \dim_{\VF} (\pi)$. Since this also holds symmetrically for $\pi_2$, the lemma follows.
\end{proof}

Of course, \cite[Definition~4.1]{Yin:special:trans} still makes sense in the current context:

\begin{defn}
The \emph{$\VF$-dimension} of $A$, denoted by $\dim_{\VF}(A)$, is the smallest number $k$ such that there is a definable finite-to-one function $f: A \fun \VF^k \times \RV^l$ or, equivalently, there is a definable injection $f: A \fun \VF^k \times \RV^l$ (see~\cite[Lemma~4.2]{Yin:special:trans}).
\end{defn}


However, the $\VF$-dimension of $A$ itself and the $\VF$-dimension of the $\RV$-partitions of $A$ are really the same thing:

\begin{lem}\label{dim:comp:fiberwise}
Let $\pi$ be an $\RV$-partition of $A$, $\dim_{\VF}(A) = k$, and $A_{\pi} = \bigcup\{\pi^{-1}(\lbar t) : \dim_{\VF} (\pi^{-1}(\lbar t)) = k \}$. Then $\dim_{\VF}(\pi) = k$ and $\dim_{\VF}(A \mi A_{\pi}) < k$.
\end{lem}
\begin{proof}
By compactness, obviously $k \leq \dim_{\VF}(\pi)$. For the other direction, suppose that $f: A \fun \VF^k \times \RV^l$ is a witness to $\dim_{\VF}(A) = k$. Let $\rho$ be an $\RV$-partition of (the graph of) $f$, which obviously also carries an $\RV$-partition $\pi'$ of $A$ such that $\ran(\pi') = \ran(\rho)$ and $\pi'^{-1}(\lbar t) = \dom(\rho^{-1}(\lbar t))$ for every $\lbar t \in \ran(\pi')$. By Lemma~\ref{part:dim:inv}, $\dim_{\VF}(\pi) = \dim_{\VF}(\pi') \leq k$.

The second item is a corollary of the first. Note that it makes sense since, by~\cite[Lemma~4.6]{Yin:special:trans}, $A_{\pi}$ is definable.
\end{proof}

For any definable function $f : \VF^n \fun \VF^m$, the derivative and the partial derivatives of $f$ at a point are defined exactly as in~\cite[Definition~9.6]{Yin:special:trans}. Standard properties of differentiation such as the product rule and the chain rule only depend on the valuation and hence hold regardless of the presence of a section $\sn$ and additional structure in the $\RV$-sort.

\begin{lem}\label{partial:exist:sec}
Let $f : \VF^n \fun \VF^m$ be a definable function. Then each partial derivative $\partial^{ij} f$ is defined almost everywhere.
\end{lem}
\begin{proof}
Let $\rho$ be an $\RV$-partition of $f$. For each $\lbar t \in \ran(\rho)$, $f_{\lbar t} = \pi^{-1}(\lbar t) \sub f$ is an $\sn(\lbar t)$-$\lan{RV}$-definable function and hence, by Lemma~\cite[Lemma~9.8]{Yin:special:trans}, there is an $\sn(\lbar t)$-$\lan{RV}$-definable subset $A_{\lbar t} \sub \dom(f_{\lbar t})$ with $\dim_{\VF}(A_{\lbar t}) < n$ such that every partial derivative $\partial^{ij} f_{\lbar t}$ is defined everywhere in $\dom(f_{\lbar t}) \mi A_{\lbar t}$. By compactness, we may take $A = \bigcup_{\lbar t} A_{\lbar t} \sub \VF^n$ to be definable, and there is an $\RV$-partition $\pi$ of $A$ such that $\ran(\rho) = \ran(\pi)$ and $\pi^{-1}(\lbar t) = A_{\lbar t}$ for every $\lbar t \in \ran(\pi)$. By Lemma~\ref{dim:comp:fiberwise}, $\dim_{\VF}(A) = \dim_{\VF}(\pi) < n$.
\end{proof}

We would like to differentiate functions between definable subsets with $\RV$-coordinates. The procedure for this is the same, with or without a section $\sn$ (or a cross-section $\csn$) and additional structure in the $\RV$-sort, as described after \cite[Corollary~9.9]{Yin:special:trans}. It follows from Lemma~\ref{partial:exist:sec} and compactness that every partial derivative of $f$ is defined almost everywhere. The \emph{Jacobian of $f$ at a point $(\lbar a, \lbar t)$} is defined in the usual way, that is, the determinant of the Jacobian matrix, and is denoted by $\jcb_{\VF} f(\lbar a, \lbar t)$. By the chain rule, we have:

\begin{lem}\label{jcb:chain}
Let $f : A \fun B$ and $g : B \fun C$ be definable functions. Then for any $\lbar x \in A$,
\[
\jcb_{\VF} ( g \circ f)(\lbar x) = \jcb_{\VF} g(f(\lbar x)) \cdot \jcb_{\VF} f(\lbar x),
\]
if both sides are defined.
\end{lem}

\begin{defn}[Coarse $\VF$-categories]\label{defn:c:VF:cat}
The objects of the category $\VF[k, \cdot]$ are the definable subsets of $\VF$-dimension $\leq k$. Its morphisms are the definable bijections between the objects. Set $\VF_*[\cdot] = \bigcup_k \VF[k, \cdot]$.

An object of the category $\mgVF[k]$ is a definable pair $(A, \omega_{\Gamma})$, where $\pvf(A) \sub \VF^k$ and $\omega_{\Gamma} : A \fun \Gamma$ is a function, which is understood as a \emph{$\Gamma$-volume form} on $A$. A \emph{morphism} between two objects $(A, \omega_{\Gamma})$, $(B, \sigma_{\Gamma})$ is a definable \emph{essential bijection} $F : A \fun B$, that is, a bijection that is defined outside of definable subsets of $A$, $B$ of $\VF$-dimension $< k$, such that, for every $\lbar x \in \dom(F)$,
\[
\omega_{\Gamma}(\lbar x) = \sigma_{\Gamma}(F(\lbar x)) + \vv(\jcb_{\VF} F(\lbar x)).
\]
We also say that such an $F$ is a \emph{$\Gamma$-measure-preserving map}. Set $\mgVF[*] = \coprod_k \mgVF[k]$.
\end{defn}

Recall from~\cite[Remark~10.2]{Yin:special:trans} that conceptually $\mgVF[k]$-morphisms (and $\mVF[k]$-morphisms below) should be treated as equivalence classes so that each of them is actually an isomorphism and the Grothendieck semigroup may be constructed in the traditional way. However, this viewpoint is not essential for our purpose and, as usual, it is less cumbersome to work with representatives.

In order to avoid verbosity, below we shall more or less ignore the coarse $\VF$- and $\RV$-categories with $\Gamma$-volume forms, since the results may be modified in the obvious way to hold for them.

For any $\mathbf{U} \in \RV[k, \cdot]$, the \emph{lift} $\bb L \mathbf{U} \in \VF[k, \cdot]$ of $\mathbf{U}$ is defined following \cite[Definition~4.18]{Yin:special:trans}. For any $\RV[k, \cdot]$-morphism $F : \mathbf{U} \fun \mathbf{V}$, a \emph{lift} $F^{\uparrow} : \bb L \mathbf{U} \fun \bb L \mathbf{V}$ of $F$ is defined following \cite[Definition~7.3]{Yin:special:trans}. With the presence of $\sn$, such an $F$ can always be lifted.

\begin{prop}\label{L:sur:c}
The lifting map $\bb L : \ob \RV[\leq k, \cdot] \fun \ob \VF[k, \cdot]$ induces a surjective homomorphism, also denoted by $\bb L$, between the Grothendieck semigroups $\bb L : \gsk \RV[\leq k, \cdot] \fun \gsk \VF[k, \cdot]$.
\end{prop}
\begin{proof}
Applying \cite[Corollary~5.6]{Yin:special:trans} piecewise over $\RV$-partition, it is clear that $\bb L$ hits every isomorphism class of $\VF[k, \cdot]$ (see also the discussion after \cite[Proposition~6.18]{Yin:int:acvf}). Due to the presence of a section $\sn$ and its immediate consequence that each $\rv$-ball has a prescribed center, the work in \cite[\S 6]{hrushovski:kazhdan:integration:vf} (as well as \cite[\S 7]{Yin:special:trans}) is not needed here, although it is needed below, and it is almost trivial that every isomorphism class of $\coprod_{i \leq k} \RV[i, \cdot]$ is mapped into an isomorphism class of $\VF[k, \cdot]$.
\end{proof}

\begin{defn}
For a definable subset $A \sub \VF^n \times \RV^m$, the \emph{$\RV$-fiber dimension} of $A$, written as $\dim_{\RV}^{\fib}(A)$, is the number $\max \{\dim_{\RV}(\fib(A, \lbar a)) : \lbar a \in \pvf(A)\}$.
\end{defn}

\begin{defn}[Fine $\VF$-categories]\label{defn:f:VF:cat}
The objects of the category $\VF[k]$ are the $\mdl L_{\wt{\bb T}}$-definable subsets of $\VF$-dimension $\leq k$ and $\RV$-fiber dimension $0$. Its morphisms are the $\mdl L_{\wt{\bb T}}$-definable bijections between the objects. Set $\VF_* = \bigcup_k \VF[k]$.

An object of the category $\mVF[k]$ is an $\mdl L_{\wt{\bb T}}$-definable pair $(A, \omega)$, where $\pvf(A) \sub \VF^k$, $A \in \VF[k]$, and $\omega : A \fun \K^{\times} \times \Gamma$ is a function, which is understood as a \emph{volume form} on $A$. We also write $\omega$ as a pair $(\omega_{\K}, \omega_{\Gamma})$. A $\mgVF[k]$-morphism $F : (A, \omega_{\Gamma}) \fun (A', \omega'_{\Gamma})$ is a \emph{pseudo-morphism} of $\mVF[k]$. If, in addition, $F$ is $\mdl L_{\wt{\bb T}}$-definable and, for every $\lbar x \in \dom(F)$,
\[
\omega_{\K}(\lbar x) = \sigma_{\K}(F(\lbar x)) \cdot (\tbk \circ \rv)(\jcb_{\VF}F(\lbar x))
\]
then $F$ is a \emph{morphism} of $\mVF[k]$. We also say that such an $F$ is a \emph{measure-preserving map}. Set $\mVF[*] = \coprod_k \mVF[k]$.
\end{defn}

For $\mathbf{A} = (A, \omega) \in \mVF[k]$, we shall sometimes write $\dim_{\VF}(\mathbf{A})$ for $\dim_{\VF}(A)$. Let $F : \mathbf{A} \fun \mathbf{B}$ be a $\mVF[k]$-morphism. For any $\Gamma$-partition $\pi$ of $F$ and every $\lbar \gamma \in \ran(\pi)$, $\pi^{-1}(\lbar \gamma) = F_{\lbar \gamma}$ is an $\csn(\lbar \gamma)$-$\lan{RV}$-definable $\mVF[k]$-morphism between the objects $\dom(F_{\lbar \gamma}) \sub \mathbf{A}$ and $\ran(F_{\lbar \gamma}) \sub \mathbf{B}$, which is also a morphism in the sense of \cite[Definition~10.1]{Yin:special:trans}. If $F$ is not a trivial morphism, that is, if $\dim_{\VF}(\mathbf{A}) = k$, then, by Lemma~\ref{dim:comp:fiberwise}, some $F_{\lbar \gamma}$ is a nontrivial morphism.

\begin{rem}\label{lift:V:iso:R:iso}
For any $(\mathbf{U}, \omega) \in \mRV[k]$, the \emph{lift} $\bb L(\mathbf{U}, \omega) = (\bb L \mathbf{U}, \bb L \omega) \in \mVF[k]$ of $(\mathbf{U}, \omega)$ is defined in the obvious way. Let $F : \mathbf{U} \fun \mathbf{V}$ be an $\RV[k]$-morphism and $F^{\uparrow}$ a lift of $F$. If $F^{\uparrow}$ is a $\mVF[k]$-morphism between $(\bb L \mathbf{U}, \bb L \omega)$ and $(\bb L \mathbf{V}, \bb L \sigma)$ then, under Hypothesis~\ref{hyp:jac}, we see that the proof of \cite[Lemma~9.15]{Yin:special:trans} may be easily adapted to show that $F$ is indeed a $\mRV[k]$-morphism between $(\mathbf{U}, \omega)$ and $(\mathbf{V}, \sigma)$.
\end{rem}

%


\begin{prop}\label{L:measure:surjective:dag}
Every $\mathbf{A} \in \mVF[k]$ is isomorphic to another object $\bb L \mathbf{U} $ of $\mVF[k]$, where $\mathbf{U} \in \mRV[k]$.
\end{prop}
\begin{proof}
Let $\pi$ be a $\Gamma$-partition of $\mathbf{A}$. By \cite[Theorem~10.4]{Yin:special:trans}, there is an $\csn(\lbar \gamma)$-$\lan{RV}$-definable $\mVF[k]$-isomorphism between $\pi^{-1}(\lbar \gamma)$ and an object $\bb L \mathbf{U}_{\lbar \gamma} \in \mVF[k]$, where $\mathbf{U}_{\lbar \gamma} \in \mRV[k]$. By Lemma~\ref{dim:comp:fiberwise} and compactness, these isomorphisms may be glued together to form one isomorphism in $\mVF[k]$.
\end{proof}

\begin{prop}\label{L:semi:gr:dag}
Let $F : (\mathbf{U}, \omega) \fun (\mathbf{U}', \omega')$ be an $\mRV[k]$-isomorphism. Then there exists a measure-preserving lift $F^{\uparrow} : \bb L(\mathbf{U}, \omega) \fun \bb L (\mathbf{U}', \omega')$ of $F$.
\end{prop}
\begin{proof}
Let $\pi$ be a $\Gamma$-partition of $F$. Every $\pi^{-1}(\lbar \gamma) = F_{\lbar \gamma}$ may be treated as a morphism as defined in \cite[Definition~10.3]{Yin:special:trans}. So the assertion follows from \cite[Theorem~10.5]{Yin:special:trans} and compactness.
\end{proof}

\begin{cor}\label{L:sur}
The lifting map $\bb L : \ob \mRV[k] \fun \ob \mVF[k]$ induces a surjective homomorphism, also denoted by $\bb L$, between the Grothendieck semigroups $\bb L : \gsk \mRV[k] \fun \gsk \mVF[k]$.
\end{cor}

The inverse of $\bb L$, denoted by $\int_+ : \gsk \mVF[k] \fun \gsk \mRV[k] / \ker(\bb L)$, where $\ker (\bb L)$ is the kernel of $\bb L$, is an isomorphism of semigroups and is in effect the integration we are after. However, to understand the isomorphism $\int_+$ better and to apply it effectively in the future, we need a concrete description of $\ker (\bb L)$. To obtain that, as in~\cite{hrushovski:kazhdan:integration:vf, Yin:int:acvf}, the notion of special bijections in $\VF$-categories plays a key role.

Below we shall refer to a special bijection as defined in~\cite[Defintion~5.1]{Yin:special:trans} as an $\lan{RV}$-definable special bijection.

\begin{defn}
Let $A \sub \VF^n \times \RV^m$ be a definable subset. A bijection $T : A \fun A^{\sharp}$ is a \emph{special bijection on $A$ of length $1$} if for each $\rv$-polydisc $\gp \sub \RVH(A)$ there is an $\sn(\rv(\gp))$-$\lan{RV}$-definable special bijection $T_{\gp}$ on $\gp$ of length at most $1$ such that $T \rest (\gp \cap A) = T_{\gp} \rest (\gp \cap A)$ (all such $T_{\gp}$ of length $1$ target the same $\VF$-coordinate). The subset $C \sub \RVH(A)$ that contains exactly those $\rv$-polydiscs $\gp$ such that $T_{\gp}$ is of length $1$ is called the \emph{locus} of $T$. For each $\rv$-polydisc $\gp \sub C$ let $\lambda_{\gp}$ be the focus map of $T_{\gp} \rest (\gp \cap A)$. The function $\lambda = \bigcup_{\gp} \lambda_{\gp}$ is called the \emph{focus map} of $T$.

Naturally a special bijection $T$ on $A$ of \emph{length $n$}, denoted by $\lh (T) = n$, is a composition of $n$ special bijections $T_i$ of length $1$. Each $T_i$ is a \emph{component} of $T$.

These notions may be formulated in the same way if we work in $\gC^{\wt{\bb T}}$. Of course, in that case, the section $\sn$ is replaced by the cross-section $\csn$ and everything is $\mdl L_{\wt{\bb T}}$-definable.
\end{defn}

\begin{rem}\label{spec:rest:acvf}
Let $A \sub \VF^n \times \RV^m$ and $T : A \fun A^{\sharp}$ be a special bijection with components $T_i$. Clearly if $A$ is an $\RV$-pullback then $A^{\sharp}$ is an $\RV$-pullback. By definition, each $T_i$ is a restriction of a special bijection $R_i$ on $\RVH(\dom(T_i))$ and hence their composition $R$ is a special bijection on $\RVH(A)$. For any $\rv$-polydisc $\gp \sub \RVH(A^{\sharp})$ and any $\sn(\rv(\gp))$-$\lan{RV}$-definable (resp.\ $\csn(\rv(\gp))$-$\lan{RV}$-definable) subset $B \sub \gp$, the restriction $R \rest R^{-1}(B)$ is an $\sn(\rv(\gp))$-$\lan{RV}$-definable (resp.\ $\csn(\rv(\gp))$-$\lan{RV}$-definable) special bijection.
\end{rem}

\begin{lem}\label{special:tran:vol:pre}
For any special bijection $T : A \fun A^{\sharp}$ of length $1$, the Jacobians of $T$ and $T^{-1}$ are equal to $1$ almost everywhere. If $A$ is a nondegenerate $\RV$-pullback then they are equal to $1$ everywhere.
\end{lem}
\begin{proof}
This is immediate by~\cite[Lemma~9.11]{Yin:special:trans}.
\end{proof}

\begin{rem}
Many results below hold in both $\gC^{\bb T}$ and $\gC^{\wt{\bb T}}$ and the proofs are essentially identical if the section $\sn$ and the cross-section $\csn$ are interchanged everywhere. We shall quote them in both versions. However, to avoid repetition, whenever this is the case we shall only present the version for $\wt{\bb T}$ and leave the other one for the reader. In particular, we shall work in $\gC^{\wt{\bb T}}$ in the rest of this section.
\end{rem}

We can easily generalize \cite[Theorem~5.4]{Yin:special:trans} if the terms in question do not contain any $\RV$-sort variables:

\begin{lem}\label{spec:bi:poly:cons:noRV}
Let $\tau(\lbar X): \VF^n \fun \VF$ be an $\mdl L_{\wt{\bb T}}$-term, $\lbar t \in \RV^n$, $R : \rv^{-1}(\lbar t) \fun A$ a special bijection, and $f = \tau \circ R^{-1}$. Then there is a special bijection $T$ on $A$ such that the function $f \circ T^{-1}$ is contractible.
\end{lem}
\begin{proof}
First observe that if the assertion holds for one such term $\tau$ then it holds simultaneously for any finite number of such terms. Let $F_{ki}(\lbar X)$ enumerate all the occurring $\VF$-terms of $\tau$ such that $| F_{ki}(\lbar X) | = k$. By compactness, it is enough to concentrate on one $\rv$-polydisc $\gp_0 \sub A$. By Remark~\ref{spec:rest:acvf} and \cite[Theorem~5.5]{Yin:special:trans}, there is an $\rv(\gp_0)$-$\lan{RV}$-definable special bijection $T_0$ on $\gp_0$ such that, for every $\rv$-polydisc $\gq \sub T_0(\gp_0)$,
\[
(\rv \circ F_{0i} \circ R^{-1} \circ T_0^{-1})(\gq)
\]
is a singleton $\{ s^{\gq}_{0i} \}$ for all $i$.

By compactness again, it is enough to concentrate on one $\rv$-polydisc $\gp_1 \sub T_0(\gp_0)$. Let $F^{\gp_1}_{1i}(\lbar X)$ be the $\lan{RV}$-term obtained from $F_{1i}(\lbar X)$ by replacing each $\rv(F_{0i}(\lbar X))$ with $s^{\gp_1}_{0i}$. Each $F^{\gp_1}_{1i}(\lbar X)$ may be written as a polynomial $\sum_j a_j \lbar X^j$ with $a_j \in \dcl^{\wt{\bb T}}(\rv(\gp_1))$. By Remark~\ref{spec:rest:acvf} and \cite[Theorem~5.5]{Yin:special:trans} again, there is a $\dcl^{\wt{\bb T}}(\rv(\gp_1))$-definable special bijection $T_1$ on $\gp_1$ such that, for every $\rv$-polydisc $\gq \sub T_1(\gp_1)$,
\[
(\rv \circ F^{\gp_1}_{1i} \circ R^{-1} \circ T_0^{-1} \circ T_1^{-1})(\gq)
\]
is a singleton $\{ s^{\gq}_{1i} \}$ for all $i$.

Repeating this procedure for all $F_{ki}(\lbar X)$ of higher complexity, we see that there is a special bijection $T$ on $A$ as desired.
\end{proof}

The following lemma should be viewed as a joint generalization of~\cite[Lemma~4.10, Lemma~4.12]{Yin:QE:ACVF:min}.

\begin{lem}\label{VF:image:RV:inj}
Let $B \sub \rv^{-1}(\lbar t) \times \RV^m$ be a definable subset such that $\prv \rest B$ is finite-to-one. Then there is a special bijection $T$ on $\rv^{-1}(\lbar t)$ such that $\pvf(T(\pvf(B))) = \{\lbar 0\}$, that is, $T(\pvf(B))$ is a union of $\rv$-polydiscs of the form $(\lbar 0, \lbar \infty, \lbar s)$. Consequently, there is a definable injection $\pvf(B) \fun \RV^l$ for some $l$.
\end{lem}
\begin{proof}
Let $\phi(\lbar X)$ be a quantifier-free formula that defines $\pvf(B)$. Let $F_i(\lbar X)$ enumerate all the top occurring $\VF$-terms of $\phi(\lbar X)$. By Lemma~\ref{spec:bi:poly:cons:noRV}, there is a special bijection $T$ on $\rv^{-1}(\lbar t)$ such that every function $F_i \circ T^{-1}$ is contractible. Therefore, for every $\rv$-polydisc $\gp \sub T(\rv^{-1}(\lbar t))$, either $T^{-1}(\gp) \sub \pvf(B)$ or $T^{-1}(\gp) \cap \pvf(B) = \0$. Since $\acl^{\wt{\bb T}}(\RV) \models \wt{\bb T}(S)$, we see that if $T^{-1}(\gp) \sub \pvf(B)$ then $\gp$ must be a point, that is, $\gp$ must be of the form $(\lbar 0, \lbar \infty, \lbar s)$.
\end{proof}

\begin{lem}\label{two:spec:contra:condi}
Let $T$, $R$ be two special bijections on $\rv^{-1}(t)$. Let $R_i$ be the components of $R$, $\lambda_i$ the focus map of $R_i$, and $A_i \sub \rv^{-1}(t)$ the image of (the graph of) $\lambda_i$ under $\wh{R}^{-1}_i$, where $\wh{R}_i = R_i \circ \cdots \circ R_1$. If $\pvf(T(A_i)) = \{0\}$ for all $i$ then $R \circ T^{-1}$ is contractible.
\end{lem}
\begin{proof}
Suppose for contradiction that $R \circ T^{-1}$ is not contractible. Then there is an $\rv$-polydisc $\gp \sub T(\rv^{-1}(t))$ such that $(R \circ T^{-1})(\gp)$ is a union of more than one $\rv$-polydiscs. It is clear that there is an $i$ such that $(\wh{R}_i \circ T^{-1})(\gp)$ is contained in one $\rv$-polydisc and $\lambda_{i+1} \cap (\wh{R}_i \circ T^{-1})(\gp) \neq \0$. Then $\pvf(T(A_{i+1})) \neq \{0\}$, contradiction.
\end{proof}

We are now ready to state a better generalization of~\cite[Theorem~5.4]{Yin:special:trans}:

\begin{thm}\label{spec:bi:term:cons:disc}
Let $\tau(\lbar X, \lbar Y) : \VF^n \times \RV^m \fun \VF$ be an $\mdl L_{\wt{\bb T}}$-term. For each $\lbar s \in \RV^m$ let $\tau_{\lbar s} = \tau(\lbar X, \lbar s)$. Let $\lbar t \in \RV^n$, $R : \rv^{-1}(\lbar t) \fun A$ a special bijection, and $f_{\lbar s} = \tau_{\lbar s} \circ R^{-1}$. Then there is a special bijection $T$ on $A$ such that every function $f_{\lbar s} \circ T^{-1}$ is contractible.
\end{thm}
\begin{proof}
As Lemma~\ref{spec:bi:poly:cons:noRV}, if the assertion holds for one such term $\tau$ then it holds simultaneously for any finite number of such terms. We do induction on $n$. In a way the proof here combines those of Lemma~\ref{spec:bi:poly:cons:noRV} and \cite[Theorem~5.4]{Yin:special:trans}. The inductive step below is copied almost verbatim from the proof of~\cite[Theorem~5.4]{Yin:special:trans}.

For the base case $n = 1$, we simply write $X$ for $\lbar X$. Let $F_{ji}(X, \lbar Y)$ enumerate the occurring $\VF$-terms of $\tau$ such that $| F_{ji}(X, \lbar Y) | = j$. Let $\gp_0$, $T_0$, $s^{\gq}_{0i}$, and $\gp_1$ be as in the proof of Lemma~\ref{spec:bi:poly:cons:noRV}. For each $\lbar s \in \RV^m$ let $F^{\gp_1}_{1i, \lbar s}(X)$ be the $\lan{RV}$-term obtained from $F_{1i}(X, \lbar s)$ by replacing each $\rv(F_{0i}(\lbar X))$ with $s^{\gp_1}_{0i}$. Each $F^{\gp_1}_{1i, \lbar s}(X)$ may be written as a polynomial $\sum_j a_j X^j$ with $a_j \in \dcl^{\wt{\bb T}}(\rv(\gp_1), \lbar s)$. As in the proof of Lemma~\ref{spec:bi:poly:cons:noRV} again, there is a $\dcl^{\wt{\bb T}}(\rv(\gp_1), \lbar s)$-$\lan{RV}$-definable special bijection $T_{1, \lbar s}$ on $\gp_1$ such that, for every $\rv$-polydisc $\gq \sub T_{1, \lbar s}(\gp_1)$,
\[
(\rv \circ F^{\gp_1}_{1i, \lbar s} \circ R^{-1} \circ T_0^{-1} \circ T_{1, \lbar s}^{-1})(\gq)
\]
is a singleton for all $i$. Let $\lambda_{1, \lbar s, k}$ be the focus maps of the components of $T_{1, \lbar s}$ and $h_{\lbar s} :\biguplus_k \dom(\lambda_{1, \lbar s, k}) \times \{ \lbar s \} \fun \gp_1$ the injection induced by $T_{1, \lbar s}$. By compactness, $\bigcup_{\lbar s} h_{\lbar s}$ is an $\rv(\gp_1)$-definable injection into $\gp_1$. By Lemma~\ref{VF:image:RV:inj}, there is a special bijection $T_1$ on $\gp_1$ such that $\pvf(T_1(\ran(h))) = \{0\}$. By Lemma~\ref{two:spec:contra:condi}, every function $T_{1, \lbar s} \circ T_1^{-1}$ is contractible. This means that for every $\rv$-polydisc $\gr \sub T_1(\gp_1)$ and every $\lbar s \in \RV^m$ there is an $\rv$-polydisc $\gq_{\gr} \sub T_{1, \lbar s}(\gp_1)$ such that $T_1^{-1}(\gr) \sub T_{1, \lbar s}^{-1}(\gq_{\gr})$ and hence
\[
(\rv \circ F^{\gp_1}_{1i, \lbar s} \circ R^{-1} \circ T_0^{-1} \circ T_1^{-1})(\gr)
\]
is a singleton $\{ s^{\gr}_{1i, \lbar s} \}$ for all $i$. Repeating this procedure for all $F_{ki}(X, \lbar Y)$ of higher complexity, we see that there is a special bijection $T$ on $A$ as desired. This completes the base case of the induction.

We now proceed to the inductive step. As above, we may concentrate on one $\rv$-polydisc $\gp = \rv^{-1}(\lbar u) \times \set{(\lbar u, \lbar r)} \sub A$. Let $\phi(\lbar X, \lbar Y, Z)$ be a quantifier-free formula such that $\phi(\lbar X, \lbar s, Z)$ defines the function $(\rv \circ f_{\lbar s}) \rest \gp$. Let $F_{i}(\lbar X, \lbar Y, Z)$ enumerate the top occurring $\VF$-terms of $\phi$. For every $a \in \rv^{-1}(u_1)$ and every $\lbar s \in \RV^{m+1}$ let
\[
F_{i, \lbar s} = F_{i}(\lbar X, \lbar s), \quad F_{i,a, \lbar s} = F_{i}(a, X_2, \ldots, X_n, \lbar s).
\]
By the inductive hypothesis, there is a special bijection $R_{a}$ on $\rv^{-1}(u_2, \ldots, u_n)$ such that every function $F_{i,a, \lbar s} \circ R_a^{-1}$ is contractible. Let $U_{k, a}$ enumerate the loci of the components of $R_{a}$ and $\lambda_{k, a}$ the corresponding focus maps. By compactness,
\begin{enumerate}
  \item for each $i$ there is a quantifier-free formula $\psi_i$ such that $\psi_i(a, \lbar s)$ defines the contraction of $F_{i,a, \lbar s} \circ R_a^{-1}$,
  \item there is a quantifier-free formula $\theta$ such that $\theta(a)$ determines the sequence $\rv(U_{k, a})$ and the $\VF$-coordinates targeted by $\lambda_{k, a}$.
\end{enumerate}
Let $H_{j}(X_1)$ enumerate the top occurring $\VF$-terms of the formulas $\psi_i$, $\theta$. For every tuple $\lbar t\in \RV$ of the right length, let $H_{j, \lbar t} = H_{j}(X_1, \lbar t)$. Applying the inductive hypothesis again, we obtain a special bijection $T_1$ on $\rv^{-1}(u_1)$ such that every function $H_{j, \lbar t} \circ T_1^{-1}$ is contractible. This means that, for every $\rv$-polydisc $\gq \sub T_1(\rv^{-1}(u_1))$ and every $a_1, a_2 \in T_1^{-1}(\gq)$,
\begin{enumerate}
  \item for every $\lbar s \in \RV^{m+1}$, the formulas $\psi_i(a_1, \lbar s)$, $\psi_i(a_2, \lbar s)$ define the same function,
  \item the special bijections $R_{a_1}$, $R_{a_2}$ may be naturally glued together to form one special bijection on $\{a_1, a_2\} \times \rv^{-1}(u_2, \ldots, u_n)$.
\end{enumerate}
Consequently, $T_1$ and $R_{a}$ naturally induce a special bijection $T$ on $\gp$ such that each function $F_{i, \lbar s} \circ T^{-1}$ is contractible. This implies that each function $f_{\lbar s} \circ T^{-1}$ is contractible and hence $T$ is as required.
\end{proof}

We immediately give a slightly more general version of Theorem~\ref{spec:bi:term:cons:disc}, which is easier to use:

\begin{thm}\label{special:bi:term:constant}
Let $A \sub \VF^n$ and $f : A \fun \RV^m$ be a definable function. Then there is a special bijection $T$ on $A$ such that $T(A)$ is an $\RV$-pullback and the function $f \circ T^{-1}$ is contractible.
\end{thm}
\begin{proof}
By compactness, we may assume that $A$ is contained in an $\rv$-polydisc $\gp$. Let $\phi$ be a quantifier-free formula that defines $f$. Let $F_i(\lbar X, \lbar Y)$ enumerate the top occurring $\VF$-terms of $\phi$. For $\lbar s \in \RV^{m}$ let $F_{i, \lbar s} = F_{i}(\lbar X, \lbar s)$. By Theorem~\ref{spec:bi:term:cons:disc} there is a special bijection $T$ on $\gp$ such that each function $F_{i, \lbar s} \circ T^{-1}$ is contractible. This means that, for each $\rv$-polydisc $\gq \sub T(\gp)$,
\begin{enumerate}
  \item either $T^{-1}(\gq) \sub A$ or $T^{-1}(\gq) \cap A = \0$,
  \item if $T^{-1}(\gq) \sub A$ then $(f \circ T^{-1})(\gq)$ is a singleton.
\end{enumerate}
So $T \rest A$ is as required.
\end{proof}

Recall that a subset $A$ is called a \emph{deformed $\RV$-pullback} if there is a special bijection $T$ such that $T(A)$ is an $\RV$-pullback. By Theorem~\ref{special:bi:term:constant} and compactness, we have:

\begin{cor}\label{all:subsets:rvproduct}
Every definable subset $A \sub \VF^n \times \RV^m$ is a deformed $\RV$-pullback.
\end{cor}

\begin{lem}\label{bijection:made:contractible}
Let $A \sub \VF^{n_1} \times \RV^{m_1}$, $B \sub \VF^{n_2} \times \RV^{m_2}$, and $f : A \fun B$ be a definable function. Then there exists a special bijection $T$ on $A$ such that $T(A)$ is an $\RV$-pullback and the function $f \circ T^{-1}$ is contractible.
\end{lem}
\begin{proof}
By compactness, we may assume that $A$ is contained in an $\rv$-polydisc. Then this is immediate by applying Theorem~\ref{special:bi:term:constant} to the function $\prv \circ f$ (recall that $\can(B)$ is substituted for $B$).
\end{proof}

Recall the definition of the open-to-open property (see \cite[Proposition~3.19]{Yin:int:acvf} for subsets of $\VF$ and \cite[Definition~3.20]{Yin:int:acvf} for the general case). It is obviously still true that for functions between subsets that have only one $\VF$-coordinate, composing with special bijections on the right and inverses of special bijections on the left preserves the open-to-open property.

\begin{lem}\label{simul:special:dim:1}
Let $A \sub \VF \times \RV^{m_1}$, $B \sub \VF \times \RV^{m_2}$, and $f : A \fun B$ be a definable bijection. Then there exist special bijections $T_A : A \fun A^{\sharp}$ and $T_B : B \fun B^{\sharp}$ such that $A^{\sharp}$, $B^{\sharp}$ are $\RV$-pullbacks and, in
the commutative diagram
\[
\bfig
  \square(0,0)/->`->`->`->/<600,400>[A`A^{\sharp}`B`B^{\sharp};
  T_A`f``T_B]
 \square(600,0)/->`->`->`->/<600,400>[A^{\sharp}`\rv(A^{\sharp})`B^{\sharp} `\rv(B^{\sharp});  \rv`f^{\sharp}`f^{\sharp}_{\downarrow}`\rv]
 \efig
\]
$f^{\sharp}_{\downarrow}$ is bijective and hence $f^{\sharp}$ is a lift of it.
\end{lem}
\begin{proof}
By Corollary~\ref{all:subsets:rvproduct} we may assume that $A$, $B$ are $\RV$-pullbacks. Let $\pi$ be a $\Gamma$-partition of $f$. For each $\lbar \gamma \in \ran(\pi)$, set $\dom(f_{\lbar \gamma}) = A_{\lbar \gamma}$ and $\ran(f_{\lbar \gamma}) = B_{\lbar \gamma}$. By~\cite[Proposition~3.21]{Yin:int:acvf}, there is a finite partition of $A_{\lbar \gamma}$ into definable subsets $A_{i, \lbar \gamma}$ such that each $f_{\lbar \gamma} \rest A_{i, \lbar \gamma}$ has the open-to-open property. For $(a, \lbar t) \in A_{i, \lbar \gamma}$ let $h(a, \lbar t) = (\rcsn(\lbar \gamma), i)$. Applying Lemma~\ref{bijection:made:contractible} to the function $h$, we obtain a special bijection $T$ on $A$ such that each $A_{i, \lbar \gamma}$ is an $\RV$-pullback. Applying it again to $f \circ T$, we may assume that $f$ is contractible and has the open-to-open property. In particular, for each $\rv$-polydisc $\gp \sub A$, $f(\gp)$ is an
open polydisc contained in an $\rv$-polydisc.

By Lemma~\ref{bijection:made:contractible} again, there is a special bijection $T_B: B \fun B^{\sharp}$ such that $(T_B \circ f)^{-1}$ is contractible. Let $T_B = T_{B, n} \circ \ldots \circ T_{B, 1}$. It is enough to construct a special bijection $T_A = T_{A, n} \circ \ldots \circ T_{A, 1}$ on $A$ such that, for each $i$, both $\wh{T}_{B, i} \circ f \circ (\wh{T}_{A, i})^{-1}$ and $\wh{T}_{A, i} \circ (T_B \circ f)^{-1}$ are contractible, where
\[
\wh{T}_{B, i} = T_{B, i} \circ \ldots \circ T_{B, 1},\quad
\wh{T}_{A, i} = T_{A, i} \circ \ldots \circ T_{A, 1}.
\]
Now we may simply use the construction in the proof of~\cite[Lemma~5.2]{Yin:int:acvf}, since it only depends on the contractibility and the open-to-open property of $f$.
\end{proof}

Recall from \cite[Definition~5.4]{Yin:int:acvf} the notion of a (special) relatively unary bijection.

\begin{lem}\label{bijection:partitioned:unary}
Let $A \sub \VF^{n} \times \RV^{m_1}$, $B \sub \VF^{n} \times
\RV^{m_2}$, and $f : A \fun B$ a definable bijection. Then there is a finite partition of $A$ into definable subsets $A_i$ such that each $f \rest A_i$ is a composition of definable relatively unary bijections.
\end{lem}
\begin{proof}
Since there are finitely many $\VF$-coordinates to choose from, this is immediate by applying \cite[Lemma~5.6]{Yin:int:acvf} over a $\Gamma$-partition of $f$ and then compactness.
\end{proof}

Let $A \sub \VF^n \times \RV^m$ be a definable subset and $\sigma$ a permutation of $I_n = \{1, \ldots, n\}$. We define a \emph{standard contraction} $\wh T_{\sigma}$ of $A$ exactly as in \cite[Definition~5.5]{Yin:int:acvf}. By Corollary~\ref{all:subsets:rvproduct}, there are abundant standard contractions of $A$ in stock.

\begin{lem}\label{subset:partitioned:2:unit:contracted}
Let $12$, $21$ denote the permutations of $I_2$. Let $A \sub \VF^2 \times \RV^m$ be a definable subset. Then there is a definable injection $f : A \fun \VF^2 \times \RV^l$ such that
\begin{enumerate}
  \item $f$ is unary relative to both coordinates,
  \item there are standard contractions $\wh T_{12}$, $\wh R_{21}$ of $f(A)$ such that $(\wh T_{12}(f(A)), \pr_{\leq 2})$, $(\wh R_{21}(f(A)), \pr_{\leq 2})$ are $\RV[2, \cdot]$-isomorphic and, if $\dim_{\RV}^{\fib}(A) = 0$, then they are $\RV[2]$-isomorphic.
\end{enumerate}
\end{lem}
\begin{proof}
Let $\pi$ be a $\Gamma$-partition of $A$. Since the bijection on $A$ given by $\lbar x \efun (\lbar x, \rcsn(\pi(\lbar x)))$ is obviously unary relative to both coordinates, it is easily seen that the assertion simply follows from ~\cite[Corollary~5.8]{Yin:int:acvf} and compactness.
\end{proof}

\section{The kernel of $\bb L$ and integration}\label{section:ker}

To understand the kernels of the semigroup homomorphisms $\bb L$ constructed above, we shall produce analogues of \cite[Proposition~6.17,  Proposition~7.8]{Yin:int:acvf}. The key notion is still that of a blowup. This is defined in almost exactly the same way as in~\cite[Definition~6.1, Definition~7.1]{Yin:int:acvf}.

We shall first work in $\gC^{\bb T}$ and discuss the coarse $\VF$- and $\RV$-categories. However, as mentioned above, we shall concentrate on the categories without $\Gamma$-volume forms and the auxiliary results will only be stated for them. For the categories with $\Gamma$-volume forms the proofs are very similar and the extra computational work involving $\Gamma$-volume forms is always straightforward.

\begin{defn}\label{defn:blowup:coa}
Suppose $k > 0$ and let $\mathbf{U} = (U, f) \in \RV[k, \cdot]$. An \emph{elementary blowup} of $\mathbf{U}$ is an object $\mathbf{U}^{\sharp} = (U^{\sharp}, f^{\sharp}) \in \RV[\leq k,\cdot]$ such that $U^{\sharp} = U \times \RV^{>1}$ and, for some $1 \leq j \leq k$ and any $(\lbar t, s) \in U^{\sharp}$,
\[
f^{\sharp}_{i}(\lbar t, s) = f_{i}(\lbar t) \text{ for } i \neq j, \quad
f^{\sharp}_{j}(\lbar t, s) = s f_{j}(\lbar t).
\]
Note that $\mathbf{U}^{\sharp}$ is an object in $\RV[\leq k,\cdot]$ (actually in $\RV[k-1,\cdot] \amalg \RV[k,\cdot]$ ) but in general not an object in $\RV[k,\cdot]$ because $f^{\sharp}_{j}(\lbar t, \infty) = \infty$.

Let $\mathbf{V} = (V, g) \in \RV[k, \cdot]$, $C \sub V$, and $\mathbf{C} = (C, g \rest C) \in \RV[k, \cdot]$. Let $F : \mathbf{U} \fun \mathbf{C}$ be an $\RV[k, \cdot]$-morphism. Then $\mathbf{U}^{\sharp} \uplus (V \mi C, g \rest (V \mi C))$ is a \emph{blowup of $(V, g)$ via $F$}, written as $\mathbf{V}^{\sharp}_F$. The subscript $F$ may be dropped in context if there is no danger of confusion. The object $\mathbf{C}$ (or the subset $C$) is called the \emph{locus} of the blowup $\mathbf{V}^{\sharp}_F$. A \emph{blowup of length $n$} is a composition of $n$ blowups.
\end{defn}

\begin{lem}\label{blowup:equi:class:coa}
Let $\mathbf{U}, \mathbf{V} \in \RV[\leq k, \cdot]$ and $\mathbf{U}^{\sharp}$, $\mathbf{V}^{\sharp}$ be two blowups. In $\gsk \RV[\leq k, \cdot]$, if $[\mathbf{U}] = [\mathbf{V}]$ then there are blowups $\mathbf{U}^{\sharp\sharp}$, $\mathbf{V}^{\sharp\sharp}$ of $\mathbf{U}^{\sharp}$, $\mathbf{V}^{\sharp}$ such that $[\mathbf{U}^{\sharp\sharp}] = [\mathbf{V}^{\sharp\sharp}]$. Therefore, if $[\mathbf{U}] = [\mathbf{U}']$, $[\mathbf{V}] = [\mathbf{V}']$ and there are isomorphic blowups of $\mathbf{U}$, $\mathbf{V}$ then there are isomorphic blowups of $\mathbf{U}'$, $\mathbf{V}'$.
\end{lem}
\begin{proof}
For the first assertion, the proof of \cite[Lemma~6.5]{Yin:int:acvf} works. The second assertion is a corollary.
\end{proof}

\begin{defn}
Let $\isp[k, \cdot]$ be the subclass of $\ob \RV[\leq k, \cdot] \times \ob \RV[\leq k, \cdot]$ of pairs $(\mathbf{U}, \mathbf{V})$ such that there exist isomorphic blowups $\mathbf{U}^{\sharp}$, $\mathbf{V}^{\sharp}$. Let $\isp[*, \cdot] = \bigcup_{k} \isp[k, \cdot]$.
\end{defn}

We will just write $\isp$ for all these classes if there is no danger of confusion. By Lemma~\ref{blowup:equi:class:coa}, $\isp$ may be regarded as a binary relation on isomorphism classes.

\begin{lem}\label{isp:congruence:vol}
$\isp[k, \cdot]$ is a semigroup congruence relation and $\isp[*, \cdot]$ is a semiring congruence relation.
\end{lem}
\begin{proof}
The proof of~\cite[Lemma~6.8]{Yin:int:acvf} works.
\end{proof}

Let $\mathbf{U}_i = (U_i, f_i) \in \RV[i, \cdot]$, $\mathbf{U} =  \coprod_{i \leq k} \mathbf{U}_i  \in \RV[\leq k, \cdot]$, and $T$ a special bijection on $\bb L \mathbf{U}$. We write $U_{i,T}$ for the subset $(\prv \circ T)(\mathbb{L} \mathbf{U}_i)$, $\mathbf{U}_{i,T}$ for the object $(U_{i,T}, \pr_{\leq i}) \in \RV[i, \cdot]$, and $\mathbf{U}_{T}$ for the object $\coprod_{i \leq k} \mathbf{U}_{i,T} \in \RV[\leq k, \cdot]$. Recall from \cite[Notation~2.37]{Yin:int:acvf} the shorthand $[U_{i,T}]_{\leq i}$ for $[(U_{i,T}, \pr_{\leq i})] \in \gsk \RV[i, \cdot]$.

\begin{lem}\label{special:to:blowup:coa}
The object $\mathbf{U}_T$ is isomorphic to a blowup of $\mathbf{U}$ of the same length as $T$.
\end{lem}
\begin{proof}
By induction on the length $\lh T$ of $T$ and Lemma~\ref{blowup:equi:class:coa}, this is immediately reduced to the case $\lh T = 1$. Then we may use the isomorphism constructed in the proof of~\cite[Lemma~6.9]{Yin:int:acvf}.\end{proof}

\begin{lem}\label{kernel:dim:1:coa}
Suppose that $[A] = [B]$ in $\VF[1, \cdot]$ and $\mathbf{U}, \mathbf{V} \in \RV[\leq 1, \cdot]$ are two standard contractions of $A$, $B$. Then $([\mathbf{U}], [\mathbf{V}]) \in \isp$.
\end{lem}
\begin{proof}
By Lemma~\ref{simul:special:dim:1}, there are special bijections $T$, $R$ on $\bb L \mathbf{U}$, $\bb L \mathbf{V}$ such that $\mathbf{U}_{T}$, $\mathbf{V}_{R}$ are isomorphic. So the assertion follows from Lemma~\ref{special:to:blowup:coa}.
\end{proof}

\begin{lem}\label{blowup:same:RV:coa}
Let $\mathbf{U}^{\sharp}$ be a blowup of $\mathbf{U} \in \RV[\leq k, \cdot]$ of length $l$. Then $\bb L \mathbf{U}^{\sharp}$ and $\bb L \mathbf{U}$ are isomorphic.
\end{lem}
\begin{proof}
By induction this is immediately reduced to the case $l=1$. Observe that, using the section $\sn$, the special bijection $T$ on $\bb L \mathbf{U}$ as described in the proof of \cite[Lemma~6.12]{Yin:int:acvf} can be (quite trivially) constructed.
\end{proof}

\begin{lem}\label{isp:VF:fiberwise:contract:isp:coa}
Let $A_1, A_2 \in \VF[k, \cdot]$ such that $\pvf (A_1) = \pvf (A_2) = A$. Suppose that there is a common subset $E$ of the indices of the $\RV$-coordinates of $A_1$, $A_2$ such that, for every $\lbar a \in A$,
\[
([\fib(A_1, \lbar a)]_{E}, [\fib(A_2, \lbar a)]_{E}) \in \isp.
\]
Let $\wh T_{\sigma}$, $\wh R_{\sigma}$ be two standard contractions of $A_1$, $A_2$. Set $E' = E \cup I_k$. Then
\[
([\wh T_{\sigma}(A_1)]_{E'}, [\wh R_{\sigma}(A_2)]_{E'}) \in \isp.
\]
\end{lem}
\begin{proof}
In the proof of~\cite[Lemma~6.14]{Yin:int:acvf}, the special bijection $T_A$ is achieved by applying \cite[Theorem~5.5]{Yin:special:trans} to the occurring polynomials of a suitable quantifier-free formula, as in \cite[Lemma~5.1]{Yin:int:acvf}. This procedure may be reproduced   here by applying Theorem \ref{special:bi:term:constant} to the top occurring $\VF$-terms of a suitable quantifier-free formula, as in Lemma~\ref{bijection:made:contractible}. For the rest of the proof, we may simply follow the proof of~\cite[Lemma~6.14]{Yin:int:acvf}.
\end{proof}

\begin{cor}\label{contraction:same:perm:isp:coa}
Let $A_1, A_2 \in \VF[k, \cdot]$ and $f : A_1 \fun A_2$ a unary bijection relative to the coordinate $i \in I_k$. Then for any permutation $\sigma$ of $I_k$ with $\sigma(1) = i$ and any standard contractions $\wh T_{\sigma}$, $\wh R_{\sigma}$ of $A_1$, $A_2$,
\[
([\wh T_{\sigma}(A_1)]_{\leq k}, [\wh R_{\sigma}(A_2)]_{\leq k}) \in \isp.
\]
\end{cor}
\begin{proof}
This is immediate by Lemma \ref{kernel:dim:1:coa} and Lemma \ref{isp:VF:fiberwise:contract:isp:coa}.
\end{proof}

\begin{lem}\label{contraction:perm:pair:isp:coa}
Let $A \in \VF[k, \cdot]$. Let $i, j \in I_k$ be distinct and $\sigma_1$, $\sigma_2$ two permutations of $I_k$ such that
\[
\sigma_1(1) = \sigma_2(2) = i, \quad \sigma_1(2) = \sigma_2(1) = j, \quad \sigma_1
\rest \set{3, \ldots, k} = \sigma_2 \rest \set{3, \ldots, k}.
\]
Then, for any standard contractions $\wh T_{\sigma_1}$, $\wh T_{\sigma_2}$ of $A$,
\[
([\wh T_{\sigma_1}(A)]_{\leq k}, [\wh T_{\sigma_2}(A)]_{\leq k}) \in \isp.
\]
\end{lem}
\begin{proof}
We have developed analogues of the results used in the proof of \cite[Lemma~6.16]{Yin:int:acvf}. Therefore its proof may be quoted here with virtually no changes.
\end{proof}

Now we have reproduced for $\VF[k, \cdot]$, $\RV[\leq k, \cdot]$ all the results that the proof of \cite[Proposition~6.17]{Yin:int:acvf} formally depends on, so the following crucial description of the kernel of $\bb L: \gsk \RV[\leq k, \cdot] \fun \gsk \VF[k, \cdot]$ may be obtained by more or less the same proof, which is reproduced in its entirety below.\footnote{In fact the wording of the proof of \cite[Proposition~6.17]{Yin:int:acvf} is somewhat terse and hence confusing. We take this opportunity to improve it.}

\begin{prop}\label{kernel:L:dag:coa}
For $\mathbf{U}, \mathbf{V} \in \RV[\leq k, \cdot]$, $[\bb L \mathbf{U}] = [\bb L \mathbf{V}]$ if and only if $([\mathbf{U}], [\mathbf{V}]) \in \isp$.
\end{prop}
\begin{proof}
The ``if'' direction simply follows from Lemma~\ref{blowup:same:RV:coa} and
Proposition~\ref{L:sur:c}.

For the ``only if'' direction, we show a stronger claim: if $[A] = [B]$ in $\VF[k, \cdot]$ and $\mathbf{U}, \mathbf{V} \in \RV[\leq k, \cdot]$ are two standard contractions of $A$, $B$ then $([\mathbf{U}], [\mathbf{V}]) \in \isp$. We do induction on $k$. The base case $k = 1$ is of course Lemma~\ref{kernel:dim:1:coa}. For the inductive step, suppose that $F : \bb L \mathbf{U} \fun \bb L \mathbf{V}$ is a definable bijection. By Lemma~\ref{bijection:partitioned:unary}, there is a partition of $\bb L \mathbf{U}$ into definable subsets $A_1, \ldots, A_n$  such that each $F_i = F \rest A_i$ is a composition of relatively unary bijections. Applying Theorem~\ref{special:bi:term:constant} as in Lemma~\ref{bijection:made:contractible}, we obtain special bijections
$T$, $R$ on $\bb L \mathbf{U}$, $\bb L \mathbf{V}$ such that $T(A_i)$, $(R \circ F)(A_i)$ are $\RV$-pullbacks for each $i$. By Lemma~\ref{special:to:blowup:coa}, it is enough to show that there are standard contractions $\wh T_{\sigma}$, $\wh R_{\tau}$ of $T(A_i)$, $(R \circ F)(A_i)$ for each $i$ such that
\[
([(\wh T_{\sigma} \circ T)(A_i)]_{\leq k}, [(\wh R_{\tau} \circ R \circ F)(A_i)]_{\leq k}) \in \isp.
\]
To that end, first note that each $(R \circ F \circ T^{-1}) \rest T(A_i)$ is a composition of relatively unary bijections, say
\[
T(A_i) = B_1 \to^{G_1} B_2 \cdots B_l \to^{G_l} B_{l+1} = (R \circ F)(A_i).
\]
For each $j \leq l - 2$ we may choose five contractions $[U_j]_{\leq k}$, $[U_{j+1}]_{\leq k}$,  $[U'_{j+1}]_{\leq k}$, $[U''_{j+1}]_{\leq k}$, and $[U_{j+2}]_{\leq k}$ with the permutations $\sigma_{j}$, $\sigma_{j+1}$, $\sigma'_{j+1}$, $\sigma''_{j+1}$, and $\sigma_{j+2}$ of $I_k$ such that
\begin{enumerate}
  \item $\sigma_{j+1}(1)$ and $\sigma_{j+1}(2)$ are the $\VF$-coordinates targeted by $G_{j}$ and $G_{j+1}$, respectively,
  \item $\sigma''_{j+1}(1)$ and $\sigma''_{j+1}(2)$ are the $\VF$-coordinates targeted by $G_{j+1}$ and $G_{j+2}$, respectively,
  \item $\sigma_{j} = \sigma_{j+1}$, $\sigma''_{j+1} =  \sigma_{j+2}$,  and $\sigma'_{j+1}(1) = \sigma''_{j+1}(1)$,
  \item the relation between $\sigma_{j+1}$ and $\sigma'_{j+1}$ is as described in Lemma~\ref{contraction:perm:pair:isp:coa}.
\end{enumerate}
Then, by Corollary~\ref{contraction:same:perm:isp:coa} and Lemma~\ref{contraction:perm:pair:isp:coa}, all the adjacent pairs of these contractions are $\isp$-congruent, except $([U'_{j+1}]_{\leq k}, [U''_{j+1}]_{\leq k})$.  Since, without loss of generality, we may assume that $[U'_{j+1}]_{\leq k}$ and $[U''_{j+1}]_{\leq k}$ start with the same contraction on the first targeted $\VF$-coordinate of $B_{j+1}$,  the resulting objects in $\VF[k-1, \cdot]$ are the same. So, by the inductive hypothesis, this last pair is also $\isp$-congruent. This completes the ``only if'' direction.
\end{proof}

We now move on to work in $\gC^{\wt{\bb T}}$ and discuss the fine $\VF$- and $\RV$-categories. The definition of a blowup needs to be slightly modified and the results in \cite[\S 7]{Yin:int:acvf} are needed. However, applying $\Gamma$-partitions and compactness, analogues of the results above may be obtained by essentially the same proofs.

\begin{defn}\label{defn:blowup:vol}
Suppose $k > 0$. Let $1 \leq j \leq k$ and $\mathbf{U} = (U, f) \in \RV[k]$. Suppose that there is a $\Gamma$-partition $\pi$ of (the graph of) $f$ such that
\[
\pi^{-1}(\lbar \gamma)_{j}(\lbar t) \in \acl(\pi^{-1}(\lbar \gamma)_{\wt j}(\lbar t), \rcsn(\lbar \gamma))
\]
for all $\lbar \gamma \in \ran(\pi)$ and all $\lbar t \in \dom(\pi^{-1}(\lbar \gamma))$. An \emph{elementary blowup} of $\mathbf{U}$ is an object $\mathbf{U}^{\sharp} = (U^{\sharp}, f^{\sharp}) \in \RV[k]$ such that $U^{\sharp} = U \times (\RV^{\times})^{>1}$ and, for any $(\lbar t, s) \in U^{\sharp}$,
\[
f^{\sharp}_{i}(\lbar t, s) = f_{i}(\lbar t) \text{ for } i \neq j, \quad
f^{\sharp}_{j}(\lbar t, s) = s f_{j}(\lbar t).
\]
Let $\omega$ be a volume form on $U$. An \emph{elementary blowup} of $(\mathbf{U}, \omega)$ is an object $(\mathbf{U}^{\sharp}, \omega^{\sharp}) \in \mRV[k]$, where $\mathbf{U}^{\sharp}$ is an elementary blowup of $\mathbf{U}$ and $\omega^{\sharp}$ is the volume form on $U^{\sharp}$ given by $\omega^{\sharp}(\lbar t, s) = \omega(\lbar t)$.

Other related notions are defined as in Definition~\ref{defn:blowup:coa}.
\end{defn}

\begin{lem}\label{elementary:blowups:preserves:iso:vol}
Let $\mathbf{U}, \mathbf{V} \in \mRV[k]$ and $\mathbf{U}^{\sharp}$, $\mathbf{V}^{\sharp}$ be two elementary blowups. If $[\mathbf{U}] = [\mathbf{V}]$ then $[\mathbf{U}^{\sharp}] = [\mathbf{V}^{\sharp}]$.
\end{lem}
\begin{proof}
This is immediate by applying \cite[Lemma~7.2]{Yin:int:acvf} over a $\Gamma$-partition of an isomorphism between $\mathbf{U}$ and $ \mathbf{V}$.
\end{proof}

\begin{defn}
Let $\misp[k]$ be the subclass of $\ob \mRV[k] \times \ob \mRV[k]$ of pairs $(\mathbf{U}, \mathbf{V})$ such that there exist isomorphic blowups $\mathbf{U}^{\sharp}$, $\mathbf{V}^{\sharp}$. Let $\misp[*] = \coprod_{k} \misp[k]$.
\end{defn}

\begin{lem}\label{isp:congruence:vol}
As a binary relation on isomorphism classes, $\misp[k]$ is a semigroup congruence relation and $\misp[*]$ is a semiring congruence relation.
\end{lem}

Let $(A, \omega) \in \mVF[k]$ and $T$ a special bijection on $A$. Set $\omega_T = \omega \circ T^{-1}$. Then $T$ is also understood as a special bijection from $(A, \omega)$ to $T(A, \omega) = (T(A), \omega_T)$. By Lemma~\ref{special:tran:vol:pre}, $T$ is indeed a $\mVF[k]$-isomorphism. Set $A_{\omega} = \set{(\lbar x, \omega(\lbar x)) : \lbar x \in A)}$. For simplicity the volume form on $A_{\omega}$ that is naturally induced by $\omega$ is still denoted by $\omega$. Clearly $(A, \omega)$ and $(A_{\omega}, \omega)$ are isomorphic. If $\wh T_{\sigma}$ is a standard contraction of $A_{\omega}$ then $\omega$ naturally induces a volume form $\omega_{\wh T_{\sigma}}$ on $(\wh T_{\sigma}(A_{\omega}), \pr_{\leq k})$. The function $\wh T_{\sigma}$ (or the object $(\wh T_{\sigma}(A_{\omega}), \pr_{\leq k}, \omega_{\wh T_{\sigma}}) \in \mRV[k]$, which is completely determined by $\wh T_{\sigma}$) is understood as a \emph{standard contraction} of $(A, \omega)$.

For $(\mathbf{U}, \omega) = (U, f, \omega) \in \mRV[k]$ and a special bijection $T$ on $\bb L (\mathbf{U}, \omega)$, we write $\omega_T$ for the volume form on $U_T$ that is naturally induced by $(\bb L \omega)_T$ and $(\mathbf{U}, \omega)_T$ for the object $(\mathbf{U}_T, \omega_T) \in \mRV[k]$.

It is straightforward to state and prove the analogues of the results from Lemma~\ref{special:to:blowup:coa} to Lemma~\ref{blowup:same:RV:coa} (note that Remark~\ref{lift:V:iso:R:iso} is needed for the analogues of Lemma~\ref{special:to:blowup:coa} and Lemma~\ref{kernel:dim:1:coa}). For the analogues of Lemma~\ref{isp:VF:fiberwise:contract:isp:coa} and Lemma~\ref{contraction:perm:pair:isp:coa}, the proofs of \cite[Lemma~7.5, Lemma~7.7]{Yin:int:acvf} can be easily adapted. From these we can deduce:

\begin{prop}\label{kernel:L:vol:dag}
For $\mathbf{U}, \mathbf{V} \in \mRV[k]$, $[\bb L \mathbf{U}] = [\bb L \mathbf{V}]$ if and only if $([\mathbf{U}], [\mathbf{V}]) \in \misp$.
\end{prop}


The dependence on \cite[Proposition~6.17, Proposition~7.8]{Yin:int:acvf} of the results concerning Grothendieck homomorphisms in \cite[\S 6, \S 7]{Yin:int:acvf} is of a formal nature. Therefore, using the results above, their analogues may be derived in more or less the same way. The (obvious) proofs are again omitted.

We emphasize here that the statements below concern two situations: in $\gC^{\bb T}$ without volumes (or with $\Gamma$-volume forms) and in $\gC^{\wt{\bb T}}$ with volume forms. It is a matter of restriction to transfer results from $\gC^{\bb T}$ to $\gC^{\wt{\bb T}}$. But it is not clear if we can successfully incorporate volume forms in the categories associated to $\gC^{\bb T}$. The difficulty is that if we simply work with an analogue of Definition~\ref{defn:f:VF:cat} then special bijections are not guaranteed to be morphisms, in particular, the inductive step of Theorem~\ref{spec:bi:term:cons:disc} seems to fail without an easy remedy.

\begin{thm}\label{main:prop:k:vol:dag}
For each $k \geq 0$ there are canonical isomorphisms of Grothendieck semigroups
\[
\int_{+} : \gsk  \VF[k, \cdot] \fun \gsk  \RV[\leq k, \cdot] /  \isp \quad \text{and} \quad \int_{+} \gsk  \mVF[k] \fun \gsk  \mRV[k] /  \misp
\]
such that
\[
\int_{+} [\mathbf{A}] = [\mathbf{U}]/  \isp \quad \text{if and only if} \quad  [\mathbf{A}] = [\bb L \mathbf{U}] \quad \text{and} \quad \int_{+} [\mathbf{A}] = [\mathbf{U}]/  \misp \quad \text{if and only if} \quad  [\mathbf{A}] = [\bb L \mathbf{U}].
\]
Putting these together, we obtain canonical isomorphisms of Grothendieck semirings
\[
\int_{+} : \gsk \VF_*[\cdot] \fun \gsk  \RV[*, \cdot] /  \isp \quad \text{and} \quad \int_{+} \gsk  \mVF[*] \fun \gsk  \mRV[*] /  \misp.
\]
\end{thm}

Recall \cite[Notation~3.16]{Yin:special:trans}. Let $A \sub \VF^n$ and $f : A \fun \mdl P(\RV^m)$ be a definable function such that every $f(\lbar a)$ codes an object in $\RV[\leq k, \cdot]_{\lbar a}$ (note that, by compactness, $k$ is bounded). We think of $f$, or rather the graph of $f$, as a representative of an equivalence class of definable functions induced by $\isp$ and the equivalence class as a \emph{definable function} $ A \fun \gsk  \RV[*, \cdot] /  \isp$, which, for simplicity, is also denoted by $f$. The set of all such functions, as in Definition~\ref{def:fn:ring}, is denoted by
\[
\fn(A, \gsk  \RV[*, \cdot] /  \isp) = \bigoplus_i \fn(A, \gsk  \RV[i, \cdot] /  \isp),
\]
which is a semimodule. Using Notation~\ref{nota:par:exp}, $f$ and $g$ represent the same function if $[f(\lbar a)] =_{\isp, \lbar a} [g(\lbar a)]$ for every $\lbar a \in A$. Let $\bb L f = \bigcup_{\lbar a \in A}\{\lbar a\} \times \bb L(f(\lbar a))$. Set
\[
\int_{+A} f = \int_{+} [f] = \int_+ [\bb L f],
\]
which, by Proposition~\ref{kernel:L:dag:coa} and compactness, does not depend on the representative $f$. Consequently we have a homomorphism of semimodules:
\[
\int_{+A} : \fn(A, \gsk  \RV[*, \cdot] /  \isp) \fun \gsk \RV[*, \cdot] /  \isp.
\]

Similarly, if each $f(\lbar a)$ codes an object in $\mRV[*]$ then $f$ represents a definable function, sometimes denoted by $(f, \omega)$, in the semimodule $\fn(A, \gsk  \mRV[*] /  \misp)$, where $\omega$ is the volume form on the graph of $f$, that is, $\omega \rest f(\lbar a)$ is the volume form carried in $f(\lbar a)$. Let $\bb L \omega$ be the volume form on $\bb L f$ naturally induced by $\omega$. Setting
\[
\int_{+A} (f, \omega) = \int_{+} [(f, \omega)] = \int_+ [(\bb L f, \bb L \omega)],
\]
we obtain a homomorphism of semimodules:
\[
\int_{+A} : \fn(A, \gsk  \mRV[*] /  \misp) \fun \gsk  \mRV[*] /  \misp.
\]

\begin{prop}\label{semi:fubini}
For any nonempty subsets $E_1, E_2 \sub I_n = \{1, \ldots, n\}$,
\[
\int_{+\lbar a \in \pr_{E_1}(A)} \int_{+ \fib(A, \lbar a)} f = \int_{+\lbar a \in \pr_{E_2}(A)} \int_{+ \fib(A, \lbar a)} f.
\]
Similarly for the case with volume forms.
\end{prop}
\begin{proof}
This is immediate by Proposition~\ref{kernel:L:dag:coa}, Proposition~\ref{kernel:L:vol:dag}, and the definition of iterated integrals.
\end{proof}

Let $B \sub \VF^n$ and assume that the $\VF$-dimensions of $A$, $B$ are $n$. For any definable bijection $\phi : A \fun B$, we can define the \emph{Jacobian transformation}
\[
\phi^{\jcb} :  \fn(A, \gsk  \mRV[*] /  \misp) \fun \fn(B, \gsk  \mRV[*] /  \misp)
\]
exactly as in~\cite[Section~7]{Yin:int:acvf} and, as~\cite[Proposition~7.12]{Yin:int:acvf}, obtain

\begin{prop}\label{semi:change:variables}
$\int_{+A} ( f,  \omega) = \int_{+B} \phi^{\jcb}( f,  \omega)$.
\end{prop}

Let $\mathbf{I}$, $\mu \mathbf{I}$ be the ideals of the groupifications of $\gsk  \RV[*, \cdot] /  \isp$, $\gsk  \mRV[*] /  \misp$. By the same calculations as in \cite[\S 6, \S 7]{Yin:int:acvf},  we see that, the ideal $\mathbf{I}$ is generated by $[1]_0 + \mathbf{j}$ and the ideal $\mu \mathbf{I}$ is generated by $\mathbf{j}_{\mu}$ (see Notation~\ref{nota:RV:ele}). Note that $\mathbf{j}$ is equal to $- [1]_0 = -1$ in $\ggk  \RV[*, \cdot] / \mathbf{I}$ and hence is not a zero-divisor in $\ggk  \RV[*, \cdot]$ (for otherwise $[1]_0 + \mathbf{j}$ would be a zero-divisor in $\ggk  \RV[*, \cdot]$, which is clearly impossible).

\begin{thm}
The Grothendieck semiring isomorphism $\int_+$ induces canonically an injective homomorphism
\[
\int :  \ggk \VF_*[\cdot] \fun \ggk  \RV[*, \cdot] [\mathbf{j}^{-1}],
\]
whose range is the entire zeroth graded piece, and two homomorphisms
\[
\Xint{\textup{R}}^g : \ggk \VF_*[\cdot] \fun  \ggk \RES[*, \cdot][\mathbf{A}^{-1}] \quad \text{and} \quad \Xint{\textup{R}}^b: \ggk \VF_*[\cdot] \fun \ggk \RES[*, \cdot][[1]^{-1}_1].
\]
\end{thm}
\begin{proof}
The first homomorphism is similar to the one in \cite[Theorem~6.22]{Yin:int:acvf}. The other two come from Proposition~\ref{prop:eu:retr}.
\end{proof}

Similarly, since $\mu \mathbf{I}$ is a homogeneous ideal, setting $\mu \mathbf{I}_k = \ggk \mRV[k] \cap \mu \bf I$, we have:

\begin{thm}\label{theorem:integration:1:vol:dag}
The Grothendieck semiring isomorphism $\int_{+}$ induces canonically a graded ring isomorphism
\[
\int \ggk  \mVF[*] \fun \ggk  \mRV[*] /  \mu \mathbf{I} = \bigoplus_{k} \ggk \mRV[k] / \mu \mathbf{I}_k
\]
and two graded ring homomorphisms
\[
\Xint{\textup{e}}^g : \ggk \mVF[*] \fun  \ggk \mRES[*] / (\mathbf{A}_{\mu}) \quad \text{and} \quad \Xint{\textup{e}}^b: \ggk \mVF[*] \fun \ggk \mRES[*] / ([1_{\mu}]_1).
\]
\end{thm}

In future applications, we will often need to modify the target Grothendieck (semi)rings of the integration maps through some standard algebraic manipulations. The following procedure is an example.

Let us abbreviate $\gsk \mRV[k] / \misp$, $\gsk  \mRV[*] /  \misp$ as $\gsk \RV_k$, $\gsk \RV$, respectively. Their groupifications are abbreviated accordingly. Clearly the groupification of the completion $\wh{\gsk \RV}$ of $\gsk \RV$ is the completion $\wh{\ggk \RV}$ of $\ggk \RV$. So there is a canonical semiring homomorphism from the former into the latter. Recall \cite[Notation~2.9]{Yin:special:trans}. For any $\lbar \gamma \in \Gamma$, we write $\bm o_{\lbar \gamma}$ and $\bm c_{\lbar \gamma}$ for the canonical images of the elements $\int [(\go(0, \lbar \gamma), (1, 0))]$ and $\int [(\gc(0, \lbar \gamma), (1, 0))]$ in $\wh{\ggk \RV}$. We localize $\wh{\ggk \RV}$ at $\bm o_{0}$, $\bm c_{0}$ and obtain the ring $\wh{\ggk \RV}[\bm o_{0}^{-1}, \bm c^{-1}_{0}]$, which is abbreviated as $\KRC$.

An (ind-)definable function $A \fun \wh{\gsk \RV}$ is a sequence $\bm f = ( f_i)_{i \in \N}$, where the $i$th component $f_i$ is a function in $\fn(A, \gsk \RV_i)$. In other words, the set of all such functions is given by
\[
\fn(A, \wh{\gsk \RV}) = \sum_i \fn(A, \gsk \RV_i).
\]
For $\bm r = (r_i)_{i \in \N} \in \wh{\gsk \RV}$, let $\bm r \cdot \bm{f}$ be the function in $\fn(A, \wh{\gsk \RV})$ such that its $k$th component $(\bm r \cdot \bm f)_k \in \fn(A, \gsk \RV_k)$ is given by $(\bm r \cdot \bm f)_k (\lbar a) = \sum_{i+j=k} r_i \cdot f_j(\lbar a)$. This operation turns $\fn(A, \wh{\gsk \RV})$ into a natural $\wh{\gsk \RV}$-semimodule. Now we may integrate $\bm f$ componentwise: $\int_{+A} \bm f = \sum_i \int_{+A} f_i$, where, since $\int_{+A}f_i \in \gsk \RV_{i+n}$, the first $n$ terms of the right-hand side are $0$. A simple computation shows that
\[
\int_{+A} : \fn(A, \wh{\gsk \RV}) \fun \wh{\gsk \RV}
\]
is indeed a homomorphism of semimodules. Using the canonical semiring homomorphism
\[
\wh{\gsk \RV} \fun \wh{\ggk \RV} \fun  \KRC,
\]
we define the set of (ind-)definable functions $A \fun \KRC$ by
\[
\fn(A, \KRC) = \fn(A, \wh{\gsk \RV}) \otimes_{\wh{\gsk \RV}} \KRC.
\]
Then $\int_{+A}$ induces a canonical homomorphism of $\KRC$-modules
\[
\int_A : \fn(A, \KRC) \fun \KRC.
\]
For any definable bijection $\phi : A \fun B \sub \VF^n$ we also have the (componentwise) Jacobian transformation:
\[
\phi^{\jcb} : \fn(A, \KRC) \fun  \fn(B, \KRC).
\]
By Proposition~\ref{semi:fubini} and Proposition~\ref{semi:change:variables}, for any $\bm f \in \fn(A, \KRC)$, we obtain the following theorems:

\begin{thm}[Fubini theorem]\label{fubini}
For any nonempty subsets $E_1, E_2 \sub I_n = \{1, \ldots, n\}$,
\[
\int_{\lbar a \in \pr_{E_1}(A)} \int_{\fib(A, \lbar a)} \bm f = \int_{\lbar a \in \pr_{E_2}(A)} \int_{\fib(A, \lbar a)} \bm f.
\]
\end{thm}

\begin{thm}[Change of variables]\label{change:variables}
$\int_{A} \bm f = \int_{B} \phi^{\jcb} (\bm f)$.
\end{thm}

We remind the reader that similar results are available for the coarse categories with $\Gamma$-volume forms (in $\gC^{\bb T}$).

\section{The uniform rationality of certain Igusa local zeta functions}

The integration theory developed so far is quite effective in showing uniform rationality of Igusa local zeta functions. In this last section we shall discuss such an application. The general idea is to specialize from the sufficiently saturated model to non-archimedean local fields. This is usually done in two steps: descent to an arbitrary henselian substructure and then specialization to all non-archimedean local fields of sufficiently large residue characteristic.

\begin{lem}\label{cut:to:hensel:substru}
Let $M \supseteq S$ be a henselian substructure of $\gC$. This means that $(\VF(M), \OO(M))$ is a nontrivially valued field and is henselian. If $M$ is a substructure of $\gC^1$ then $M = \dcl^{1}(M)$. If $M$ is a $\VF$-generated substructure of $\gC^2$ (resp.\ $\gC^3$) and $\Gamma(M)$ is divisible then $M = \dcl^{2}(M)$ (resp.\ $M = \dcl^{3}(M)$).
\end{lem}
\begin{proof}
From the proof of \cite[Theorem~3.14]{Yin:QE:ACVF:min} it is clear that any valued field automorphism of $\gC^1$ over $M$ is an $\lan{RV}^1$-automorphism of $\gC$ over $M$ and hence the proof of \cite[Lemma~6.2]{Yin:special:trans} may be easily adapted here. The other cases are similar.
\end{proof}

\begin{prop}
In all the three cases above, $M$ admits elimination of $\VF$-quantifiers.
\end{prop}
\begin{proof}
It suffices to reproduce~\cite[Lemma~6.3]{Yin:special:trans} for the current setting. The key of its proof is to apply \cite[Theorem~5.5]{Yin:special:trans} to the occurring polynomials in question and then apply \cite[Lemma~6.2]{Yin:special:trans}. To imitate this argument, we may obviously apply Theorem~\ref{special:bi:term:constant} to the top occurring $\VF$-terms in question and then apply Lemma~\ref{cut:to:hensel:substru}. The details are left to the reader.
\end{proof}

%
%

From now on we shall work in $\gC^2$ and assume that the substructure $S = \dcl^2(S)$ is generated by a ``universal uniformizer'' $\varpi \in \csn(\Gamma)$, where $\Gamma(S)$ is identified with (the additive group of) $\Q$. A subset in the $\Gamma$-sort is a \emph{rational polyhedron} if it is defined by a finite system of linear inequalities.

We shall only work with non-archimedean local fields. Let $L$ range over all local fields and denote its residue characteristic, residue degree, and ramification degree by $\epsilon_L$, $\delta_L$, and $\rho_L$, respectively. Set $q_L =  \epsilon_L^{\delta_L}$. We consider $L$ as an $\mdl L_{\RV}^2$-structure with $\varpi = p$ if $L$ is an extension of $\Q_p$ and $\varpi = t$ if $L$ is an extension of $\F_p \dlbr t \drbr$. This is why we do not normalize the value group of $L$ to be $\Z$ if $\cha(L) > 0$. The Haar measure $|\der \lbar X|_{L}$ on $L$ is normalized so that the maximal ideal has measure $1$. For any subset $A$ in the $\VF$-sort of $L$, the volume $\int_A |\der \lbar X|_{L} $ of $A$, if defined, is denoted by $\vol(A)$. For example, $\vol(\OO(L)) = q_L$, $\vol(\vv^{-1}(\gamma)) = (q_L -1)q_L^{- \rho_L \gamma }$ for $\gamma \in \Gamma$, and, for $t \in \RV(L)$ with $\vrv(t) = \gamma$, $\vol(\rv^{-1}(t)) = q_L^{- \rho_L \gamma}$. We adopt the convention $q_L^{- \infty} = 0$. For any $a \in \VF(L)$, the norm of $a$, denoted by $|a|_{L}$, is by definition the number $q_L^{- \rho_L \vv(a)}$.

Let $F$ be either $\Q_p$ or $\F_p \dlbr t \drbr$ and $b_F \in \VF(F)$ such that $\vv(b_F) \in \Z$ is constant as $F$ varies. Let $(L_i)$ be an infinite sequence of local fields such that $(\epsilon_{L_i})$ is unbounded. Any ultraproduct $M$ of $(L_i)$ can be regarded as a substructure of $\gC^2$ and the $\lan{RV}$-reduct of $M$ as a substructure of $\gC$. Let $b \in \VF(M)$ be the element that corresponds to the sequence $(b_F)$. Suppose that $\Gamma(\dcl^2(b)) = \Q$. Let $A \sub \VF^{n}$ be a $b$-definable subset such that $\vv(A(M))$ is bounded from below in $\Gamma(M)^n$. Let $\lbar f = (f_1, \ldots, f_k)$ be a sequence of $b$-definable functions $A \fun \Gamma$ such that $\lbar f(A(M)) \sub \Gamma(M)^k$ is a Presburger subset that is independent of $b$, that is, $\lbar f(A(M))$ is definable in the $\Gamma$-sort of $M$ in the Presburger language without parameters. For each $\lbar \gamma \in \Gamma^k$ let $\pi_{\lbar \gamma}$ be a $\Gamma$-partition of the subset $\{\lbar a \in A : \lbar f(\lbar a) = \lbar \gamma\}$. By Lemma~\ref{gam:same} and Remark~\ref{rem:gam:stab}, the image of $\pi_{\lbar \gamma}$ is $\lbar \gamma$-$\lan{RV}$-definable (independent of $b$). Let $\pi$ be the function on $A$ given by $\lbar a \efun (\lbar f(\lbar a), \pi_{\lbar f(\lbar a)}(\lbar a))$ and $A_{\lbar \gamma} = \pi^{-1}(\lbar \gamma)$. We write $J$ for $\pi(A(M))$. It is easy to see that $\pi$ may be chosen so that $J \sub \pi(A)(M)$, where the inclusion may be proper. Consequently, $J$ is also a Presburger subset that is independent of $b$. Now, for every $A_{\lbar \gamma}$, by Corollary~\ref{all:subsets:rvproduct} (in fact we only need the special case \cite[Corollary~5.6]{Yin:special:trans}), there is an $\RV$-pullback $A^{\sharp}_{\lbar \gamma} \sub \VF^n \times \RV^m$ and a $\csn(\lbar \gamma)$-$\lan{RV}$-definable special bijection $T_{\lbar \gamma}$ between $A_{\lbar \gamma}$ and $A^{\sharp}_{\lbar \gamma}$. By Lemma~\ref{cut:to:hensel:substru}, $T_{\lbar \gamma}(M)$ is a bijection between $A(M)$ and $A^{\sharp}_{\lbar \gamma} (M)$. In particular, this implies that if $\lbar \gamma \in \pi(A)(M) \mi J$ then $A^{\sharp}_{\lbar \gamma} (M) = \0$.

Let $\lbar \kappa = (\kappa_1, \ldots, \kappa_k)$ be a sequence of positive real numbers. Consider the (generalized) Igusa local zeta function
\[
\zeta(A, L, \lbar \kappa) = \int_{A}  q_L^{- \rho_L \sum _i \kappa_i f_{i}(\lbar X)}  |\der \lbar X|_{L}.
\]
The evaluation of $\zeta(A, L, \lbar \kappa)$ may be reduced to computing the volume of each $A_{\lbar \gamma}$, which is $\csn(\lbar \gamma)$-$\lan{RV}$-definable. In \cite{Pa89, Pas:1990}, it is shown that, if $\cha(L_i) = 0$ for all $i$ and, among the three sequences $(\epsilon_{L_i})$, $(\delta_{L_i})$, and $(\rho_{L_i})$, if the first or the second is the only unbounded one (here and below ``unbounded'' means ``going to infinity'') then $\zeta(A, L_i, \lbar \kappa)$ is uniformly rational (see \cite{cluckers:loeser:constructible:motivic:functions, cluckers:loeser:mixchar} for a motivic interpretation of these results). In fact, it is easy to see that the results in \cite{Pa89, Pas:1990} imply that $\zeta(A, L_i, \lbar \kappa)$ is uniformly rational as long as $(\rho_{L_i})$ is bounded. We shall generalize these results such that local fields of positive characteristic are included and all three sequences are unbounded.

By Lemma~\ref{cut:to:hensel:substru}, we may simultaneously work in all but finitely many $L_i$ as far as $\mdl L_{\RV}$-formulas are concerned. In other words, from now on, by an $\lan{RV}$-definable subset $A$ we mean a uniformly $\lan{RV}$-definable subset in $(L_i)$, which in turn means a sequence of subsets $(A_i)_{i \geq k}$ for some sufficiently large $k$ such that every $A_i$ is a subset of the $\mdl L_{\RV}$-reduct of $L_i$ defined by a fixed quantifier-free $\mdl L_{\RV}$-formula $\phi$. To reason about such an $A$, or rather about all $A_i$ for $i \geq k$, we can (and shall tacitly) work with the subset defined by $\phi$ in $\gC$ and then state the results with respect to each $L_i$ uniformly. The reader should note that in this process the number $k$ may increase. With this understanding, for example, we may talk about the size of a definable subset in the $\K$-sort and infinite summation over a definable subset in the $\Gamma$-sort (since it may be identified with a union of rational polyhedrons). For simplicity, we shall drop ``$L_i$'' from some of the notations below. For example, $\VF(L_i)$, $\rho_{L_i}$, etc.\ will simply be written as $\VF$, $\rho$, etc.

\begin{prop}\label{zeta:mult:k}
The integral $\zeta(A, \lbar \kappa)$ may be written as a finite sum of $\mu$ terms of the form $e \sum_{d = 1}^{\rho^{\nu}} \tau_d(q)$, where $e$, $\nu$ are natural numbers,
\[
\tau_d(q) = \sum_{(\lbar n, \lbar m, \lbar l) \in \Delta_d} q^{- \lbar \kappa \cdot \lbar l - \Sigma \lbar n},
\]
and $\Delta_1, \ldots, \Delta_{\rho^{\nu}}$ are pairwise disjoint Presburger subsets. This expression of $\zeta(A, \lbar \kappa)$ is uniform for the sequence $(L_i)$ in the following sense:
\begin{enumerate}
  \item $\mu$ and $\nu$ do not depend on $L_i$,
  \item if $\rho_{L_i} = c \rho_{L_j}$ then every $\Delta_d$ that occurs in $\tau_d(q_{L_j})$ also occurs in $\tau_d(q_{L_i})$ as $c \Delta_d$.
\end{enumerate}
\end{prop}
\begin{proof}
To compute $\zeta(A, \lbar \kappa)$, we may assume that, for all $\lbar \gamma \in J$, the $\csn(\lbar \gamma)$-$\lan{RV}$-definable $\RV$-pullback $A^{\sharp}_{\lbar \gamma}$ is nondegenerate. Let $B_{\lbar \gamma} = \prv(A^{\sharp}_{\lbar \gamma})$. Then we have
\[
\vol(A_{\lbar \gamma}) = \sum_{(\lbar t, \lbar s) \in B_{\lbar \gamma} } q^{- \rho \sum \vrv(\lbar t)}.
\]
By Lemma~\ref{shift:K:csn:def}, without loss of generality, we may assume that $B_{\lbar \gamma}$ is a twistoid. Let $U_{\lbar \gamma}$ be the twistback of $B_{\lbar \gamma}$ and $I_{\lbar \gamma} = \vrv(B_{\lbar \gamma})$. Note that $I_{\lbar \gamma}$ is $\lbar \gamma$-$\lan{RV}$-definable (independent of $b$). Let $\sharp(U_{\lbar \gamma})$ be the size of $U_{\lbar \gamma}$. Then
\[
\vol(A_{\lbar \gamma}) = \sharp(U_{\lbar \gamma}) \sum_{(\lbar \alpha, \lbar \beta) \in I_{\lbar \gamma}} q^{- \rho \Sigma \lbar \alpha}.
\]
We may assume that $U_{\lbar \gamma}$ is the same and $\sharp(U_{\lbar \gamma}) = e$ for all $\lbar \gamma \in J$. Let $\Delta = \bigcup_{\lbar \gamma \in J} I_{\lbar \gamma} \times \{\lbar \gamma\} \sub \Gamma^{\nu}$, which is of course a Presburger subset that is independent of $b$. Then
\[
\zeta(A, \kappa) = \sum_{\lbar \gamma \in J} q^{- \rho \lbar \kappa \cdot \lbar \gamma} \vol(A_{\lbar \gamma}) = e \sum_{(\lbar \alpha, \lbar \beta, \lbar \gamma) \in \Delta} q^{- \rho \lbar \kappa \cdot \lbar \gamma - \rho \Sigma \lbar \alpha}.
\]

For the desired uniformity, it remains to dispose of $\rho$ in this expression. To that end, we identify $\Gamma$ with $\Z$ via the canonical homomorphism $\gamma \efun \rho \gamma$. Then $\Delta$ is identified with a Presburger subset $\Delta_* \sub \Z^{\nu}$. Obviously $\Delta_*$ may be decomposed into $\rho^{\nu}$ Presburger subsets $\Delta_d$ for $1 \leq d \leq \rho^{\nu}$ such that the $i$th coordinate of $\Delta_d$ equals $d_i$ modulo $\rho$ for some $1 \leq d_i \leq \rho$. Note that if $\gcd(d_1, \ldots, d_{\nu}) = c$ then $\Delta_d$ may also be defined using $\rho / c$, $d_1 /c, \ldots, d_{\nu} /c$. Now the uniformity condition is clear since $\Gamma(L_j) = (1/ \rho_{L_j}) \Z$ is a subgroup of $\Gamma(L_i) = (1/ \rho_{L_i}) \Z$.
\end{proof}

\begin{defn}\label{def:uni}
Let $\gL$ be a set of local fields. We say that $\zeta (A, \lbar \kappa)$ is \emph{uniformly rational} for $\gL$ if  there is a finite set of rational functions $R_{m,i}(T_0, T_1, \ldots, T_k)$ for each $m \in \N$ such that, for \emph{every} $L \in \gL$,
\[
\zeta (A, L, \lbar \kappa) = \sum_{m | \rho_L} \sum_i e_{m, i, L} R_{m,i}(q^{\rho_L / m}_L, q_L^{- \rho_L \kappa_1 / m}, \ldots, q_L^{- \rho_L \kappa_k / m}),
\]
where the coefficients $e_{m, i, L} \in \N$ only depend on $q_L$.
\end{defn}

Note that our notion of uniformity is different from but implies that in \cite{Pa89, Pas:1990}. It is clear from the proofs of the main results of \cite{Pa89, Pas:1990} that they may be reformulated using our notion.

We can always make $\zeta (A, \lbar \kappa)$ uniformly rational for $(L_i)$ by deleting finitely many entries from it. If local fields of positive characteristic are included then this cannot be improved at the moment, since rationality of Igusa zeta function, with or without cross-section, is not known in general for local fields of small positive characteristic. On the other hand, if we concentrate on \p-adic fields then, using results in \cite{Denef1984, denef:85, Pa89, Pas:1990},  we can deduce very general results about uniform rationality.

\begin{thm}\label{uni:tho}
If each $L_i$ is a local field of characteristic $0$ then $\zeta (A, \lbar \kappa)$ is uniformly rational for $(L_i)$.
\end{thm}
\begin{proof}
This follows immediately from \cite[Theorem~4.3]{denef:85}, (the proof of) Proposition~\ref{zeta:mult:k}, and \cite[Theorem~4.4.1]{cluckers:loeser:constructible:motivic:functions}. There are (potentially) infinitely many rational functions because $(\rho_{L_i})$ may be unbounded, which gives rise to infinitely many Presburger subsets.
\end{proof}

For each $n \in \N$ let $\gL_{n}$ be the set of all local fields $L$ of characteristic $0$ such that $\rho_L \leq n$.

\begin{thm}\label{igusa:uni}
For all local fields of characteristic $0$, $\zeta (A, \lbar \kappa)$ is uniformly rational.
\end{thm}
\begin{proof}
By \cite[Theorem~7.1]{Pa89} and \cite[Theorem~5.1]{Pas:1990}, we only need to work with a finite subset $\gL'_n \sub \gL_n$ for each $n$. Then the assertion follows from Theorem~\ref{uni:tho}.
\end{proof}

\begin{rem}
There are variations of these theorems. For example, they hold if we work in $\gC$ instead of $\gC^2$ (see \cite[Theorem~1.3]{hrushovski:kazhdan:integration:vf}). There are also analogues if we work in $\gC^3$. However, in that case local fields of characteristic $0$ have to be excluded (see Remark~\ref{rem:loc}) and hence there always are finitely many exceptions about which we can say nothing at the moment.

Suppose that $(\epsilon_{L_i})$, $(\delta_{L_i})$ are bounded and $(\rho_{L_i})$ is unbounded in $(L_i)$. What can one say about these infinitely many rational functions? Is it true that $\zeta (A, \lbar \kappa)$ is uniform for $(L_i)$ with respect to finitely many rational functions? These are difficult and deep questions, and are related to the asymptotic behavior of the poles of $\zeta(A, \lbar \kappa)$ and its rationality in a local field of small positive characteristic. We hope that future development of the present theory will be able to offer some clues.
\end{rem}

\end{document}